\documentclass[a4paper,10pt]{amsart}  % JP prefers 10 pt!
\usepackage[utf8]{inputenc} % accents under linux
\usepackage{amssymb}
\usepackage{colordvi}
\usepackage{graphicx,tikz}
\usetikzlibrary{patterns,calc,intersections,through,backgrounds}
\usepackage{vmargin}
\usepackage{todonotes}
\usepackage{hyperref}
\usepackage{overpic,graphics}
\usepackage{algpseudocode}
\usepackage{enumerate}
\usepackage{multirow}
\usepackage{mathtools}
\usepackage{cleveref}

\setmargrb{1in}{1in}{1in}{1in} % --- sets all four margins.  Cool.
\newtheorem{theorem}{Theorem}[section]
\newtheorem{proposition}[theorem]{Proposition}
\newtheorem{lemma}[theorem]{Lemma}
\newtheorem{corollary}[theorem]{Corollary}
\newtheorem{conjecture}[theorem]{Conjecture}

\theoremstyle{definition}
\newtheorem{definition}[theorem]{Definition}
\newtheorem{question}[theorem]{Question}

\theoremstyle{definition}
\newtheorem{example}[theorem]{Example}
\newtheorem{remark}[theorem]{Remark}

\makeatletter
\@addtoreset{proofpart}{theorem}
\@addtoreset{proofcase}{theorem}
\makeatother

\newcommand{\R}{\mathbb{R}}

\newcommand{\reverse}[1]{\overleftarrow{#1}} % reading the black and white sequence backwards

\DeclareMathOperator{\conv}{conv} % convex hull
\DeclareMathOperator{\asc}{asc} % convex hull
\DeclareMathOperator{\tasc}{tasc} % number of tamari ascents
\DeclareMathOperator{\inv}{inv} % inversion set
\DeclareMathOperator{\card}{\#} % cardinality of inversions
\DeclareMathOperator{\Var}{Var} % The variation set
\DeclareMathOperator{\EVar}{EVar} % the essential variations
\DeclareMathOperator{\sdim}{dim} % dimension

\newcommand{\wole}{\preccurlyeq} % weak order on decreasing trees

\newcommand{\tc}[1]{#1^{\mathsf{tc}}} % transitive closure of a relation
\newcommand{\meet}{\wedge} % meet in lattice
\newcommand{\join}{\vee} % join in lattice
\newcommand{\maxs}{\Sigma} % maximal inversion set

\newcommand{\Perm}[1]{\operatorname{Perm}(#1)} % permutahedron
\newcommand{\Asso}[1]{\operatorname{Asso}(#1)} % associahedron

\newcommand{\vT}{{\mathcal T}} % v-Tamari tree
\newcommand{\vF}{{\mathcal F}} % v-Tamari face
\newcommand{\vA}{{\mathcal A}} % a set of v-ascents

\newcommand{\streestovtrees}{{\varphi}} % bijection from s-trees to v-trees

\definecolor{darkblue}{rgb}{0,0,0.7} % darkblue color
\newcommand{\darkblue}{\color{darkblue}} % darkblue command
\newcommand{\defn}[1]{\emph{\darkblue #1}} % emphasis of a definition

\usepackage{todonotes}

\graphicspath{{figures/}} 

\makeatletter
\def\part{\@startsection{part}{1}%
\z@{.7\linespacing\@plus\linespacing}{.8\linespacing}%
{\LARGE\sffamily\centering}}
\@addtoreset{section}{part}
\makeatother

\makeatletter
\def\l@section{\@tocline{1}{5pt}{0pc}{}{}}
\makeatother
\let\oldtocpart=\tocpart
\renewcommand{\tocpart}[2]{\bf\large\oldtocpart{#1}{#2}}
\let\oldtocsection=\tocsection
\renewcommand{\tocsection}[2]{\bf\oldtocsection{#1}{#2}}

\definecolor{blue}{rgb}{0, 0.445, 0.695}
\definecolor{bluishgreen}{rgb}{0, 0.626, 0.456}
\definecolor{purple}{rgb}{0.783, 0.464, 0.640}
\definecolor{skyblue}{rgb}{0.359, 0.752, 0.973}
\definecolor{orange}{rgb}{0.999, 0.706, 0.0}
\definecolor{yellow}{rgb}{0.937, 0.890, 0.258}

\definecolor{olive}{RGB}{116,141,19}
\definecolor{green}{RGB}{108,208,48}
\definecolor{teal}{RGB}{47,77,62}
\definecolor{turquoise}{RGB}{86,235,211}
\definecolor{lightblue}{RGB}{150,178,153}
\definecolor{blue2}{RGB}{25,50,191}
\definecolor{indigo}{RGB}{142,128,251}
\definecolor{indigo2}{RGB}{114,32,246}
\definecolor{lightpurple}{RGB}{243,197,250}
\definecolor{purple2}{RGB}{105,66,131}
\definecolor{magenta}{RGB}{206,43,188}
\definecolor{brown}{RGB}{110,57,13} 
\definecolor{darkblue}{rgb}{0.0, 0.0, 0.7}

\title[The $s$-week order and $s$-permutahedra II]{The $s$-weak order and $s$-permutahedra II: \\
The combinatorial complex of pure intervals}
\date{\today}

\author[C.~Ceballos]{Cesar Ceballos$^{\diamond}$}
\address[C.~Ceballos]{Institute of Geometry, TU Graz, Graz, Austria}
\email{\href{mailto:cesar.ceballos@tugraz.at}{\texttt{cesar.ceballos@tugraz.at}}}
\urladdr{http://www.geometrie.tugraz.at/ceballos/}

\author[V.~Pons]{Viviane Pons$^{\star}$}
\address[V.~Pons]{CNRS, IRL CRM Montr\'eal, Canada -- Universit\'e Paris-Saclay, CNRS,
Laboratoire Interdisciplinaire des Sciences du Num\'erique, Orsay, France.}
\email{\href{mailto:viviane.pons@lisn.upsaclay.fr}{\texttt{viviane.pons@lisn.upsaclay.fr}}}
\urladdr{\url{https://www.lri.fr/~pons/}}

\thanks{This project was supported by the 
ANR-FWF International Cooperation Project PAGCAP, funded by
the ANR Project ANR-21-CE48-0020 and
the FWF Project I 5788. 
Cesar Ceballos was also supported by the Austrian Science Fund FWF, Project P 33278.
}


\begin{document}

\maketitle

\begin{abstract}
This paper introduces the geometric foundations for the study of the $s$-permutahedron and the $s$-associahedron, two objects that encode the underlying geometric structure of the $s$-weak order and the $s$-Tamari lattice. We introduce the $s$-permutahedron as the complex of pure intervals of the $s$-weak order, present enumerative results about its number of faces, and prove that it is a combinatorial complex. This leads, in particular, to an explicit combinatorial description of the intersection of two faces. 
We also introduce the $s$-associahedron as the complex of pure $s$-Tamari intervals of the $s$-Tamari lattice, show some enumerative results, and prove that it is isomorphic to a well chosen $\nu$-associahedron. Finally, we present three polytopality conjectures, evidence supporting them, and some hints about potential generalizations to other finite Coxeter groups.   
\end{abstract}

\tableofcontents

\section*{Introduction}

This paper is the second contribution (after~\cite{CP22}) to a series of articles related to the $s$-weak order and the $s$-permutahedron. We originally introduced these objects in an extended abstract~\cite{FPSAC2019}, which is now developed into two long versions: the prequel~\cite{CP22} and this present paper. 

\begin{figure}[h!]
\includegraphics[width=0.6\textwidth]{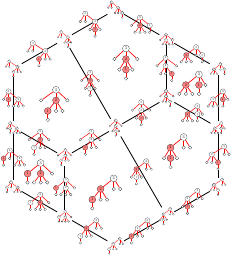}
\caption{The $s$-weak order as the edge graph of the $s$-permutahedron for $s=(0,2,2)$. }
\label{fig:s022-complex}
\end{figure}

\begin{figure}[h!]
\includegraphics[width=0.4\textwidth]{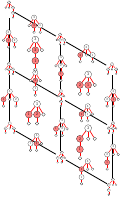} 

\caption{The $s$-weak order as the edge graph of the $s$-permutahedron for $s=(0,0,2)$. }
\label{fig:s002-complex}
\end{figure}

The purpose of the first paper~\cite{CP22} was to develop the combinatorial foundations for the study of the $s$-weak order and the $s$-Tamari lattice, two partially ordered structures defined on certain families of decreasing trees, see examples in Figures~\ref{fig:s022-complex}-~\ref{fig:s002-complex} and~\ref{fig_associahedron_s022}. 
Our motivation for introducing those objects was to fill the gap between recent developments on generalizations of the classical Tamari lattice on one side, and its connections to the classical weak order of permutations on the other.  

\begin{figure}[htb]
\includegraphics[width=0.8\textwidth]{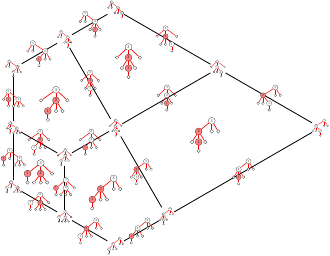}
\caption{The $s$-Tamari lattice as the edge graph of the  $s$-associahedron for $s=(0,2,2)$}
\label{fig_associahedron_s022}
\end{figure}

On one side, we have the classical Tamari lattice (see~\cite{Tam51, hoissen12associahedra}), a partial order on Catalan objects (such as binary trees) whose Hasse diagram is the edge graph of a polytope called the associahedron. It is a central object in algebraic combinatorics, related to many interesting structures such as AVL trees in efficient sorting algorithms~\cite{AVL62}, Hopf algebras~\cite{LodayRonco02_permutations,LodayRonco98_binarytrees,AguiarSottile_StructureLodayRonco,HNT05}, Cambrian lattices~\cite{Reading-cambrianLattices}, Permutrees~\cite{PP18}, cluster algebras~\cite{FominZelevinsky-ClusterAlgebrasI, FominZelevinsky-ClusterAlgebrasII}, representation theory~\cite{haiman_conjectures_1994,BPR12}, and more.
Two important generalizations of the Tamari lattice that are relevant for our work are the $m$-Tamari lattice of Bergeron and Pr\'eville-Ratelle~\cite{BPR12} and the more general $\nu$-Tamari lattice of Pr\'eville-Ratelle and Viennot~\cite{PrevilleRatelleViennot}, whose study is motivated by conjectural connections with the theory of (multivariate higher) diagonal harmonics in representation theory, see e.g.~\cite{BPR12,bousquet_representation_2013,bousquet_number_2011,HopfDreams_2022}.

On the other side, we have the classical weak order on permutations, whose Hasse diagram is the edge graph of a polytope called the permutahedron. It is known to be a lattice~\cite{GR70} obtained from an orientation of the Cayley graph of the symmetric group. Both the lattice construction and the associated polytope also exist for other finite Coxeter groups~\cite{Bjo84,Hoh12}.
The weak order has well known connections with the classical Tamari lattice, which can be obtained from it both as a sublatttice~\cite{bjorner_wachs_shellableII_1997} and as a lattice quotient~\cite{Reading-cambrianLattices}. These properties have interesting geometric and algebraic counterparts. Geometrically, the associahedron can be obtained by removing certain facets of the permutahedron~\cite{HL07,HLT11}, and this translates into algebraic properties between certain related Hopf algebras~\cite{LodayRonco02_permutations,LodayRonco98_binarytrees}.

The recent extensions of the Tamari lattice and the lack of (type $A$)\footnote{There are further generalizations in the context of Coxeter groups. In this paper, we mainly focus on generalizations which are related to Coxeter groups of type $A$. Some discussions about other types are mentioned in Section~\ref{sec_polytopal_conjectures}.}  generalizations of the weak order lead to the question: 
what is the equivalent of the weak order for the $\nu$-Tamari lattice?
The $s$-weak order introduced in~\cite{CP22} gives an answer to this question.
Given a sequence of non-negative integers $s$, the \defn{$s$-weak order} is a lattice structure on a family of combinatorial objects called \emph{$s$-decreasing trees}. We also introduced the \defn{$s$-Tamari lattice} as a sublattice of the $s$-weak order and showed that it is isomorphic to a well chosen $\nu$-Tamari lattice. 
%Examples the Hasse diagrams of the $s$-weak order and the $s$-Tamari lattice for $s=(0,2,2)$ are the edge graphs of the subdivisions in Figure~\ref{fig:s022-complex} and Figure~\ref{fig_associahedron_s022}, respectively.

\begin{center}
\begin{tikzpicture}
\node (Tam) at (0,0) {Tamari lattice};
\node (Weak) at (0,1.5) {Weak order};
\draw[->] (Tam) -- (Weak);
\node[draw] (sTam) at (4,0) {
\begin{tabular}{c}
 $s$-Tamari lattice \\  
 $\nu$-Tamari lattice     
\end{tabular}
};
\node[draw] (sWeak) at (4,1.5) {$s$-Weak order};
\draw[->] (sTam) -- (sWeak);
\draw[->] (Tam) -- (sTam);
\draw[->] (Weak) -- (sWeak);
\node at (0,1.9) {}; % to create a small horizontal space between text and figure
\end{tikzpicture}
\end{center}

The purpose of this second paper is to {investigate the underlying geometric structure of the $s$-weak order and the $s$-Tamari lattice}. As we can see from Figures~\ref{fig:s022-complex} and~\ref{fig_associahedron_s022}, the $s$-weak order and the $s$-Tamari lattice can be realized as the edge graph of a much richer geometric structure, whose faces can be explicitly described in terms of a special family of intervals which we call pure intervals and pure $s$-Tamari intervals. 
This leads to two main definitions: (1) the \defn{$s$-permutahedron} is the complex of pure intervals of the $s$-weak order and (2) the \defn{$s$-associahedron} is the complex of pure $s$-Tamari intervals of the $s$-Tamari lattice.  
In this second paper, we introduce these concepts and develop the geometric foundations for the study of~$s$-permutahedra and $s$-associahedra.

\begin{center}
\begin{tikzpicture}
\node (Asso) at (0,0) {Associahedron};
\node (Perm) at (0,1.5) {Permutahedron};
\draw[->] (Asso) -- (Perm);
\node[draw] (sAsso) at (4,0) {
\begin{tabular}{c}
 $s$-Associahedron\\  
 $\nu$-Associahedron     
\end{tabular}
};
\node[draw] (sPerm) at (4,1.5) {$s$-Permutahedron};
\draw[->] (sAsso) -- (sPerm);
\draw[->] (Asso) -- (sAsso);
\draw[->] (Perm) -- (sPerm);
\node at (0,1.9) {}; % to create a small horizontal space between text and figure
\end{tikzpicture}
\end{center}

The paper is subdivided into three parts. Part 1 is dedicated to the introduction and study of the $s$-permutahedron (Definition~\ref{def:s-perm}), 
whose faces are given by pure intervals of the $s$-weak order (Definition~\ref{def:pure-intervals}). 
%A pure interval (Definition~\ref{def:pure-intervals}) is characterized by an $s$-decreasing tree (at the bottom of the interval) and a subset of its covers in the $s$-weak order. The top of the interval is the join of the selected covers. 
We present enumerative results about the number of faces by giving an explicit formula for the $f$-polynomial (Proposition~\ref{prop:f-poly}), and describe a more efficient recursive formula to compute it (Proposition~\ref{prop:f-poly-recur}).
The main result of this part is to prove that the $s$-permutahedron is a \defn{combinatorial complex} (Theorem~\ref{sperm_combinatorial_complex}) which, in particular, requires that the intersection of two faces is also a face of the complex (Theorem~\ref{thm:intersection} and Corollary~\ref{cor:intersection}). Our proof is purely combinatorial and is based on an explicit characterization of pure intervals (Theorem~\ref{thm:pure-interval-char}) using the notions of \defn{variations} and \defn{essential variations}.

In Part 2, we introduce and study the $s$-associahedron (Definition~\ref{def_sAssociahedron}), 
whose faces are given by pure $s$-Tamari intervals of the $s$-Tamari lattice (Definition~\ref{def_pure_sTamari_intervals}). 
We present an explicit formula for its $f$-polynomial (Proposition~\ref{prop:f-poly_stamari}), which provides natural definitions of the \defn{$s$-Catalan} and the \defn{$s$-Narayana} numbers. 
The main result of this part is to show that the $s$-associahedron 
coincides with the $\nu$-associahedron introduced by Ceballos, Padrol, and Sarmiento in~\cite{CeballosPadrolSarmiento-geometryNuTamari} (Theorem~\ref{thm:nu-ass-s-ass}). 
The $\nu$-associahedron is a polytopal complex of bounded faces induced by certain arrangement of tropical hyperplanes, and its edge graph gives a geometric realization of the $\nu$-Tamari lattice~\cite{CeballosPadrolSarmiento-geometryNuTamari}.  
This proves, in particular, that the $s$-associahedron is also a polytopal complex. 

In Part 3, we present three conjectures about geometric realizations of the $s$-permutahedron and the $s$-associahedron (Section~\ref{sec_polytopal_conjectures}), as well as some hints about potential generalizations of our work in the context of finite Coxeter groups (Section~\ref{sec_Coxeter}).  
The first conjecture (Conjecture~\ref{conj:pure-polytopal}) states that pure intervals of the $s$-weak order can be realized as convex polytopes. 
The second conjecture (Conjecture~\ref{coj:spermutahedra}) asserts that the $s$-permutahedron can be geometrically realized as a polytopal subdivision of a zonotope, and the third (Conjecture~\ref{con:sassociahedra}) that there is such a realization such that the $s$-associahedron can be obtained from it by removing some facet defining inequalities.    
All these conjectures are strongly supported by computer evidence. 
Several examples are illustrated in Figure~\ref{fig_geometricrealizations_intro}.
In particular, we show that they hold in full generality in dimensions 2 and 3 (Section~\ref{sec_geometric_realizations_dim_2_3}).

\begin{figure}[htbp]
\begin{center}
\begin{tabular}{ccc}
\includegraphics[height=4.3cm]{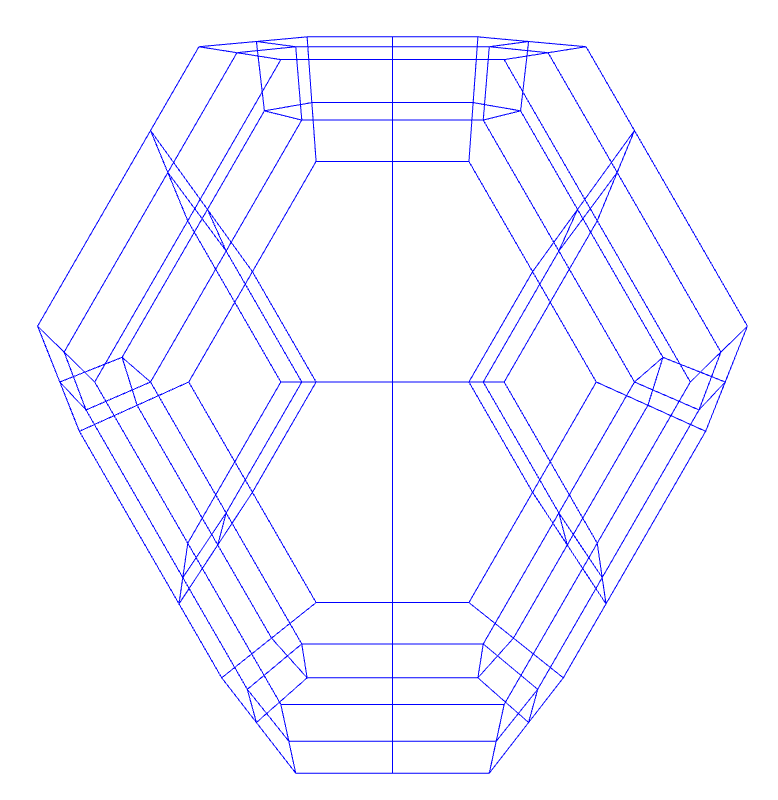}&
\includegraphics[height=4.3cm]{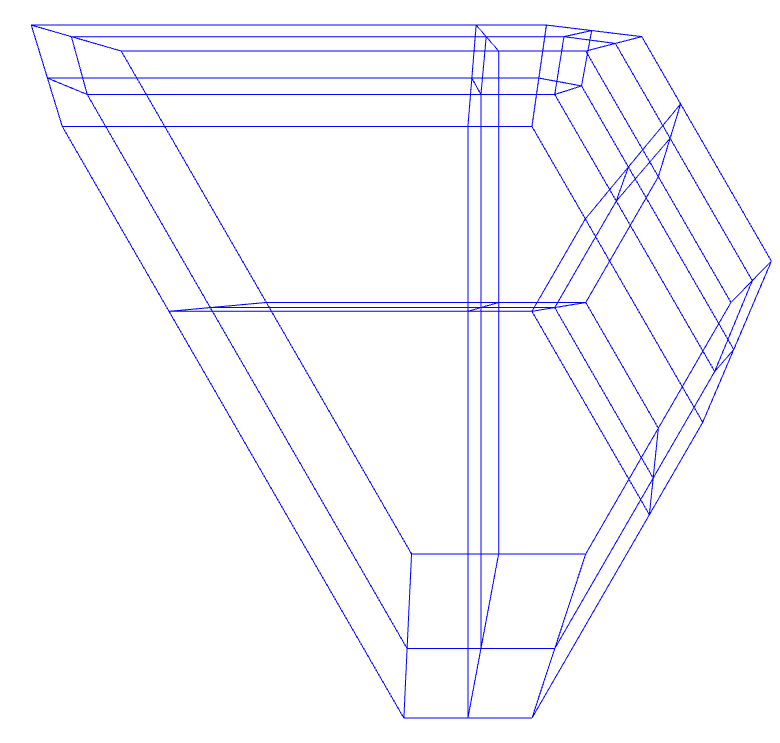}&
\includegraphics[height=4.3cm]{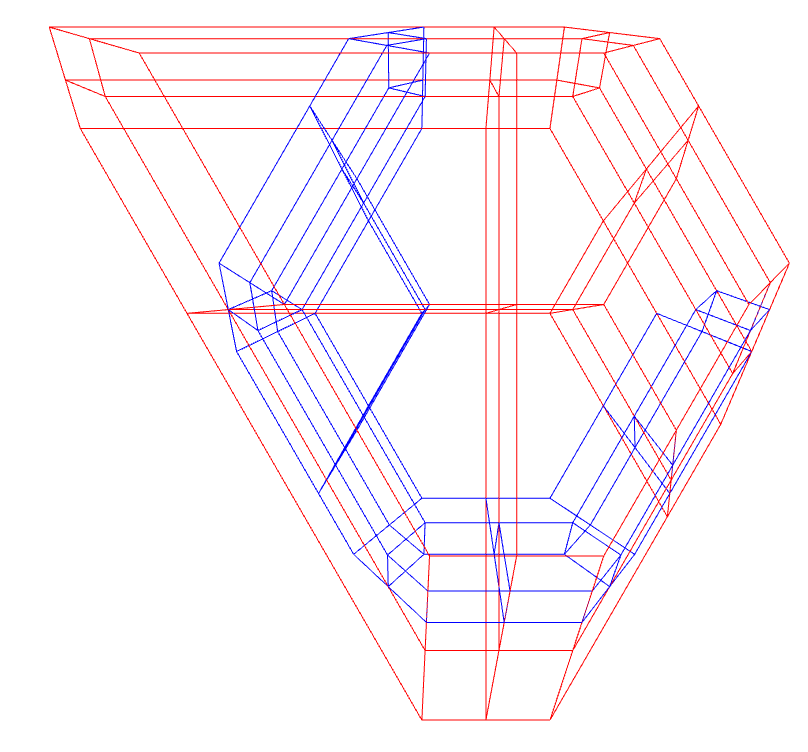} \\
\includegraphics[height=4.3cm]{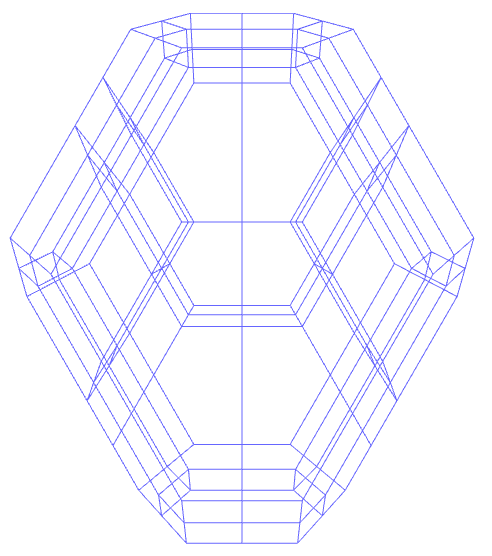}&
\includegraphics[height=4.3cm]{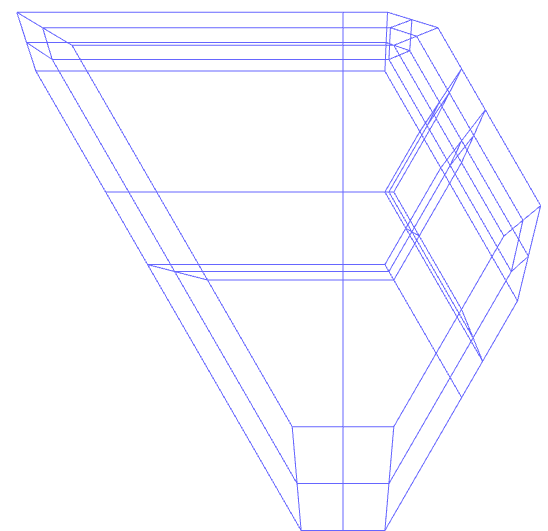}&
\includegraphics[height=4.3cm]{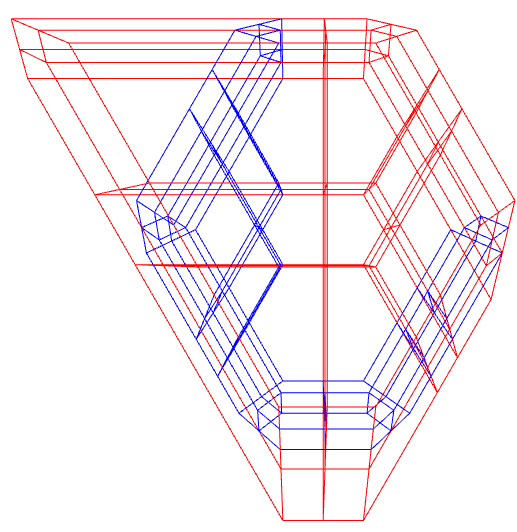} \\
\includegraphics[height=4.3cm]{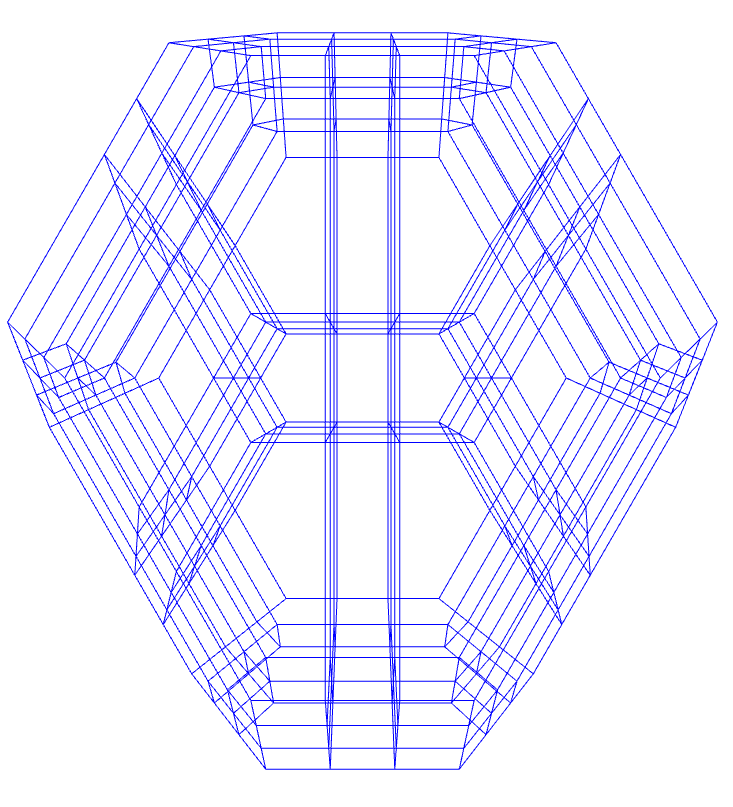}&
\includegraphics[height=4.3cm]{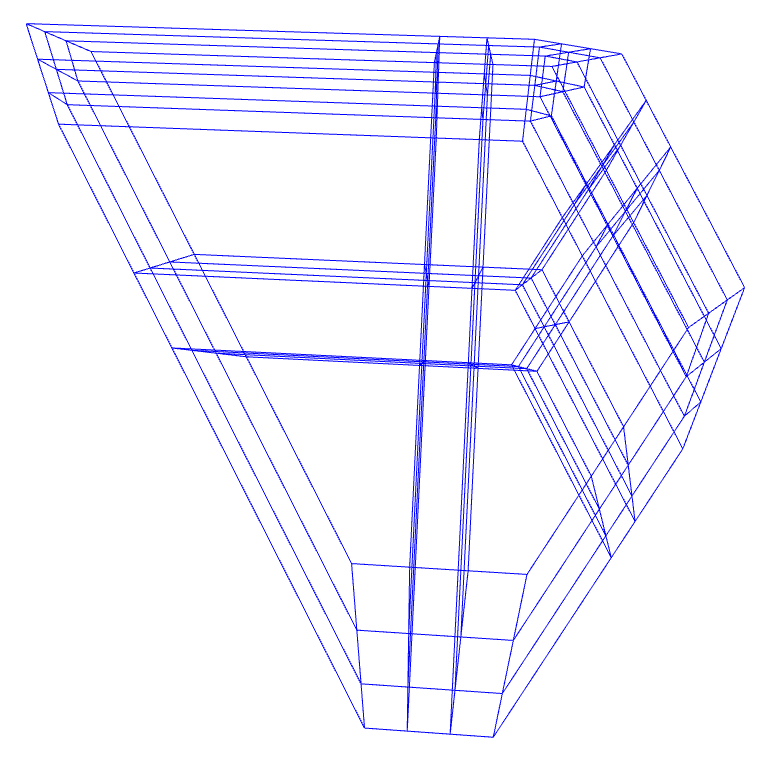}&
\includegraphics[height=4.3cm]{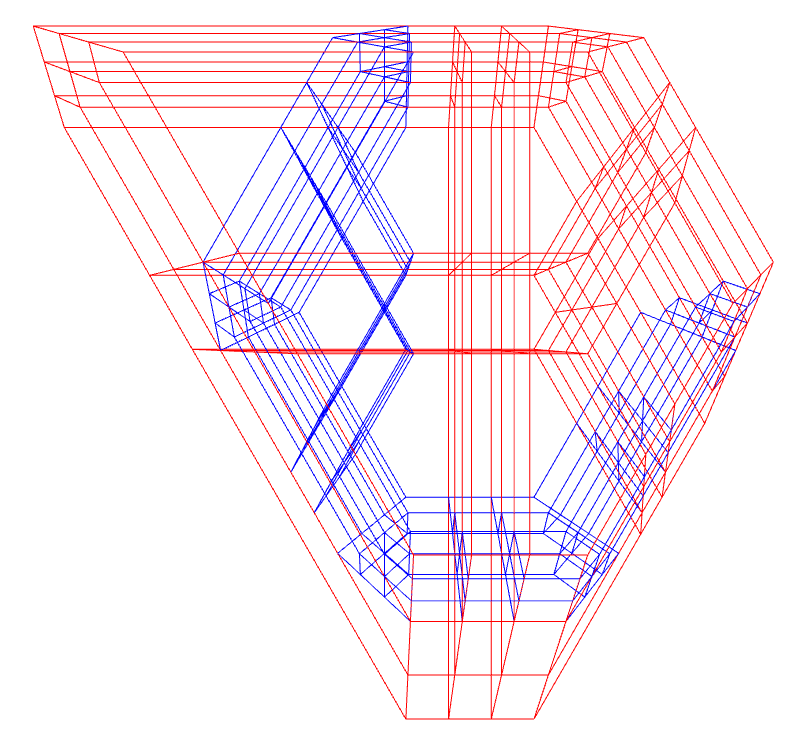} \\
\end{tabular}
\caption{The $s$-permutahedron and the $s$-associahedron obtained from it by removing certain facets,  for $s=(0,2,2,2)$, $s =(0,2,3,2)$, and $s=(0,3,3,3)$.}
\label{fig_geometricrealizations_intro}
\end{center}
\end{figure}

We remark that these conjectures were already present in our initial extended abstract~\cite{FPSAC2019}, and have been partially solved since. Indeed, in~\cite{GMPT23}, for a composition $s$ without any $0$, the authors provide a realization of the $s$-permutahedron as the dual of the complex of interior faces of a triangulation of a flow polytope. 
From this, using the Cayley trick and techniques from tropical geometry, they are able to construct the desired geometric realization of the $s$-permutahedron. The results in~\cite{GMPT23} solve Conjectures~\ref{conj:pure-polytopal} and~\ref{coj:spermutahedra} for the special case where $s$ contains no zeros. 
The case where $s$ contains zeros is still open, and Conjecture~\ref{con:sassociahedra} remains open in general. We present a further discussion on these connections in Section~\ref{sec:flows}.
We remark that pure intervals also appear in the work of Lacina~\cite{SWeakSB}, where he shows that they are the only intervals of the $s$-weak order which are homotopy equivalent to a sphere. 

These recent advances motive the further study of the $s$-weak order and the $s$-permutahedron, as well as a deeper understanding on the nice properties of the family of pure intervals. 
In comparison to~\cite{GMPT23}, this paper develops the foundations for the study of $s$-permutahedra and $s$-associahedra and explores the combinatorial/geometric properties of the $s$-permutahedra for all~$s$ (with or without zeros). 

Beside the paper, we provide the SageMath~\cite{SageMath2022} code used to reach the results as well as a demo SageMath worksheet with computations of all examples in this paper~\cite{SageDemoII}. 

\vfill

\part{The $s$-permutahedron}\label{part_one}
\section{The $s$-weak order (background)}

\subsection{Lattice definition through $s$-decreasing trees and tree-inversions}

In this Section, we recall the main definitions and results of~\cite{CP22} needed for Part~\ref{part_one} of this paper. We call a sequence $s = (s(1), s(2), \dots, s(n))$ of non-negative integers a \defn{weak composition} of length $\ell(s) = n$. If for all $i$, $s(i) > 0$, this is simply a composition. An \defn{$s$-decreasing tree} is a \emph{rooted planar} tree with $n$ internal nodes labeled bijectively by $1,\dots,n$, such that node $i$ has $s(i) + 1$ children and such that labels are decreasing from root to leaves. An example is given on Figure~\ref{fig:ex-dtree}.

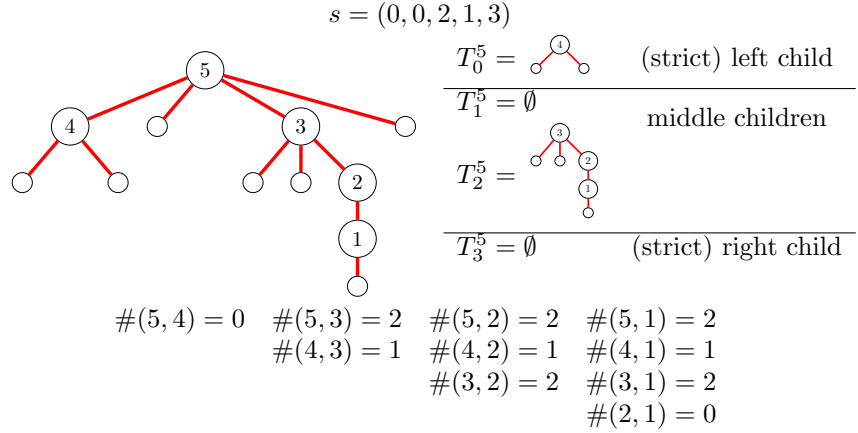
\begin{figure}[ht]
$s = (0,0,2,1,3)$

\begin{tabular}{clc}
\multirow{4}{*}{\scalebox{.8}{{ \newcommand{\nodea}{\node[draw,circle] (a) {$5$}
;}\newcommand{\nodeb}{\node[draw,circle] (b) {$4$}
;}\newcommand{\nodec}{\node[draw,circle] (c) {$ $}
;}\newcommand{\noded}{\node[draw,circle] (d) {$ $}
;}\newcommand{\nodee}{\node[draw,circle] (e) {$ $}
;}\newcommand{\nodef}{\node[draw,circle] (f) {$3$}
;}\newcommand{\nodeg}{\node[draw,circle] (g) {$ $}
;}\newcommand{\nodeh}{\node[draw,circle] (h) {$ $}
;}\newcommand{\nodei}{\node[draw,circle] (i) {$2$}
;}\newcommand{\nodej}{\node[draw,circle] (j) {$1$}
;}\newcommand{\nodeba}{\node[draw,circle] (ba) {$ $}
;}\newcommand{\nodebb}{\node[draw,circle] (bb) {$ $}
;}\begin{tikzpicture}[auto]
\matrix[column sep=.3cm, row sep=.3cm,ampersand replacement=\&]{
         \&         \&         \&         \& \nodea  \&         \&         \&         \&         \\ 
         \& \nodeb  \&         \& \nodee  \&         \&         \& \nodef  \&         \& \nodebb \\ 
 \nodec  \&         \& \noded  \&         \&         \& \nodeg  \& \nodeh  \& \nodei  \&         \\ 
         \&         \&         \&         \&         \&         \&         \& \nodej  \&         \\ 
         \&         \&         \&         \&         \&         \&         \& \nodeba \&         \\
};

\path[ultra thick, red] (b) edge (c) edge (d)
	(j) edge (ba)
	(i) edge (j)
	(f) edge (g) edge (h) edge (i)
	(a) edge (b) edge (e) edge (f) edge (bb);
\end{tikzpicture}}
}}&
$T_0^5 = \begin{aligned}\scalebox{.4}{{ \newcommand{\nodea}{\node[draw,circle] (a) {$4$}
;}\newcommand{\nodeb}{\node[draw,circle] (b) {$ $}
;}\newcommand{\nodec}{\node[draw,circle] (c) {$ $}
;}\begin{tikzpicture}[auto]
\matrix[column sep=.3cm, row sep=.3cm,ampersand replacement=\&]{
         \& \nodea  \&         \\ 
 \nodeb  \&         \& \nodec  \\
};

\path[ultra thick, red] (a) edge (b) edge (c);
\end{tikzpicture}}}\end{aligned} $ & (strict) left child \\ \cline{2-3}
 & $T_1^5 = \emptyset $ & \multirow{2}{*}{middle children} \\
 & $T_2^5 = \begin{aligned}\scalebox{.4}{{ \newcommand{\nodea}{\node[draw,circle] (a) {$3$}
;}\newcommand{\nodeb}{\node[draw,circle] (b) {$ $}
;}\newcommand{\nodec}{\node[draw,circle] (c) {$ $}
;}\newcommand{\noded}{\node[draw,circle] (d) {$2$}
;}\newcommand{\nodee}{\node[draw,circle] (e) {$1$}
;}\newcommand{\nodef}{\node[draw,circle] (f) {$ $}
;}\begin{tikzpicture}[auto]
\matrix[column sep=.3cm, row sep=.3cm,ampersand replacement=\&]{
         \& \nodea  \&         \\ 
 \nodeb  \& \nodec  \& \noded  \\ 
         \&         \& \nodee  \\ 
         \&         \& \nodef  \\
};

\path[ultra thick, red] (e) edge (f)
	(d) edge (e)
	(a) edge (b) edge (c) edge (d);
\end{tikzpicture}}
}\end{aligned}$ & \\ \cline{2-3}
 & $T_3^5 = \emptyset $ & (strict) right child \\
\end{tabular}

\vspace{.5cm}
\begin{tabular}{llll}
$\card(5,4) = 0$ & $\card(5,3) = 2$ & $\card(5,2) = 2$ & $\card(5,1) = 2$ \\
                 & $\card(4,3) = 1$ & $\card(4,2) = 1$ & $\card(4,1) = 1$ \\
                 &                  & $\card(3,2) = 2$ & $\card(3,1) = 2$ \\
                 &                  &                  & $\card(2,1) = 0$
\end{tabular}
\caption{an $s$-decreasing tree and its tree-inversions (Figure 4 of~\cite{CP22})}
\label{fig:ex-dtree}
\end{figure}

As each node is given a unique label, we usually write \emph{the node $y$} instead of \emph{the node labeled by $y$} for convenience. For an $s$-decreasing tree $T$ and a node $y$ such that $s(y) = m$, we write $T_0^y, \dots, T_{m}^y$ the $m+1$ children of $y$. Then $T_0^y$ is the \defn{left} child of $y$, while $T_1^y, \dots, T_{m-1}^y$ are the \defn{middle children} of $y$ and $T_m^y$ is the \defn{right} child of $y$. Besides, if $m > 0$, we say that $T_0^y$ (resp. $T_{m}^y$) is a \defn{strict left} (resp. strict right) child of $y$. Note that the word \emph{child} refers here to the whole subtree not just the child node. 

An $s$-decreasing tree $T$ is characterized by its multi-set of \defn{tree-inversions} $\inv(T)$. Let $c < a$ be two nodes a of an $s$-decreasing tree $T$, the \defn{cardinality} of the tree-inversion $(c,a)$, written $\card_T(c,a)$, is given by the relative position of $a$ towards $c$. If $a$ is left of $c$ or belongs to $T_0^c$, then $\card_T(c,a) = 0$. If $a$ belongs to a middle child $T_i^c$ (with $0 < i < m$), then $\card_T(c,a) = i$. Finally, if $a$ belongs to $T_m^c$ or is right of $c$, then $\card_T(c,a) = m$. We give many examples on Figure~\ref{fig:ex-dtree}.

As we explain in~\cite{CP22}, $s$-decreasing trees are a generalization of permutations. More precisely, if~$s$ does not contain any zeros, $s$-decreasing trees correspond to certain multi-permutations where the letter~$i$ is repeated $s(i)$ times: they are in bijection with \defn{stirling permutations}, \emph{i.e.},  $121$-avoiding multi-permutations. In particular, if $s = (1,1,\dots)$, the $s$-decreasing trees are the decreasing binary trees, known to be in bijection with classical permutations. Thus, the tree-inversions are a generalization of the classical inversion sets of permutations. It can be understood as a \emph{multi-set} where each inversion $(c,a)$ appears $\card(c,a)$ times. Many classical notions and operations on sets can be generalized to multi-sets.

\begin{definition}[Definitions~1.4 and~1.5 of~\cite{CP22}]
\label{def:multi-sets}

Let $I$ be a multi-set of inversions $(c,a)$ with $a < c$. We write $\card_I(c,a)$ the number of occurrences of $(c,a)$ in $I$. If there is no occurrence of $(c,a)$ in $I$ we write $(c,a) \notin I$ or equivalently $\card_I(c,a) = 0$.

Given two multi-sets of inversions $I$ and $J$, we say that $I$ is \defn{included} in $J$ (resp. strictly included) and write $I \subseteq J$ (resp. $I \subset J$) if $\card_I(c,a) \leq \card_J(c,a)$ (resp. $\card_I(c,a) < \card_J(c,a)$) for all $(c,a) \in I$.

Given a weak composition $s$ with $\ell(s) = n$, we write $\maxs_s$ the \defn{maximal $s$-inversion set} defined by  $\card_{\maxs_s}(c,a) = s(c)$ for all $1 \leq a < c \leq n$.

We say that a multi-set of inversions $I$ is \defn{transitive} if for all triplets $a < b < c$, either $\card_I(b,a) = 0$ or $\card_I(c,a) \geq \card_I(c,b)$.

Given a weak composition $s$ with $\ell(s) = n$ and $I \subseteq \maxs_s$, then $I$ is \defn{planar} if for all triplets $1 \leq a < b < c \leq n$, either $\card_I(b,a) = s(b)$ or $\card_I(c,b) \geq \card_I(b,a)$.
\end{definition}

We prove the following.

\begin{proposition}[Proposition 1.6 of~\cite{CP22}]
\label{prop:tree-inversion-set}
For $s$ a given weak composition, a multi-set of inversions is the multi-set of tree-inversions of an $s$-decreasing tree $T$ if and only if it is planar, transitive, and included in $\maxs_s$. We then call it an \defn{$s$-tree-inversion set}.
\end{proposition}

Even though they might seem technical, the transitivity and planarity conditions can be naturally explained by picture, as in Figure~\ref{fig:transitivity_planarity}. These conditions regard the possible positions for the triplet $c > b > a$ inside the tree. By picture, it is then clear that the transitivity and planarity conditions are satisfied by any multi-set of inversions corresponding to an $s$-decreasing tree. The proof that they are sufficient conditions to obtain an $s$-decreasing tree is done using an explicit construction algorithm~\cite[Algorithm 1.8]{CP22}.

\begin{figure}[h]
\begin{center}
\includegraphics[width = 0.5\textwidth]{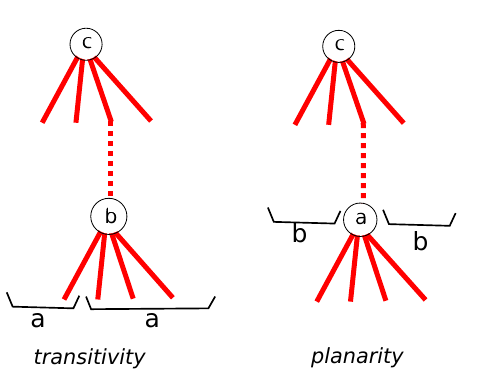}
\caption{Illustration of the transitivity and planarity conditions on $s$-tree inversion sets. (Figure 5 of~\cite{CP22})}
\label{fig:transitivity_planarity}
\end{center}
\end{figure}

Tree-inversions and tree-inversion sets are fundamental notions to work on $s$-decreasing trees, and are specially useful in proofs. The methodology we often use is to first \emph{understand} looking at the trees, and then \emph{prove} using tree-inversions. In particular, we use tree-inversions to define a partial order on $s$-decreasing trees by inclusion of their tree-inversion sets. This is the \defn{$s$-weak order}, generalizing the classical weak order which is also defined by inclusion of inversions. Furthermore, we proved in~\cite{CP22} that it is a lattice using the following two operations on multi-sets.

\begin{definition}[Definition 1.12 of~\cite{CP22}]
\label{def:union-tree-inversions}
The \defn{union} of two multi-sets of inversions $I$ and $J$ is given by 
\begin{equation}
\card_{I \cup J} (c,a) = \max \left( \card_{I}(c,a), \card_{J}(c,a) \right)
\end{equation}
for all $a < c$. 
\end{definition} 

\begin{definition}[Definition 1.4 of~\cite{CP22}]
\label{def:tc-tree-inversions}
 Let $I$ be a multi-set  of inversions. A \defn{transitivity path} between two values $c > a$ is a list of $k \geq 2$ values $c = b_1 > b_2 > \dots > b_k = a$ such that $\card_I(b_i, b_{i+1}) > 0$ for all $1 \leq i < k$. Note that if $\card_I(c,a) > 0$, then $(c,a)$ itself is a transitivity path. We write $\tc{I}$ for the \defn{transitive closure} of $I$, which is defined as follows. For all $a < c$, $\card_{\tc{I}}(c,a) = v$ with $v$ the maximal value of $\card_I(b_1,b_2)$ for $c = b_1 > \dots > b_k = a$ a transitivity path between $c$ and $a$.
\end{definition}

\begin{definition}[{\cite[Defintion~1.9]{CP22}}]
Let $R$ and $T$ be two $s$-decreasing trees of a same weak composition $s$ with $\ell(s) = n$. We say that $R \wole T$ if 
$\inv(R) \subseteq \inv(T)$ using the inclusion of multi-sets. We call the relation $\wole$ the \defn{$s$-weak order}.
\end{definition}

\begin{theorem}[Theorem~1.21 of~\cite{CP22}]
 The $s$-weak order is a lattice for every weak composition $s$. The join of two $s$-decreasing trees $T$ and $R$ is given by
\begin{equation}
\inv(T \join R) = \tc{ \left( \inv(T) \cup \inv(R) \right) }.
\end{equation}
\end{theorem}

We show an example of the computation of the join on Figure~\ref{fig:join}, please see~\cite{CP22} for more detailed examples and computations.

\begin{figure}[h]
\begin{tabular}{ccc}
$T$ & $\begin{aligned}\scalebox{.8}{{ \newcommand{\nodea}{\node[draw,circle] (a) {$3$}
;}\newcommand{\nodeb}{\node[draw,circle] (b) {$1$}
;}\newcommand{\nodec}{\node[draw,circle] (c) {$ $}
;}\newcommand{\noded}{\node[draw,circle] (d) {$ $}
;}\newcommand{\nodee}{\node[draw,circle] (e) {$2$}
;}\newcommand{\nodef}{\node[draw,circle] (f) {$ $}
;}\newcommand{\nodeg}{\node[draw,circle] (g) {$ $}
;}\newcommand{\nodeh}{\node[draw,circle] (h) {$ $}
;}\begin{tikzpicture}[auto]
\matrix[column sep=.3cm, row sep=.3cm,ampersand replacement=\&]{
         \& \nodea  \&         \&         \&         \\ 
 \nodeb  \& \noded  \&         \& \nodee  \&         \\ 
 \nodec  \&         \& \nodef  \& \nodeg  \& \nodeh  \\
};

\path[ultra thick, red] (b) edge (c)
	(e) edge (f) edge (g) edge (h)
	(a) edge (b) edge (d) edge (e);
\end{tikzpicture}}}\end{aligned}$ &
\begin{tabular}{ll}
$\card_T(3,2) = 2$ & $\card_T(3,1) = 0$ \\
                   & $\card_T(2,1) = 0$ 
\end{tabular} \\ \hline
$R$ & $\begin{aligned}\scalebox{.8}{{ \newcommand{\nodea}{\node[draw,circle] (a) {$3$}
;}\newcommand{\nodeb}{\node[draw,circle] (b) {$2$}
;}\newcommand{\nodec}{\node[draw,circle] (c) {$ $}
;}\newcommand{\noded}{\node[draw,circle] (d) {$1$}
;}\newcommand{\nodee}{\node[draw,circle] (e) {$ $}
;}\newcommand{\nodef}{\node[draw,circle] (f) {$ $}
;}\newcommand{\nodeg}{\node[draw,circle] (g) {$ $}
;}\newcommand{\nodeh}{\node[draw,circle] (h) {$ $}
;}\begin{tikzpicture}[auto]
\matrix[column sep=.3cm, row sep=.3cm,ampersand replacement=\&]{
         \&         \&         \& \nodea  \&         \\ 
         \& \nodeb  \&         \& \nodeg  \& \nodeh  \\ 
 \nodec  \& \noded  \& \nodef  \&         \&         \\ 
         \& \nodee  \&         \&         \&         \\
};

\path[ultra thick, red] (d) edge (e)
	(b) edge (c) edge (d) edge (f)
	(a) edge (b) edge (g) edge (h);
\end{tikzpicture}}}\end{aligned}$ &
\begin{tabular}{ll}
$\card_R(3,2) = 0$ & $\card_R(3,1) = 0$ \\
                   & $\card_R(2,1) = 1$ 
\end{tabular} \\ \hline
$T \join R$ & $\begin{aligned}\scalebox{.8}{{ \newcommand{\nodea}{\node[draw,circle] (a) {$3$}
;}\newcommand{\nodeb}{\node[draw,circle] (b) {$ $}
;}\newcommand{\nodec}{\node[draw,circle] (c) {$ $}
;}\newcommand{\noded}{\node[draw,circle] (d) {$2$}
;}\newcommand{\nodee}{\node[draw,circle] (e) {$ $}
;}\newcommand{\nodef}{\node[draw,circle] (f) {$1$}
;}\newcommand{\nodeg}{\node[draw,circle] (g) {$ $}
;}\newcommand{\nodeh}{\node[draw,circle] (h) {$ $}
;}\begin{tikzpicture}[auto]
\matrix[column sep=.3cm, row sep=.3cm,ampersand replacement=\&]{
         \& \nodea  \&         \&         \&         \\ 
 \nodeb  \& \nodec  \&         \& \noded  \&         \\ 
         \&         \& \nodee  \& \nodef  \& \nodeh  \\ 
         \&         \&         \& \nodeg  \&         \\
};

\path[ultra thick, red] (f) edge (g)
	(d) edge (e) edge (f) edge (h)
	(a) edge (b) edge (c) edge (d);
\end{tikzpicture}}}\end{aligned}$ &
\begin{tabular}{ll}
$\card_{T \join R}(3,2) = 2$ & $\card_{T \join R}(3,1) = 2$ \\
                   & $\card_{T \join R}(2,1) = 1$ 
\end{tabular} \\ 
\end{tabular}
\caption{The computation of $T \join R$ in the $s$-weak lattice (Figure 7 of~\cite{CP22})}
\label{fig:join}
\end{figure}
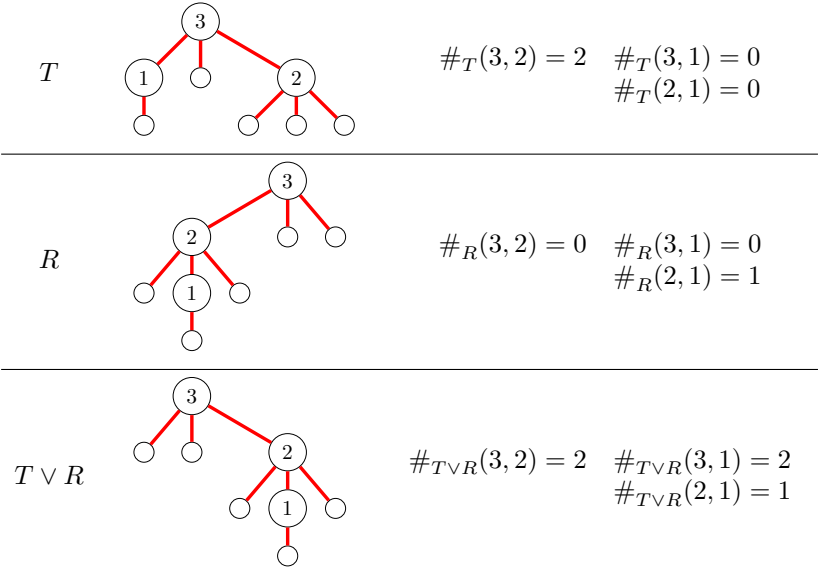

\subsection{Tree-ascents and cover relations}

%In this paper, we study certain intervals of the $s$-weak order using subsets of \defn{cover relations}. We thus recover the notion of \defn{tree ascents} defined in~\cite{CP22}
The {cover relations} of the $s$-weak order can be combinatorially described in terms of a notion of {tree ascents} defined in~\cite{CP22}.

\begin{definition}[Definition 1.24 from~\cite{CP22}]
\label{def:tree-ascent}
Let $T$ be an $s$-decreasing tree of some weak composition~$s$ of length $n$. We say that~$(a,c)$ with $a < c$ is a \defn{tree-ascent} of $T$ if
\begin{enumerate}[(i)]
\item $a$ is a descendant of $c$;
\label{cond:tree-ascent-desc}
\item $a$ does not belong to the right child of $c$;
\label{cond:tree-ascent-non-final}
\item if $a$ is a descendant of $b$ with $c > b > a$, then $a$ belongs to the right descendant of $b$;
\label{cond:tree-ascent-middle}
\item if it exists, the strict right child of $a$ is empty. 
\label{cond:tree-ascent-smaller}
\end{enumerate}
\end{definition}

The above definition can be visually interpreted from the tree: 
basically, $a$ is located at the \emph{right end} of a non right subtree of $c$, with the exception $s(a)=0$, when $a$ is allowed to have more children on its unique leg. 
Figure~\ref{fig:s_tree_ascents} shows a schematic illustration of these two possible occurrences of an $(a,c)$ ascent. 

\begin{figure}[h]
\begin{center}
\includegraphics[width=0.5\textwidth]{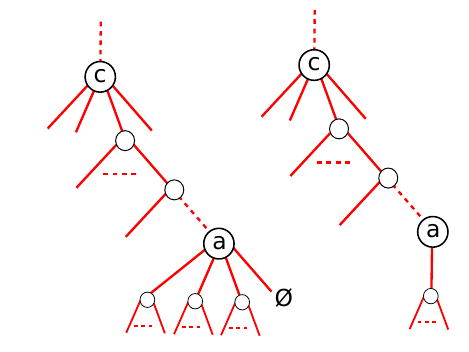}
\caption{Schematic illustration of two $s$-trees containing an ascent $(a,c)$. (Figure~8 from~\cite{CP22})}
\label{fig:s_tree_ascents}
\end{center}
\end{figure}

We prove in~\cite[Theorem 1.32]{CP22} that tree-ascents correspond to cover relations of the $s$-weak order. More precisely, to each tree ascent $(a,c)$ of a tree $T$, one can associate an \defn{$s$-tree rotation} by increasing the cardinality of $\card_T(c,a)$ by one and performing the transitive closure of the resulting multi-set of inversions. In~\cite[Lemma 1.29]{CP22}, we show that it increases all cardinalities $\card(c,a')$ by one where $a'$ is~$a$ or a non-left descendant of $a$. The obtained tree covers $T$ in the $s$-weak order.  

It is sometimes useful to characterize tree-ascents using solely properties on tree-inversions. In~\cite{CP22}, we prove the following result, which basically translates all the conditions of Definition~\ref{def:tree-ascent} in terms of tree-inversions.

\begin{proposition}[Proposition 1.27 of~\cite{CP22}]
\label{prop:tree-ascent-inversions}
A couple $(a,c)$ is a tree-ascent if and only if it satisfies the following statements
\begin{enumerate}[(i)]
\item for all $d$ such that $c < d \leq n$, we have $\card(d,c) = \card(d,a)$;
\label{cond:inv-tree-ascent-desc}
\item $\card(c,a) < s(c)$;
\label{cond:inv-tree-ascent-non-final}
\item for all $b$ such that  $a < b < c$, then $\card(c,b) = \card(c,a)$ implies that $\card(b,a) = s(b)$;
\label{cond:inv-tree-ascent-middle}
\item if $s(a) > 0$, for all $a' < a$, then $\card(a,a') = s(a)$ implies that $\card(c,a') > \card(c,a)$.
\label{cond:inv-tree-ascent-smaller}
\end{enumerate}
\end{proposition}

\section{The $s$-permutahedron}

The key idea of this paper is that even if the $s$-week order is just a poset, it actually posses a much richer underlying geometric structure, encoded in a combinatorial object which we will call the $s$-permutahedron. The ``faces" of the $s$-permutahedron will be given by pure intervals of the $s$-weak order, which we now introduce.

\subsection{Pure intervals of the $s$-weak order}

%In this section, we study a certain family of intervals with remarkable properties. They are the first hint that the $s$-weak order processes an underlying geometric structure. The purpose of the present paper is to explore the combinatorial properties of these intervals to lay the ground of the future geometric study.

Let $T$ be an $s$-decreasing tree and $A$ be a (possibly empty) subset of tree-ascents of $T$. We denote by $T+A$ the $s$-decreasing tree whose inversion set is obtained by increasing the cardinality $\card_T(c,a)$ for $(a,c) \in A$ by one  and then taking the transitive closure.\footnote{We use this notation for simplicity but it is not entirely consistent with~\cite{CP22}. Indeed, we used the notation $+(c,a)$ in~\cite{CP22} on multi-sets to increase the cardinality $\card(c,a)$ \emph{without} taking the transitive closure. Besides, $A$ is formally a set of tree-ascents $(a,c)$ and not a set of tree-inversions $(c,a)$.}

\begin{definition}
\label{def:pure-intervals}
Let $T_1 \wole T_2$ be two $s$-decreasing trees for a given weak composition $s$. We say that the interval $[T_1, T_2]$ is a \defn{pure interval} if $T_2 = T_1 + A$ with $A$ a subset of tree-ascents of $T_1$. 
\end{definition}

We present examples on Figures~\ref{fig:pure-interval},~\ref{fig:pure-interval10}, and~\ref{fig:pure-interval-ascentope}. 

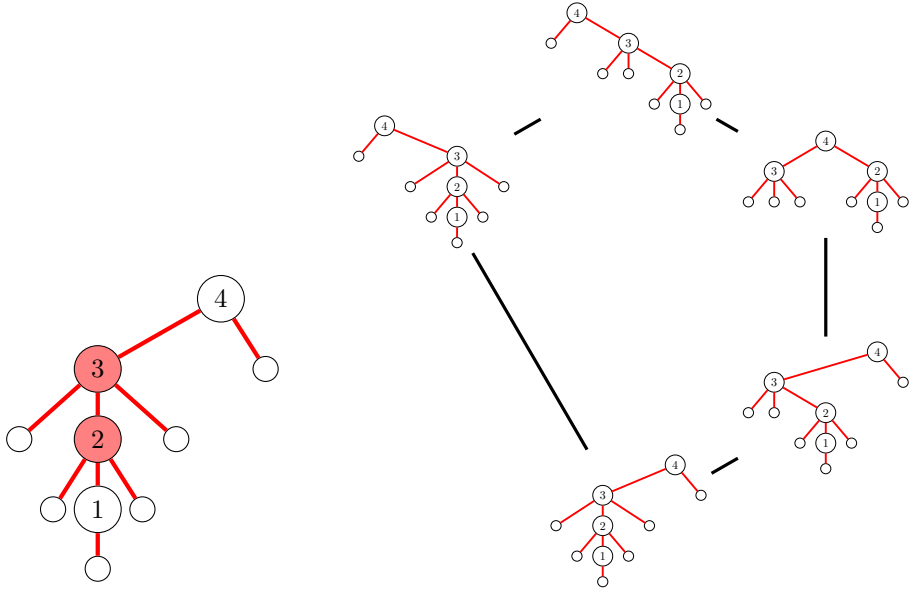
\begin{figure}[ht]
\begin{tabular}{cc}
{ \newcommand{\nodea}{\node[draw,circle,fill=white] (a) {$4$}
;}\newcommand{\nodeb}{\node[draw,circle,fill=red!50] (b) {$3$}
;}\newcommand{\nodec}{\node[draw,circle,fill=white] (c) {$$}
;}\newcommand{\noded}{\node[draw,circle,fill=red!50] (d) {$2$}
;}\newcommand{\nodee}{\node[draw,circle,fill=white] (e) {$$}
;}\newcommand{\nodef}{\node[draw,circle,fill=white] (f) {$1$}
;}\newcommand{\nodeg}{\node[draw,circle,fill=white] (g) {$$}
;}\newcommand{\nodeh}{\node[draw,circle,fill=white] (h) {$$}
;}\newcommand{\nodei}{\node[draw,circle,fill=white] (i) {$$}
;}\newcommand{\nodej}{\node[draw,circle,fill=white] (j) {$$}
;}\begin{tikzpicture}[auto]
\matrix[column sep=0.1cm, row sep=.3cm,ampersand replacement=\&]{
         \&         \&         \&         \&         \& \nodea  \&         \\ 
         \&         \& \nodeb  \&         \&         \&         \& \nodej  \\ 
 \nodec  \&         \& \noded  \&         \& \nodei  \&         \&         \\ 
         \& \nodee  \& \nodef  \& \nodeh  \&         \&         \&         \\ 
         \&         \& \nodeg  \&         \&         \&         \&         \\
};

\path[ultra thick, red] (f) edge (g)
	(d) edge (e) edge (f) edge (h)
	(b) edge (c) edge (d) edge (i)
	(a) edge (b) edge (j);
\end{tikzpicture}}
&
\scalebox{3}{
\begin{tikzpicture}[every node/.style={inner sep = 0pt}]
\node(tree0) at (0.866025403784439,0.500000000000000) {\scalebox{0.15}{$
{ \newcommand{\nodea}{\node[draw,circle] (a) {$4$}
;}\newcommand{\nodeb}{\node[draw,circle] (b) {$3$}
;}\newcommand{\nodec}{\node[draw,circle] (c) {$$}
;}\newcommand{\noded}{\node[draw,circle] (d) {$2$}
;}\newcommand{\nodee}{\node[draw,circle] (e) {$$}
;}\newcommand{\nodef}{\node[draw,circle] (f) {$1$}
;}\newcommand{\nodeg}{\node[draw,circle] (g) {$$}
;}\newcommand{\nodeh}{\node[draw,circle] (h) {$$}
;}\newcommand{\nodei}{\node[draw,circle] (i) {$$}
;}\newcommand{\nodej}{\node[draw,circle] (j) {$$}
;}\begin{tikzpicture}[every node/.style={inner sep = 3pt}]
\matrix[column sep=.3cm, row sep=.3cm,ampersand replacement=\&]{
         \&         \&         \&         \&         \& \nodea  \&         \\ 
         \&         \& \nodeb  \&         \&         \&         \& \nodej  \\ 
 \nodec  \&         \& \noded  \&         \& \nodei  \&         \&         \\ 
         \& \nodee  \& \nodef  \& \nodeh  \&         \&         \&         \\ 
         \&         \& \nodeg  \&         \&         \&         \&         \\
};

\path[ultra thick, red] (f) edge (g)
	(d) edge (e) edge (f) edge (h)
	(b) edge (c) edge (d) edge (i)
	(a) edge (b) edge (j);
\end{tikzpicture}}
$}};
\node(tree1) at (0.000000000000000,2.00000000000000) {\scalebox{0.15}{$
{ \newcommand{\nodea}{\node[draw,circle] (a) {$4$}
;}\newcommand{\nodeb}{\node[draw,circle] (b) {$$}
;}\newcommand{\nodec}{\node[draw,circle] (c) {$3$}
;}\newcommand{\noded}{\node[draw,circle] (d) {$$}
;}\newcommand{\nodee}{\node[draw,circle] (e) {$2$}
;}\newcommand{\nodef}{\node[draw,circle] (f) {$$}
;}\newcommand{\nodeg}{\node[draw,circle] (g) {$1$}
;}\newcommand{\nodeh}{\node[draw,circle] (h) {$$}
;}\newcommand{\nodei}{\node[draw,circle] (i) {$$}
;}\newcommand{\nodej}{\node[draw,circle] (j) {$$}
;}\begin{tikzpicture}[every node/.style={inner sep = 3pt}]
\matrix[column sep=.3cm, row sep=.3cm,ampersand replacement=\&]{
         \& \nodea  \&         \&         \&         \&         \&         \\ 
 \nodeb  \&         \&         \&         \& \nodec  \&         \&         \\ 
         \&         \& \noded  \&         \& \nodee  \&         \& \nodej  \\ 
         \&         \&         \& \nodef  \& \nodeg  \& \nodei  \&         \\ 
         \&         \&         \&         \& \nodeh  \&         \&         \\
};

\path[ultra thick, red] (g) edge (h)
	(e) edge (f) edge (g) edge (i)
	(c) edge (d) edge (e) edge (j)
	(a) edge (b) edge (c);
\end{tikzpicture}}
$}};
\node(tree2) at (1.73205080756888,1.00000000000000) {\scalebox{0.15}{$
{ \newcommand{\nodea}{\node[draw,circle] (a) {$4$}
;}\newcommand{\nodeb}{\node[draw,circle] (b) {$3$}
;}\newcommand{\nodec}{\node[draw,circle] (c) {$$}
;}\newcommand{\noded}{\node[draw,circle] (d) {$$}
;}\newcommand{\nodee}{\node[draw,circle] (e) {$2$}
;}\newcommand{\nodef}{\node[draw,circle] (f) {$$}
;}\newcommand{\nodeg}{\node[draw,circle] (g) {$1$}
;}\newcommand{\nodeh}{\node[draw,circle] (h) {$$}
;}\newcommand{\nodei}{\node[draw,circle] (i) {$$}
;}\newcommand{\nodej}{\node[draw,circle] (j) {$$}
;}\begin{tikzpicture}[every node/.style={inner sep = 3pt}]
\matrix[column sep=.3cm, row sep=.3cm,ampersand replacement=\&]{
         \&         \&         \&         \&         \& \nodea  \&         \\ 
         \& \nodeb  \&         \&         \&         \&         \& \nodej  \\ 
 \nodec  \& \noded  \&         \& \nodee  \&         \&         \&         \\ 
         \&         \& \nodef  \& \nodeg  \& \nodei  \&         \&         \\ 
         \&         \&         \& \nodeh  \&         \&         \&         \\
};

\path[ultra thick, red] (g) edge (h)
	(e) edge (f) edge (g) edge (i)
	(b) edge (c) edge (d) edge (e)
	(a) edge (b) edge (j);
\end{tikzpicture}}
$}};
\node(tree3) at (1.73205080756888,2.00000000000000) {\scalebox{0.15}{$
{ \newcommand{\nodea}{\node[draw,circle] (a) {$4$}
;}\newcommand{\nodeb}{\node[draw,circle] (b) {$3$}
;}\newcommand{\nodec}{\node[draw,circle] (c) {$$}
;}\newcommand{\noded}{\node[draw,circle] (d) {$$}
;}\newcommand{\nodee}{\node[draw,circle] (e) {$$}
;}\newcommand{\nodef}{\node[draw,circle] (f) {$2$}
;}\newcommand{\nodeg}{\node[draw,circle] (g) {$$}
;}\newcommand{\nodeh}{\node[draw,circle] (h) {$1$}
;}\newcommand{\nodei}{\node[draw,circle] (i) {$$}
;}\newcommand{\nodej}{\node[draw,circle] (j) {$$}
;}\begin{tikzpicture}[every node/.style={inner sep = 3pt}]
\matrix[column sep=.3cm, row sep=.3cm,ampersand replacement=\&]{
         \&         \&         \& \nodea  \&         \&         \&         \\ 
         \& \nodeb  \&         \&         \&         \& \nodef  \&         \\ 
 \nodec  \& \noded  \& \nodee  \&         \& \nodeg  \& \nodeh  \& \nodej  \\ 
         \&         \&         \&         \&         \& \nodei  \&         \\
};

\path[ultra thick, red] (b) edge (c) edge (d) edge (e)
	(h) edge (i)
	(f) edge (g) edge (h) edge (j)
	(a) edge (b) edge (f);
\end{tikzpicture}}
$}};
\node(tree4) at (0.866025403784439,2.50000000000000) {\scalebox{0.15}{$
{ \newcommand{\nodea}{\node[draw,circle] (a) {$4$}
;}\newcommand{\nodeb}{\node[draw,circle] (b) {$$}
;}\newcommand{\nodec}{\node[draw,circle] (c) {$3$}
;}\newcommand{\noded}{\node[draw,circle] (d) {$$}
;}\newcommand{\nodee}{\node[draw,circle] (e) {$$}
;}\newcommand{\nodef}{\node[draw,circle] (f) {$2$}
;}\newcommand{\nodeg}{\node[draw,circle] (g) {$$}
;}\newcommand{\nodeh}{\node[draw,circle] (h) {$1$}
;}\newcommand{\nodei}{\node[draw,circle] (i) {$$}
;}\newcommand{\nodej}{\node[draw,circle] (j) {$$}
;}\begin{tikzpicture}[every node/.style={inner sep = 3pt}]
\matrix[column sep=.3cm, row sep=.3cm,ampersand replacement=\&]{
         \& \nodea  \&         \&         \&         \&         \&         \\ 
 \nodeb  \&         \&         \& \nodec  \&         \&         \&         \\ 
         \&         \& \noded  \& \nodee  \&         \& \nodef  \&         \\ 
         \&         \&         \&         \& \nodeg  \& \nodeh  \& \nodej  \\ 
         \&         \&         \&         \&         \& \nodei  \&         \\
};

\path[ultra thick, red] (h) edge (i)
	(f) edge (g) edge (h) edge (j)
	(c) edge (d) edge (e) edge (f)
	(a) edge (b) edge (c);
\end{tikzpicture}}
$}};

\draw (tree0) -- (tree1);
\draw (tree0) -- (tree2);
\draw (tree1) -- (tree4);
\draw (tree2) -- (tree3);
\draw (tree3) -- (tree4);
\end{tikzpicture}
}
\end{tabular}
\caption{Example of a pure interval. On the left, the minimal tree with colored tree-ascents and on the right the corresponding interval.}
\label{fig:pure-interval}
\end{figure}

\begin{figure}[ht]
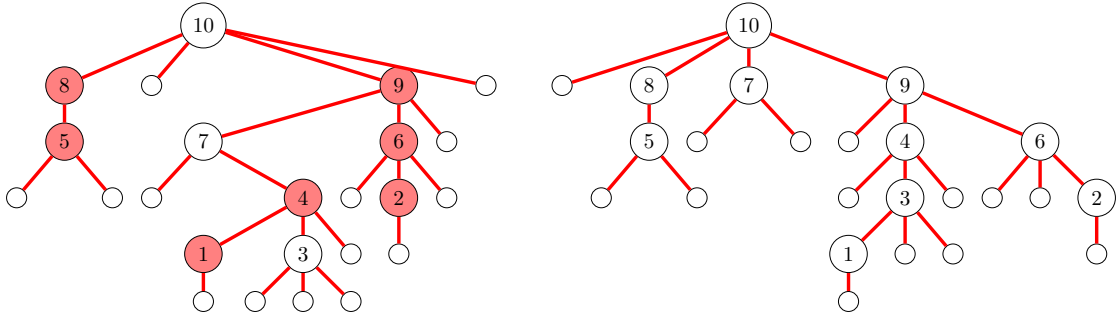

\begin{tabular}{cc}
\scalebox{.8}{\input{figures/dtrees/F10}} &
\scalebox{.8}{\input{figures/dtrees/T10_example_pure_max}}
\end{tabular}
\caption{Example of a pure interval. On the left, the minimal tree of the interval with selected tress-ascents and on the right, the maximal tree of the corresponding interval.}
\label{fig:pure-interval10}
\end{figure}

\begin{figure}[ht]
\begin{tabular}{cc}
\scalebox{.8}{\input{figures/dtrees/Fascentope}} &
\includegraphics[width=.4\textwidth]{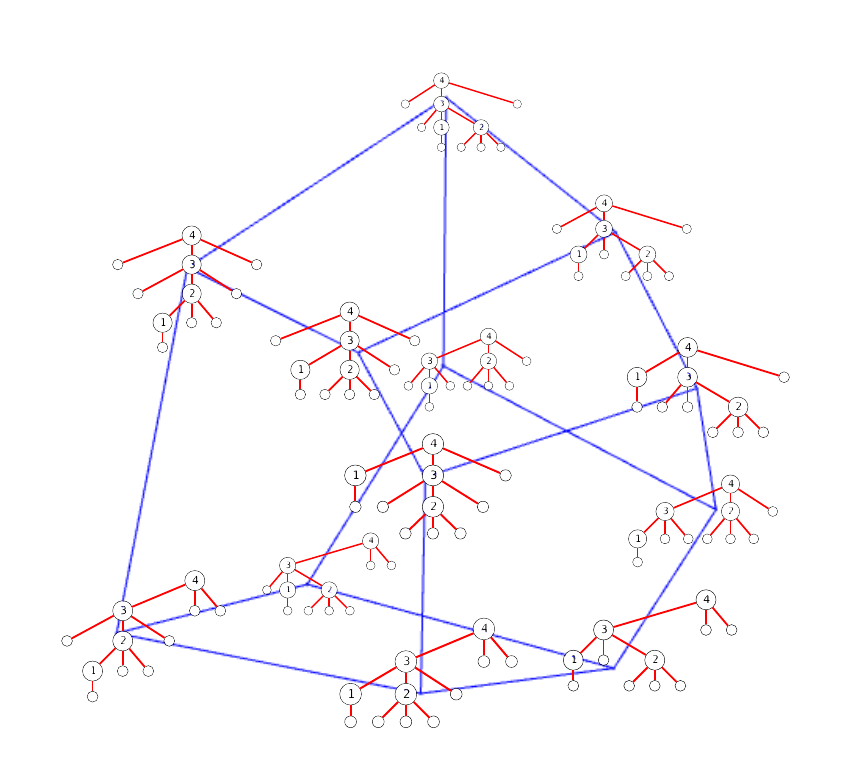}
\end{tabular}
\caption{Example of a pure interval. the minimal tree of the interval with selected tress-ascents and on the right, the interval drawn as the skeleton of a $3$d polytope.}
\label{fig:pure-interval-ascentope}
\end{figure}

A pure interval $[T_1, T_1 + A]$ is represented by the tree $T_1$ where we color in red all nodes $a$ such that there is a tree-ascent $(a,c)$ in $A$. By Remark~1.25 of~\cite{CP22}, the smaller value $a$ is enough to identify the tree-ascent $(a,c)$: indeed $c$ is the ascendant of $a$ with minimal value such that $\card_{T_1}(c,a) < s(c)$. It is the only possibility for a tree-ascent $(a,*)$ to exist (whereas there could be other tree-ascents $(*,c)$). 

On Figure~\ref{fig:pure-interval}, the nodes $2$ and $3$ are colored, representing respectively the tree-ascents $(2,3)$ and $(3,4)$. On the right, the corresponding interval is shown. The maximum is obtained by increasing the cardinalities of the tree-inversions $(3,2)$ and $(4,3)$ and taking the transitive closure. For example, we see that $(4,1)$ has been increased by transitivity because $(4,3)$ was increased and $\card(3,1) > 0$. This is also equivalent to taking the lattice join of the two trees obtained by applying the two $s$-tree-rotations of the two selected tree-ascents (see \Cref{lem_maximal_equal_join}). We present a second example on Figure~\ref{fig:pure-interval10}, showing the minimal element of the interval with selected tree-ascents and the maximal element of the interval. You can check that all tree-inversions corresponding to tree-ascents have been increased.

\begin{lemma}\label{lem_maximal_equal_join}
Let $T_1$ and $T_2$ be two $s$-decreasing trees such that 
$T_2 = T_1+A$ with $A$ a subset of tree-ascents of $T_1$. Then $T_2$ can be obtained as the join
$
T_2 = \bigvee_{a\in A} (T_1+a)
$ 
over all trees $T_1+a$ obtained from~$T_1$ by rotating a tree-ascent~$a\in A$.  
\end{lemma}

\begin{proof}
We denote by $T_2$ the join of the rotated trees by the tree-ascents of $A$ and prove that \mbox{$T_2 = T_1 +A$}. For a tree-ascent $(a,c)$, let $M_{(a,c)}$ be the multi-set of inversions such that $\card_{M_{(a,c)}}(c,a) = \card_{T_1}(c,a) + 1$ while all other values are unchanged. Similarly, we write $M_A := \cup_{(a,c) \in A} M_{(a,c)}$, \emph{i.e.}, all cardinalities $(c,a)$ with $(a,c) \in A$ have been increased by one. In particular, the tree-inversions of the tree $T_1 + (a,c)$ resulting from a single rotation $(a,c)$ are given by $\tc{M_{(a,c)}}$, \emph{i.e.}, the transitive closure of $M_{(a,c)}$. Similarly, the tree-inversions of $T+A$ are given by $\tc{M_A}$.  

By definition, we have $M_{(a,c)} \subseteq M_A$ for all $(a,c) \in A$ where the inclusion is the multi-set inclusion defined in Definition~\ref{def:multi-sets}. This is compatible with the multi-set transitive closure of Definition~\ref{def:tc-tree-inversions}, and so we have $\tc{M_{(a,c)}} \subseteq \tc{M_A}$. In other words, $T_1 + (a,c) \wole T_1 + A$. By taking the lattice join, we obtain $T_2 \wole T_1 +A$. 

Now by definition of the join, the inversions of $T_2$ are given by $\inv(T_2) = \tc{M_2}$ where $M_2 = \cup_{(a,c) \in A} \tc{M_{(a,c)}}$. In particular, $M_2 \neq M_A$ as it is the union of the transitive closures of the $M_{(a,c)}$ multi-sets. But, as the transitive closure can only increase the cardinality of a tree-inversion, for each $(a,c) \in A$ we know that $\card_{M_2}(c,a) \geq \card_{M_{(a,c)}}(c,a) = \card_{T_1}(c,a) + 1 = \card_{M_A}(c,a)$. Besides, for each couple $a' < c'$ such that $(a',c') \not\in A$, we have $\card_{M_2}(c',a') \geq \card_{T_1}(c',a') = \card_{M_A}(c',a')$. We obtain $M_A \subseteq M_2$. By taking the transitive closure, we have $\inv(T_1 + A) \subseteq \inv(T_2)$ and so $T_1 + A \wole T_2$. 
\end{proof}

\subsection{The combinatorial complex of pure intervals}

We are now ready to define one of the main objects of this paper, the $s$-permutahedron.

 \begin{definition}[The $s$-permutahedron]
 \label{def:s-perm}
 The \defn{$s$-permutahedron} $\Perm{s}$ is the collection of pure intervals $[T,T+A]$ of the $s$-weak order. Here, $T$ denotes an $s$-decreasing tree and $A$ a subset of tree-ascents of $T$. The \defn{dimension} of $[T,T+A]$ is said to be equal to $|A|$. In particular, 
 \begin{enumerate}
\item the vertices of $\Perm{s}$ are $s$-decreasing trees $T$, and
\item two $s$-decreasing trees are connected by an edge if and only if they are related by an $s$-tree rotation. 
\end{enumerate} 
We often refer to pure intervals $[T,T+A]$ as \defn{faces} of $\Perm{s}$, and say that one face is contained in another if the containment holds as intervals in the $s$-weak order. 
 \end{definition}

Figure~\ref{fig:s022-complex} illustrates an example of the $s$-permutahedron $\Perm{0,2,2}$. As we can see, it is a polytopal complex whose faces are labeled by pure intervals, and whose edge graph is the Hasse diagram of the $s$-weak order. 
 
As we have a (combinatorial) notion of face and dimension, we can also define the \defn{$f$-polynomial} associated to a weak composition $s$ by

\begin{equation}
f_s(t) := \sum_{F \in \Perm{s}} t^{\sdim(F)}
\end{equation}

The $f$-polynomial for our example in Figure~\ref{fig:s022-complex} is

\begin{equation}
f_{(0,2,2)}(t) = 15 + 20 t + 6 t^2.
\end{equation}

Indeed, there are $15$ $s$-decreasing trees (faces of dimension $0$), $20$ edges (faces of dimension $1$) and $6$ pure intervals of dimension $2$, which correspond to the $6$ polygons that appear naturally on the lattice. The following proposition is straightforward from the definition.

\begin{proposition}
\label{prop:f-poly}
The $f$-polynomial of the $s$-permutahedron $\Perm{s}$ is given by
\[
\sum_{T} (1+t)^{\asc(T)},
\]
where the sum runs over all $s$-decreasing trees $T$ and $\asc(T)$ denotes the number of ascents of $T$.
\end{proposition}

\begin{proof}
Let 
\[
f_s(t) = \sum_{T} (1+t)^{\asc(T)}
= c_0+c_1t+c_2t^2+\dots.
\]
We need to show that $c_k$ counts the number of $k$-dimensional faces of $\Perm{s}$. 
This follows from the fact that every subset of ascents $A$ of $T$, of size $k$, contributes a $t^k$ to the term $(1+t)^{\asc(T)}$. 
\end{proof}

The polynomial $\sum_{T} t^{\asc(T)}$ is a natural generalization of the Eulerian polynomial which we call the \defn{$s$-Eulerian polynomial}. Its coefficients are called the \defn{$s$-Eulerian numbers}.

There is actually a recursive way to compute $f_s$ which is much better in practice than enumerating all the $s$-decreasing trees. This is described in the following proposition.

\begin{proposition}
\label{prop:f-poly-recur}
If $s$ is a weak composition of length $1$, then $f_s(t) = 1$. 
If $s$ be a weak composition of length $n \geq 2$, then we write $s = (\tilde{s}, u, v)$ where $\tilde{s}$ is the sequence obtained by removing the last two values in $s$, $u = s(n-1)$ and $v = s(n)$. Then

\begin{enumerate}
\item if $n > 2$ and $u = 0$, then $f_s = f_{(\tilde{s}, v)} f_{(0, v)}$;
\label{prop-case:f-poly-recur-0}
\item if $n = 2$ or $u > 0$, then $f_s(t) = (v + 1)f_{(\tilde{s}, u + v)}(t) + vt f_{(\tilde{s}, u +  v - 1)}(t).$
\label{prop-case:f-poly-recur-1}
\end{enumerate}
\end{proposition}

For example if $s = (u,v)$ is a weak composition of length $2$. The $s$-decreasing trees have two nodes labeled $1$ and $2$. The node $2$ is the root and has $v+1$ children. The $s$-weak order is the chain between the $s$-decreasing tree where $1$ is at the extreme left and the $s$-decreasing tree where $1$ is at the extreme right. There are $v+1$ $s$-decreasing trees (corresponding to the possible positions of $1$) and $v$ edges. We get $f_{(u,v)}(t) = v + 1 + vt$. By Proposition~\ref{prop:f-poly}, we also obtain $f_{(u,v)}(t) = v(1+t) + 1$. Indeed all trees but the maximal one have one ascent. The recursive computation of Proposition~\ref{prop:f-poly-recur} also gives $f_{(u,v)}(t) = (v+1)f_{(u+v)}(t) + vt f_{(u+v - 1)} = v + 1 + vt$.

We can also compute the $f$-polynomial of $\Perm{0,2,2}$ in Figure~\ref{fig:s022-complex} using our recursive formula in Proposition~\ref{prop:f-poly-recur}. We obtain 
$$f_{(0,2,2)}(t) = 3 f_{(0, 4)}(t) + 2t f_{(0, 3)}(t) = 3 (5 + 4t) + 2t (4 + 3t) = 15 + 20t + 6 t^2.$$

For the $\Perm{0,0,2}$ (Figure~\ref{fig:s002-complex}), we get

$$f_{(0,0,2)}(t) = f_{(0,2)}(t) f_{(0,2)}(t) = (3 + 2t)^2 = 9 + 12t + 4t^2$$.

We present many computations in Table~\ref{tab:fs}. You can also compute more example using our demo {\tt SageMath} worsheet~\cite{SageDemoII}.

\begin{proof}
The proof can be done bijectively. The base case is trivial: if $s$ is of length $1$, whatever the value of $s(1)$ is, there is a unique $s$-decreasing tree with no ascent giving $f(s) = 1$. We then define two different bijections corresponding to Cases~\eqref{prop-case:f-poly-recur-0} and~\eqref{prop-case:f-poly-recur-1} and illustrate them on examples in Figures~\ref{fig:fpoly-recurs0} and~\ref{fig:fpoly-recurs1} respectively. For Case~\eqref{prop-case:f-poly-recur-1}, the bijection is divided into two different cases, depending on whether or not $(n-1,n)$ is a selected tree-ascent. This corresponds to the two summands of the recursion. Looking at the examples, we believe that the bijections are straightforward. We explain them in detail below.

Let us prove Case~\eqref{prop-case:f-poly-recur-0}. We have $s = (\tilde{s}, 0, v)$. We prove that the pure intervals of $s$ are in bijection with products of pure intervals of $(\tilde{s}, v)$ and $(0,v)$. Let $T$ be an $s$-decreasing tree and $A$ a subset of tree-ascents of $T$. The root of $T$ is necessarily $n$ with $v+1$ subtrees. The node $n-1$ is the root of one of the subtrees. As we have $s(n-1) = 0$, the node $n-1$ has a unique child which we write $Q$ (it can be empty or has a root $i < n-1$). Let $s' = (\tilde{s}, v)$, we construct an $s'$-decreasing tree $T'$ by replacing the node $n-1$ by its unique child $Q$. All the other subtrees stay the same. We relabel the root to $n-1$ instead of $n$ to keep a standard labeling. If $(a,c)$ with $a < c < n - 1$ is a tree-ascent in~$T$, it is still a tree-ascent in $T'$: it is either in $Q$ or in any of the other subtrees of $n$. As $s(n-1) = 0$, we cannot have a tree-ascent $(a,n-1)$. Besides, if $(a,n)$ is a tree-ascent in $T$ with $a \neq n-1$, then $(a,n-1)$ is a tree-ascent in $T'$. Indeed, all conditions of Definition~\ref{def:tree-ascent} are still satisfied. We set $A' = \lbrace (a,c) \in A; a < c < n-1 \rbrace \cup \lbrace (a,n-1); (a,n) \in A \text{ and } a \neq n-1 \rbrace$. 

Now we define $s'' = (0,v)$ and $T''$ the $s''$-decreasing tree obtained by keeping only the nodes $n$ and $n-1$, which we relabel into $2$ and $1$ respectively. If $(n-1, n)$ is a tree-ascent of $T$, we set $A'' = \lbrace (1,2) \rbrace$, otherwise $A'' = \varnothing$. In the end, we have $|A'| + |A''| = |A|$ so the monomial corresponding to the pure interval $(T,A)$ is obtained by a product of monomials of $(T',A')$ and $(T'', A'')$. Besides, the reverse operation is easy: $T''$ gives the position where $n-1$ needs to be inserted. The selected tree-ascents stay the same up to relabeling. In particular, if $(a,n-1)$ is a tree-ascent in $T'$, then $(a,n)$ is a tree-ascent in~$T$ even if $a \in Q$ because it is still at the right end of the subtree. We illustrate this bijection on Figure~\ref{fig:fpoly-recurs0} with $s = (0,0,2,0,3)$.

\begin{figure}[ht]
\scalebox{.8}{\begin{tabular}{ccccc}
{ \newcommand{\nodea}{\node[draw,circle,fill=white] (a) {$5$}
;}\newcommand{\nodeb}{\node[draw,circle,fill=white] (b) {$2$}
;}\newcommand{\nodec}{\node[draw,circle,fill=white] (c) {$$}
;}\newcommand{\noded}{\node[draw,circle,fill=white] (d) {$$}
;}\newcommand{\nodee}{\node[draw,circle,fill=red!50] (e) {$4$}
;}\newcommand{\nodef}{\node[draw,circle,fill=red!50] (f) {$3$}
;}\newcommand{\nodeg}{\node[draw,circle,fill=white] (g) {$$}
;}\newcommand{\nodeh}{\node[draw,circle,fill=red!50] (h) {$1$}
;}\newcommand{\nodei}{\node[draw,circle,fill=white] (i) {$$}
;}\newcommand{\nodej}{\node[draw,circle,fill=white] (j) {$$}
;}\newcommand{\nodeba}{\node[draw,circle,fill=white] (ba) {$$}
;}\begin{tikzpicture}[baseline]
\matrix[column sep=0.1cm, row sep=.3cm,ampersand replacement=\&]{
         \&         \& \nodea  \&         \&         \&         \\ 
 \nodeb  \& \noded  \&         \& \nodee  \&         \& \nodeba \\ 
 \nodec  \&         \&         \& \nodef  \&         \&         \\ 
         \&         \& \nodeg  \& \nodeh  \& \nodej  \&         \\ 
         \&         \&         \& \nodei  \&         \&         \\
};

\path[ultra thick, red] (b) edge (c)
	(h) edge (i)
	(f) edge (g) edge (h) edge (j)
	(e) edge (f)
	(a) edge (b) edge (d) edge (e) edge (ba);
\end{tikzpicture}}
 & $\rightarrow$ 
& { \newcommand{\nodea}{\node[draw,circle,fill=white] (a) {$4$}
;}\newcommand{\nodeb}{\node[draw,circle,fill=white] (b) {$2$}
;}\newcommand{\nodec}{\node[draw,circle,fill=white] (c) {$$}
;}\newcommand{\noded}{\node[draw,circle,fill=white] (d) {$$}
;}\newcommand{\nodee}{\node[draw,circle,fill=red!50] (e) {$3$}
;}\newcommand{\nodef}{\node[draw,circle,fill=white] (f) {$$}
;}\newcommand{\nodeg}{\node[draw,circle,fill=red!50] (g) {$1$}
;}\newcommand{\nodeh}{\node[draw,circle,fill=white] (h) {$$}
;}\newcommand{\nodei}{\node[draw,circle,fill=white] (i) {$$}
;}\newcommand{\nodej}{\node[draw,circle,fill=white] (j) {$$}
;}\begin{tikzpicture}[baseline]
\matrix[column sep=0.1cm, row sep=.3cm,ampersand replacement=\&]{
         \&         \& \nodea  \&         \&         \&         \\ 
 \nodeb  \& \noded  \&         \& \nodee  \&         \& \nodej  \\ 
 \nodec  \&         \& \nodef  \& \nodeg  \& \nodei  \&         \\ 
         \&         \&         \& \nodeh  \&         \&         \\
};

\path[ultra thick, red] (b) edge (c)
	(g) edge (h)
	(e) edge (f) edge (g) edge (i)
	(a) edge (b) edge (d) edge (e) edge (j);
\end{tikzpicture}}
& & { \newcommand{\nodea}{\node[draw,circle,fill=white] (a) {$2$}
;}\newcommand{\nodeb}{\node[draw,circle,fill=white] (b) {$$}
;}\newcommand{\nodec}{\node[draw,circle,fill=white] (c) {$$}
;}\newcommand{\noded}{\node[draw,circle,fill=red!50] (d) {$1$}
;}\newcommand{\nodee}{\node[draw,circle,fill=white] (e) {$$}
;}\newcommand{\nodef}{\node[draw,circle,fill=white] (f) {$$}
;}\begin{tikzpicture}[baseline]
\matrix[column sep=0.1cm, row sep=.3cm,ampersand replacement=\&]{
         \&         \& \nodea  \&         \&         \\ 
 \nodeb  \& \nodec  \&         \& \noded  \& \nodef  \\ 
         \&         \&         \& \nodee  \&         \\
};

\path[ultra thick, red] (d) edge (e)
	(a) edge (b) edge (c) edge (d) edge (f);
\end{tikzpicture}} \\
$t^3$ & $=$ & $t^2$ & $\times$ & $t$ \\ \hline
{ \newcommand{\nodea}{\node[draw,circle,fill=white] (a) {$5$}
;}\newcommand{\nodeb}{\node[draw,circle,fill=white] (b) {$$}
;}\newcommand{\nodec}{\node[draw,circle,fill=white] (c) {$4$}
;}\newcommand{\noded}{\node[draw,circle,fill=red!50] (d) {$3$}
;}\newcommand{\nodee}{\node[draw,circle,fill=white] (e) {$2$}
;}\newcommand{\nodef}{\node[draw,circle,fill=white] (f) {$$}
;}\newcommand{\nodeg}{\node[draw,circle,fill=white] (g) {$$}
;}\newcommand{\nodeh}{\node[draw,circle,fill=white] (h) {$$}
;}\newcommand{\nodei}{\node[draw,circle,fill=red!50] (i) {$1$}
;}\newcommand{\nodej}{\node[draw,circle,fill=white] (j) {$$}
;}\newcommand{\nodeba}{\node[draw,circle,fill=white] (ba) {$$}
;}\begin{tikzpicture}[baseline]
\matrix[column sep=0.1cm, row sep=.3cm,ampersand replacement=\&]{
         \&         \&         \& \nodea  \&         \&         \\ 
 \nodeb  \&         \& \nodec  \&         \& \nodei  \& \nodeba \\ 
         \&         \& \noded  \&         \& \nodej  \&         \\ 
         \& \nodee  \& \nodeg  \& \nodeh  \&         \&         \\ 
         \& \nodef  \&         \&         \&         \&         \\
};

\path[ultra thick, red] (e) edge (f)
	(d) edge (e) edge (g) edge (h)
	(c) edge (d)
	(i) edge (j)
	(a) edge (b) edge (c) edge (i) edge (ba);
\end{tikzpicture}} & $\rightarrow$ 
& { \newcommand{\nodea}{\node[draw,circle,fill=white] (a) {$4$}
;}\newcommand{\nodeb}{\node[draw,circle,fill=white] (b) {$$}
;}\newcommand{\nodec}{\node[draw,circle,fill=red!50] (c) {$3$}
;}\newcommand{\noded}{\node[draw,circle,fill=white] (d) {$2$}
;}\newcommand{\nodee}{\node[draw,circle,fill=white] (e) {$$}
;}\newcommand{\nodef}{\node[draw,circle,fill=white] (f) {$$}
;}\newcommand{\nodeg}{\node[draw,circle,fill=white] (g) {$$}
;}\newcommand{\nodeh}{\node[draw,circle,fill=red!50] (h) {$1$}
;}\newcommand{\nodei}{\node[draw,circle,fill=white] (i) {$$}
;}\newcommand{\nodej}{\node[draw,circle,fill=white] (j) {$$}
;}\begin{tikzpicture}[baseline]
\matrix[column sep=0.1cm, row sep=.3cm,ampersand replacement=\&]{
         \&         \&         \& \nodea  \&         \&         \\ 
 \nodeb  \&         \& \nodec  \&         \& \nodeh  \& \nodej  \\ 
         \& \noded  \& \nodef  \& \nodeg  \& \nodei  \&         \\ 
         \& \nodee  \&         \&         \&         \&         \\
};

\path[ultra thick, red] (d) edge (e)
	(c) edge (d) edge (f) edge (g)
	(h) edge (i)
	(a) edge (b) edge (c) edge (h) edge (j);
\end{tikzpicture}}
& & { \newcommand{\nodea}{\node[draw,circle] (a) {$2$}
;}\newcommand{\nodeb}{\node[draw,circle] (b) {$$}
;}\newcommand{\nodec}{\node[draw,circle] (c) {$1$}
;}\newcommand{\noded}{\node[draw,circle] (d) {$$}
;}\newcommand{\nodee}{\node[draw,circle] (e) {$$}
;}\newcommand{\nodef}{\node[draw,circle] (f) {$$}
;}\begin{tikzpicture}[baseline]
\matrix[column sep=.3cm, row sep=.3cm,ampersand replacement=\&]{
         \&         \& \nodea  \&         \&         \\ 
 \nodeb  \& \nodec  \&         \& \nodee  \& \nodef  \\ 
         \& \noded  \&         \&         \&         \\
};

\path[ultra thick, red] (c) edge (d)
	(a) edge (b) edge (c) edge (e) edge (f);
\end{tikzpicture}} \\
$t^2$ & $=$ & $t^2$ & $\times$ & $1$
\end{tabular}}
\caption{Bijective illustration of Case~\eqref{prop-case:f-poly-recur-0} of the recursive $f$-polynomial computation.}
\label{fig:fpoly-recurs0}
\end{figure}
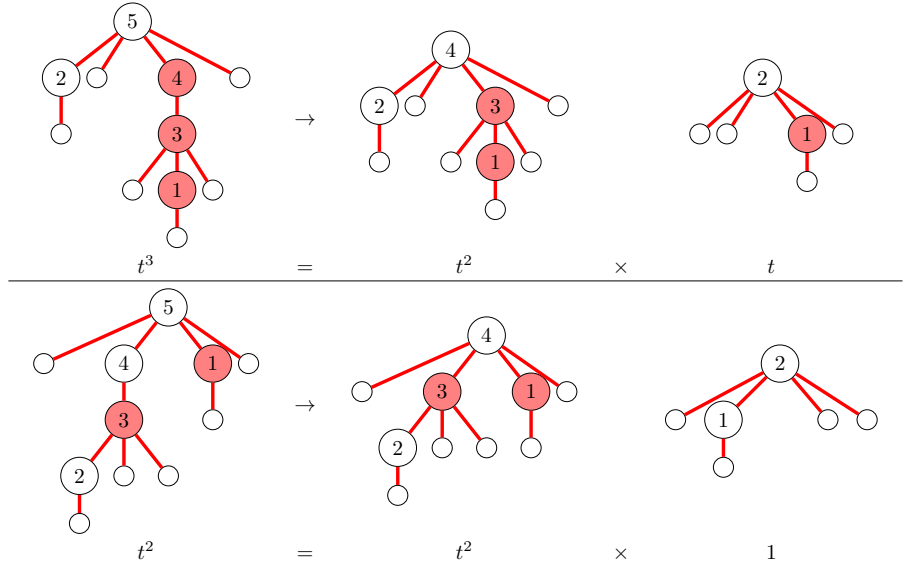

\begin{figure}[ht]
\scalebox{.8}{\begin{tabular}{ccc|ccc}
{ \newcommand{\nodea}{\node[draw,circle,fill=white] (a) {$5$}
;}\newcommand{\nodeb}{\node[draw,circle,fill=white] (b) {$4$}
;}\newcommand{\nodec}{\node[draw,circle,fill=white] (c) {$$}
;}\newcommand{\noded}{\node[draw,circle,fill=red!50] (d) {$3$}
;}\newcommand{\nodee}{\node[draw,circle,fill=white] (e) {$$}
;}\newcommand{\nodef}{\node[draw,circle,fill=white] (f) {$1$}
;}\newcommand{\nodeg}{\node[draw,circle,fill=white] (g) {$$}
;}\newcommand{\nodeh}{\node[draw,circle,fill=red!50] (h) {$2$}
;}\newcommand{\nodei}{\node[draw,circle,fill=white] (i) {$$}
;}\newcommand{\nodej}{\node[draw,circle,fill=white] (j) {$$}
;}\newcommand{\nodeba}{\node[draw,circle,fill=white] (ba) {$$}
;}\begin{tikzpicture}[baseline]
\matrix[column sep=0.1cm, row sep=.3cm,ampersand replacement=\&]{
         \&         \&         \&         \& \nodea  \&         \&         \\ 
         \& \nodeb  \&         \& \nodeh  \&         \& \nodej  \& \nodeba \\ 
 \nodec  \& \noded  \& \nodef  \& \nodei  \&         \&         \&         \\ 
         \& \nodee  \& \nodeg  \&         \&         \&         \&         \\
};

\path[ultra thick, red] (d) edge (e)
	(f) edge (g)
	(b) edge (c) edge (d) edge (f)
	(h) edge (i)
	(a) edge (b) edge (h) edge (j) edge (ba);
\end{tikzpicture}} & $\rightarrow$ & 
{ \newcommand{\nodea}{\node[draw,circle,fill=white] (a) {$4$}
;}\newcommand{\nodeb}{\node[draw,circle,fill=white] (b) {$$}
;}\newcommand{\nodec}{\node[draw,circle,fill=red!50] (c) {$3$}
;}\newcommand{\noded}{\node[draw,circle,fill=white] (d) {$$}
;}\newcommand{\nodee}{\node[draw,circle,fill=white] (e) {$1$}
;}\newcommand{\nodef}{\node[draw,circle,fill=white] (f) {$$}
;}\newcommand{\nodeg}{\node[draw,circle,fill=red!50] (g) {$2$}
;}\newcommand{\nodeh}{\node[draw,circle,fill=white] (h) {$$}
;}\newcommand{\nodei}{\node[draw,circle,fill=white] (i) {$$}
;}\newcommand{\nodej}{\node[draw,circle,fill=white] (j) {$$}
;}\begin{tikzpicture}[baseline]
\matrix[column sep=0.1cm, row sep=.7cm,ampersand replacement=\&]{
         \&         \&         \& \nodea  \&         \&         \\ 
 \nodeb  \& \nodec  \& \nodee  \& \nodeg  \& \nodei  \& \nodej  \\ 
         \& \noded  \& \nodef  \& \nodeh  \&         \&         \\
};

\path[ultra thick, red] (c) edge (d)
	(e) edge (f)
	(g) edge (h)
	(a) edge (g) edge (i) edge (j);

\path[ultra thick, dashed, red] (a) edge (b) edge (c) edge (e);
\end{tikzpicture}}
&
{ \newcommand{\nodea}{\node[draw,circle,fill=white] (a) {$5$}
;}\newcommand{\nodeb}{\node[draw,circle,fill=red!50] (b) {$4$}
;}\newcommand{\nodec}{\node[draw,circle,fill=white] (c) {$$}
;}\newcommand{\noded}{\node[draw,circle,fill=red!50] (d) {$3$}
;}\newcommand{\nodee}{\node[draw,circle,fill=white] (e) {$$}
;}\newcommand{\nodef}{\node[draw,circle,fill=white] (f) {$$}
;}\newcommand{\nodeg}{\node[draw,circle,fill=white] (g) {$1$}
;}\newcommand{\nodeh}{\node[draw,circle,fill=white] (h) {$$}
;}\newcommand{\nodei}{\node[draw,circle,fill=red!50] (i) {$2$}
;}\newcommand{\nodej}{\node[draw,circle,fill=white] (j) {$$}
;}\newcommand{\nodeba}{\node[draw,circle,fill=white] (ba) {$$}
;}\begin{tikzpicture}[baseline]
\matrix[column sep=0.1cm, row sep=.3cm,ampersand replacement=\&]{
         \&         \&         \&         \& \nodea  \&         \&         \\ 
         \& \nodeb  \&         \& \nodeg  \&         \& \nodei  \& \nodeba \\ 
 \nodec  \& \noded  \& \nodef  \& \nodeh  \&         \& \nodej  \&         \\ 
         \& \nodee  \&         \&         \&         \&         \&         \\
};

\path[ultra thick, red] (d) edge (e)
	(b) edge (c) edge (d) edge (f)
	(g) edge (h)
	(i) edge (j)
	(a) edge (b) edge (g) edge (i) edge (ba);
\end{tikzpicture}} & $\rightarrow$ &
{ \newcommand{\nodea}{\node[draw,circle,fill=white] (a) {$4$}
;}\newcommand{\nodeb}{\node[draw,circle,fill=white] (b) {$$}
;}\newcommand{\nodec}{\node[draw,circle,fill=red!50] (c) {$3$}
;}\newcommand{\noded}{\node[draw,circle,fill=white] (d) {$$}
;}\newcommand{\nodee}{\node[draw,circle,fill=white] (e) {$1$}
;}\newcommand{\nodef}{\node[draw,circle,fill=white] (f) {$$}
;}\newcommand{\nodeg}{\node[draw,circle,fill=red!50] (g) {$2$}
;}\newcommand{\nodeh}{\node[draw,circle,fill=white] (h) {$$}
;}\newcommand{\nodei}{\node[draw,circle,fill=white] (i) {$$}
;}\begin{tikzpicture}[baseline]
\matrix[column sep=0.1cm, row sep=.7cm,ampersand replacement=\&]{
         \&         \& \nodea  \&         \&         \\ 
 \nodeb  \& \nodec  \& \nodee  \& \nodeg  \& \nodei  \\ 
         \& \noded  \& \nodef  \& \nodeh  \&         \\
};

\path[ultra thick, red] (c) edge (d)
	(e) edge (f)
	(g) edge (h)
	(a) edge (e) edge (g) edge (i);
	
\path[ultra thick, dashed, red] (a) edge (b) edge (c);
\end{tikzpicture}}
\\
{ \newcommand{\nodea}{\node[draw,circle,fill=white] (a) {$5$}
;}\newcommand{\nodeb}{\node[draw,circle,fill=white] (b) {$$}
;}\newcommand{\nodec}{\node[draw,circle,fill=white] (c) {$4$}
;}\newcommand{\noded}{\node[draw,circle,fill=red!50] (d) {$3$}
;}\newcommand{\nodee}{\node[draw,circle,fill=white] (e) {$$}
;}\newcommand{\nodef}{\node[draw,circle,fill=white] (f) {$1$}
;}\newcommand{\nodeg}{\node[draw,circle,fill=white] (g) {$$}
;}\newcommand{\nodeh}{\node[draw,circle,fill=red!50] (h) {$2$}
;}\newcommand{\nodei}{\node[draw,circle,fill=white] (i) {$$}
;}\newcommand{\nodej}{\node[draw,circle,fill=white] (j) {$$}
;}\newcommand{\nodeba}{\node[draw,circle,fill=white] (ba) {$$}
;}\begin{tikzpicture}[baseline]
\matrix[column sep=0.1cm, row sep=.3cm,ampersand replacement=\&]{
         \&         \&         \& \nodea  \&         \&         \\ 
 \nodeb  \&         \& \nodec  \&         \& \nodej  \& \nodeba \\ 
         \& \noded  \& \nodef  \& \nodeh  \&         \&         \\ 
         \& \nodee  \& \nodeg  \& \nodei  \&         \&         \\
};

\path[ultra thick, red] (d) edge (e)
	(f) edge (g)
	(h) edge (i)
	(c) edge (d) edge (f) edge (h)
	(a) edge (b) edge (c) edge (j) edge (ba);
\end{tikzpicture}} & $\rightarrow$ & 
{ \newcommand{\nodea}{\node[draw,circle,fill=white] (a) {$4$}
;}\newcommand{\nodeb}{\node[draw,circle,fill=white] (b) {$$}
;}\newcommand{\nodec}{\node[draw,circle,fill=red!50] (c) {$3$}
;}\newcommand{\noded}{\node[draw,circle,fill=white] (d) {$$}
;}\newcommand{\nodee}{\node[draw,circle,fill=white] (e) {$1$}
;}\newcommand{\nodef}{\node[draw,circle,fill=white] (f) {$$}
;}\newcommand{\nodeg}{\node[draw,circle,fill=red!50] (g) {$2$}
;}\newcommand{\nodeh}{\node[draw,circle,fill=white] (h) {$$}
;}\newcommand{\nodei}{\node[draw,circle,fill=white] (i) {$$}
;}\newcommand{\nodej}{\node[draw,circle,fill=white] (j) {$$}
;}\begin{tikzpicture}[baseline]
\matrix[column sep=0.1cm, row sep=.7cm,ampersand replacement=\&]{
         \&         \&         \& \nodea  \&         \&         \\ 
 \nodeb  \& \nodec  \& \nodee  \& \nodeg  \& \nodei  \& \nodej  \\ 
         \& \noded  \& \nodef  \& \nodeh  \&         \&         \\
};

\path[ultra thick, red] (c) edge (d)
	(e) edge (f)
	(g) edge (h)
	(a) edge (b) edge (i) edge (j);

\path[ultra thick, dashed, red] (a) edge (c) edge (e) edge (g);
\end{tikzpicture}}
&
{ \newcommand{\nodea}{\node[draw,circle,fill=white] (a) {$5$}
;}\newcommand{\nodeb}{\node[draw,circle,fill=white] (b) {$$}
;}\newcommand{\nodec}{\node[draw,circle,fill=red!50] (c) {$4$}
;}\newcommand{\noded}{\node[draw,circle,fill=red!50] (d) {$3$}
;}\newcommand{\nodee}{\node[draw,circle,fill=white] (e) {$$}
;}\newcommand{\nodef}{\node[draw,circle,fill=white] (f) {$1$}
;}\newcommand{\nodeg}{\node[draw,circle,fill=white] (g) {$$}
;}\newcommand{\nodeh}{\node[draw,circle,fill=white] (h) {$$}
;}\newcommand{\nodei}{\node[draw,circle,fill=red!50] (i) {$2$}
;}\newcommand{\nodej}{\node[draw,circle,fill=white] (j) {$$}
;}\newcommand{\nodeba}{\node[draw,circle,fill=white] (ba) {$$}
;}\begin{tikzpicture}[baseline]
\matrix[column sep=0.1cm, row sep=.3cm,ampersand replacement=\&]{
         \&         \&         \& \nodea  \&         \&         \\ 
 \nodeb  \&         \& \nodec  \&         \& \nodei  \& \nodeba \\ 
         \& \noded  \& \nodef  \& \nodeh  \& \nodej  \&         \\ 
         \& \nodee  \& \nodeg  \&         \&         \&         \\
};

\path[ultra thick, red] (d) edge (e)
	(f) edge (g)
	(c) edge (d) edge (f) edge (h)
	(i) edge (j)
	(a) edge (b) edge (c) edge (i) edge (ba);
\end{tikzpicture}} & $\rightarrow$ &
{ \newcommand{\nodea}{\node[draw,circle,fill=white] (a) {$4$}
;}\newcommand{\nodeb}{\node[draw,circle,fill=white] (b) {$$}
;}\newcommand{\nodec}{\node[draw,circle,fill=red!50] (c) {$3$}
;}\newcommand{\noded}{\node[draw,circle,fill=white] (d) {$$}
;}\newcommand{\nodee}{\node[draw,circle,fill=white] (e) {$1$}
;}\newcommand{\nodef}{\node[draw,circle,fill=white] (f) {$$}
;}\newcommand{\nodeg}{\node[draw,circle,fill=red!50] (g) {$2$}
;}\newcommand{\nodeh}{\node[draw,circle,fill=white] (h) {$$}
;}\newcommand{\nodei}{\node[draw,circle,fill=white] (i) {$$}
;}\begin{tikzpicture}[baseline]
\matrix[column sep=0.1cm, row sep=.7cm,ampersand replacement=\&]{
         \&         \& \nodea  \&         \&         \\ 
 \nodeb  \& \nodec  \& \nodee  \& \nodeg  \& \nodei  \\ 
         \& \noded  \& \nodef  \& \nodeh  \&         \\
};

\path[ultra thick, red] (c) edge (d)
	(e) edge (f)
	(g) edge (h)
	(a) edge (b) edge (g) edge (i);
	
\path[ultra thick, dashed, red] (a) edge (c) edge (e);
\end{tikzpicture}}
\\
{ \newcommand{\nodea}{\node[draw,circle,fill=white] (a) {$5$}
;}\newcommand{\nodeb}{\node[draw,circle,fill=white] (b) {$$}
;}\newcommand{\nodec}{\node[draw,circle,fill=red!50] (c) {$3$}
;}\newcommand{\noded}{\node[draw,circle,fill=white] (d) {$$}
;}\newcommand{\nodee}{\node[draw,circle,fill=white] (e) {$4$}
;}\newcommand{\nodef}{\node[draw,circle,fill=white] (f) {$1$}
;}\newcommand{\nodeg}{\node[draw,circle,fill=white] (g) {$$}
;}\newcommand{\nodeh}{\node[draw,circle,fill=red!50] (h) {$2$}
;}\newcommand{\nodei}{\node[draw,circle,fill=white] (i) {$$}
;}\newcommand{\nodej}{\node[draw,circle,fill=white] (j) {$$}
;}\newcommand{\nodeba}{\node[draw,circle,fill=white] (ba) {$$}
;}\begin{tikzpicture}[baseline]
\matrix[column sep=0.1cm, row sep=.3cm,ampersand replacement=\&]{
         \&         \& \nodea  \&         \&         \&         \\ 
 \nodeb  \& \nodec  \&         \& \nodee  \&         \& \nodeba \\ 
         \& \noded  \& \nodef  \& \nodeh  \& \nodej  \&         \\ 
         \&         \& \nodeg  \& \nodei  \&         \&         \\
};

\path[ultra thick, red] (c) edge (d)
	(f) edge (g)
	(h) edge (i)
	(e) edge (f) edge (h) edge (j)
	(a) edge (b) edge (c) edge (e) edge (ba);
\end{tikzpicture}} & $\rightarrow$ & 
{ \newcommand{\nodea}{\node[draw,circle,fill=white] (a) {$4$}
;}\newcommand{\nodeb}{\node[draw,circle,fill=white] (b) {$$}
;}\newcommand{\nodec}{\node[draw,circle,fill=red!50] (c) {$3$}
;}\newcommand{\noded}{\node[draw,circle,fill=white] (d) {$$}
;}\newcommand{\nodee}{\node[draw,circle,fill=white] (e) {$1$}
;}\newcommand{\nodef}{\node[draw,circle,fill=white] (f) {$$}
;}\newcommand{\nodeg}{\node[draw,circle,fill=red!50] (g) {$2$}
;}\newcommand{\nodeh}{\node[draw,circle,fill=white] (h) {$$}
;}\newcommand{\nodei}{\node[draw,circle,fill=white] (i) {$$}
;}\newcommand{\nodej}{\node[draw,circle,fill=white] (j) {$$}
;}\begin{tikzpicture}[baseline]
\matrix[column sep=0.1cm, row sep=.7cm,ampersand replacement=\&]{
         \&         \&         \& \nodea  \&         \&         \\ 
 \nodeb  \& \nodec  \& \nodee  \& \nodeg  \& \nodei  \& \nodej  \\ 
         \& \noded  \& \nodef  \& \nodeh  \&         \&         \\
};

\path[ultra thick, red] (c) edge (d)
	(e) edge (f)
	(g) edge (h)
	(a) edge (b) edge (c) edge (j);

\path[ultra thick, dashed, red] (a) edge (e) edge (g) edge (i);
\end{tikzpicture}}
&
{ \newcommand{\nodea}{\node[draw,circle,fill=white] (a) {$5$}
;}\newcommand{\nodeb}{\node[draw,circle,fill=white] (b) {$$}
;}\newcommand{\nodec}{\node[draw,circle,fill=red!50] (c) {$3$}
;}\newcommand{\noded}{\node[draw,circle,fill=white] (d) {$$}
;}\newcommand{\nodee}{\node[draw,circle,fill=red!50] (e) {$4$}
;}\newcommand{\nodef}{\node[draw,circle,fill=white] (f) {$1$}
;}\newcommand{\nodeg}{\node[draw,circle,fill=white] (g) {$$}
;}\newcommand{\nodeh}{\node[draw,circle,fill=red!50] (h) {$2$}
;}\newcommand{\nodei}{\node[draw,circle,fill=white] (i) {$$}
;}\newcommand{\nodej}{\node[draw,circle,fill=white] (j) {$$}
;}\newcommand{\nodeba}{\node[draw,circle,fill=white] (ba) {$$}
;}\begin{tikzpicture}[baseline]
\matrix[column sep=0.1cm, row sep=.3cm,ampersand replacement=\&]{
         \&         \& \nodea  \&         \&         \&         \\ 
 \nodeb  \& \nodec  \&         \& \nodee  \&         \& \nodeba \\ 
         \& \noded  \& \nodef  \& \nodeh  \& \nodej  \&         \\ 
         \&         \& \nodeg  \& \nodei  \&         \&         \\
};

\path[ultra thick, red] (c) edge (d)
	(f) edge (g)
	(h) edge (i)
	(e) edge (f) edge (h) edge (j)
	(a) edge (b) edge (c) edge (e) edge (ba);
\end{tikzpicture}}
 & $\rightarrow$ &
{ \newcommand{\nodea}{\node[draw,circle,fill=white] (a) {$4$}
;}\newcommand{\nodeb}{\node[draw,circle,fill=white] (b) {$$}
;}\newcommand{\nodec}{\node[draw,circle,fill=red!50] (c) {$3$}
;}\newcommand{\noded}{\node[draw,circle,fill=white] (d) {$$}
;}\newcommand{\nodee}{\node[draw,circle,fill=white] (e) {$1$}
;}\newcommand{\nodef}{\node[draw,circle,fill=white] (f) {$$}
;}\newcommand{\nodeg}{\node[draw,circle,fill=red!50] (g) {$2$}
;}\newcommand{\nodeh}{\node[draw,circle,fill=white] (h) {$$}
;}\newcommand{\nodei}{\node[draw,circle,fill=white] (i) {$$}
;}\begin{tikzpicture}[baseline]
\matrix[column sep=0.1cm, row sep=.7cm,ampersand replacement=\&]{
         \&         \& \nodea  \&         \&         \\ 
 \nodeb  \& \nodec  \& \nodee  \& \nodeg  \& \nodei  \\ 
         \& \noded  \& \nodef  \& \nodeh  \&         \\
};

\path[ultra thick, red] (c) edge (d)
	(e) edge (f)
	(g) edge (h)
	(a) edge (b) edge (c) edge (i);
	
\path[ultra thick, dashed, red] (a) edge (e) edge (g);
\end{tikzpicture}}
\\
{ \newcommand{\nodea}{\node[draw,circle,fill=white] (a) {$5$}
;}\newcommand{\nodeb}{\node[draw,circle,fill=white] (b) {$$}
;}\newcommand{\nodec}{\node[draw,circle,fill=red!50] (c) {$3$}
;}\newcommand{\noded}{\node[draw,circle,fill=white] (d) {$$}
;}\newcommand{\nodee}{\node[draw,circle,fill=white] (e) {$1$}
;}\newcommand{\nodef}{\node[draw,circle,fill=white] (f) {$$}
;}\newcommand{\nodeg}{\node[draw,circle,fill=white] (g) {$4$}
;}\newcommand{\nodeh}{\node[draw,circle,fill=red!50] (h) {$2$}
;}\newcommand{\nodei}{\node[draw,circle,fill=white] (i) {$$}
;}\newcommand{\nodej}{\node[draw,circle,fill=white] (j) {$$}
;}\newcommand{\nodeba}{\node[draw,circle,fill=white] (ba) {$$}
;}\begin{tikzpicture}[baseline]
\matrix[column sep=0.1cm, row sep=.3cm,ampersand replacement=\&]{
         \&         \& \nodea  \&         \&         \&         \\ 
 \nodeb  \& \nodec  \& \nodee  \&         \& \nodeg  \&         \\ 
         \& \noded  \& \nodef  \& \nodeh  \& \nodej  \& \nodeba \\ 
         \&         \&         \& \nodei  \&         \&         \\
};

\path[ultra thick, red] (c) edge (d)
	(e) edge (f)
	(h) edge (i)
	(g) edge (h) edge (j) edge (ba)
	(a) edge (b) edge (c) edge (e) edge (g);
\end{tikzpicture}}
 & $\rightarrow$ & 
{ \newcommand{\nodea}{\node[draw,circle,fill=white] (a) {$4$}
;}\newcommand{\nodeb}{\node[draw,circle,fill=white] (b) {$$}
;}\newcommand{\nodec}{\node[draw,circle,fill=red!50] (c) {$3$}
;}\newcommand{\noded}{\node[draw,circle,fill=white] (d) {$$}
;}\newcommand{\nodee}{\node[draw,circle,fill=white] (e) {$1$}
;}\newcommand{\nodef}{\node[draw,circle,fill=white] (f) {$$}
;}\newcommand{\nodeg}{\node[draw,circle,fill=red!50] (g) {$2$}
;}\newcommand{\nodeh}{\node[draw,circle,fill=white] (h) {$$}
;}\newcommand{\nodei}{\node[draw,circle,fill=white] (i) {$$}
;}\newcommand{\nodej}{\node[draw,circle,fill=white] (j) {$$}
;}\begin{tikzpicture}[baseline]
\matrix[column sep=0.1cm, row sep=.7cm,ampersand replacement=\&]{
         \&         \&         \& \nodea  \&         \&         \\ 
 \nodeb  \& \nodec  \& \nodee  \& \nodeg  \& \nodei  \& \nodej  \\ 
         \& \noded  \& \nodef  \& \nodeh  \&         \&         \\
};

\path[ultra thick, red] (c) edge (d)
	(e) edge (f)
	(g) edge (h)
	(a) edge (b) edge (c) edge (e);

\path[ultra thick, dashed, red] (a) edge (g) edge (i) edge (j);
\end{tikzpicture}}
\end{tabular}}
\caption{Bijective illustration of Case~\eqref{prop-case:f-poly-recur-1} of the recursive $f$-polynomial computation.}
\label{fig:fpoly-recurs1}
\end{figure}
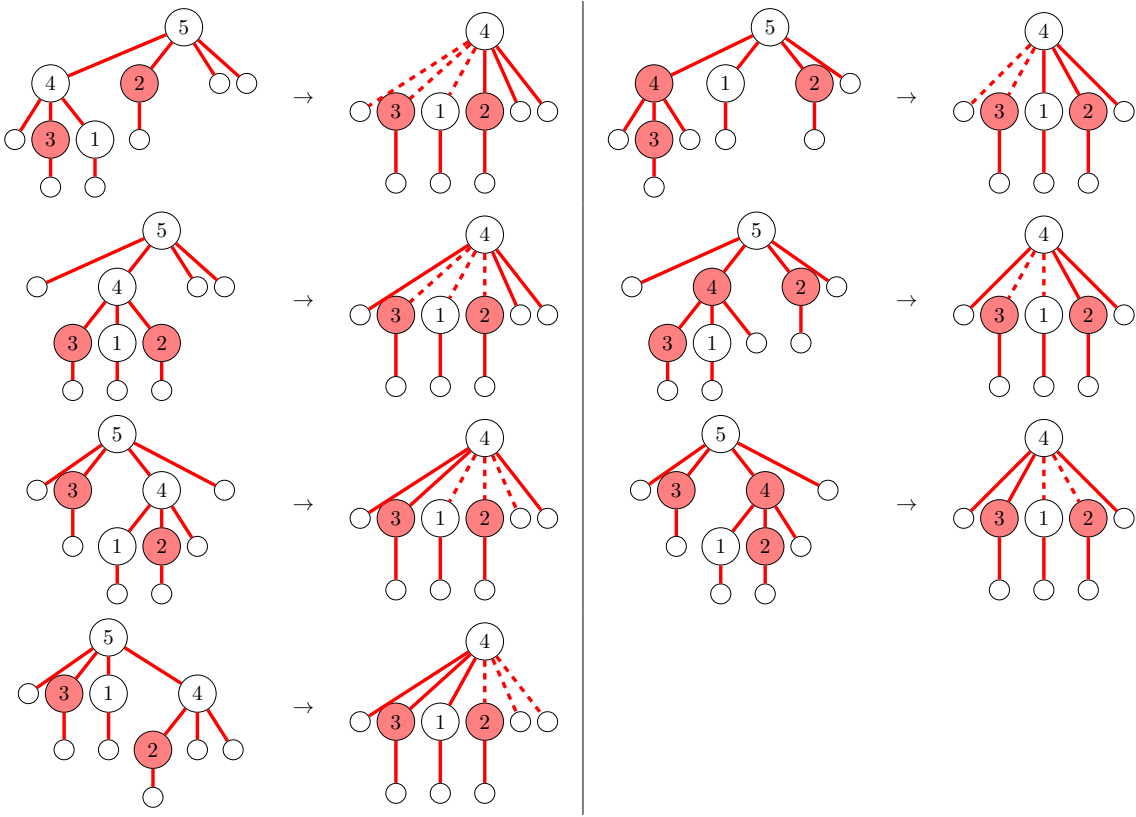

\begin{table}[t]
\caption{Some values of $f_s(t)$}
\scalebox{.8}{
\begin{tabular}{l|l|l}
$f_{(0,0, 1)}(t) = t^{2} + 4 t + 4 $ & $f_{(0,0, 0, 1)}(t) = t^{3} + 6 t^{2} + 12 t + 8 $ & $f_{(0,2, 0, 1)}(t) = 2 t^{3} + 13 t^{2} + 26 t + 16 $ \\
$f_{(0,0, 2)}(t) = 4 t^{2} + 12 t + 9 $ & $f_{(0,0, 0, 2)}(t) = 8 t^{3} + 36 t^{2} + 54 t + 27 $ & $f_{(0,2, 0, 2)}(t) = 12 t^{3} + 58 t^{2} + 90 t + 45 $ \\
$f_{(0,0, 3)}(t) = 9 t^{2} + 24 t + 16 $ & $f_{(0,0, 0, 3)}(t) = 27 t^{3} + 108 t^{2} + 144 t + 64 $ & $f_{(0,2, 0, 3)}(t) = 36 t^{3} + 153 t^{2} + 212 t + 96 $ \\
$f_{(0,1, 1)}(t) = t^{2} + 6 t + 6 $ & $f_{(0,0, 1, 1)}(t) = t^{3} + 12 t^{2} + 28 t + 18 $ & $f_{(0,2, 1, 1)}(t) = 2 t^{3} + 21 t^{2} + 48 t + 30 $ \\
$f_{(0,1, 2)}(t) = 4 t^{2} + 15 t + 12 $ & $f_{(0,0, 1, 2)}(t) = 8 t^{3} + 51 t^{2} + 90 t + 48 $ & $f_{(0,2, 1, 2)}(t) = 12 t^{3} + 76 t^{2} + 135 t + 72 $ \\
$f_{(0,1, 3)}(t) = 9 t^{2} + 28 t + 20 $ & $f_{(0,0, 1, 3)}(t) = 27 t^{3} + 136 t^{2} + 208 t + 100 $ & $f_{(0,2, 1, 3)}(t) = 36 t^{3} + 185 t^{2} + 288 t + 140 $ \\
$f_{(0,2, 1)}(t) = 2 t^{2} + 9 t + 8 $ & $f_{(0,0, 2, 1)}(t) = 4 t^{3} + 30 t^{2} + 57 t + 32 $ & $f_{(0,2, 2, 1)}(t) = 6 t^{3} + 44 t^{2} + 85 t + 48 $ \\
$f_{(0,2, 2)}(t) = 6 t^{2} + 20 t + 15 $ & $f_{(0,0, 2, 2)}(t) = 18 t^{3} + 96 t^{2} + 152 t + 75 $ & $f_{(0,2, 2, 2)}(t) = 24 t^{3} + 130 t^{2} + 210 t + 105 $ \\
$f_{(0,2, 3)}(t) = 12 t^{2} + 35 t + 24 $ & $f_{(0,0, 2, 3)}(t) = 48 t^{3} + 220 t^{2} + 315 t + 144 $ & $f_{(0,2, 2, 3)}(t) = 60 t^{3} + 282 t^{2} + 413 t + 192 $ \\
$f_{(0,3, 1)}(t) = 3 t^{2} + 12 t + 10 $ & $f_{(0,0, 3, 1)}(t) = 9 t^{3} + 56 t^{2} + 96 t + 50 $ & $f_{(0,2, 3, 1)}(t) = 12 t^{3} + 75 t^{2} + 132 t + 70 $ \\
$f_{(0,3, 2)}(t) = 8 t^{2} + 25 t + 18 $ & $f_{(0,0, 3, 2)}(t) = 32 t^{3} + 155 t^{2} + 230 t + 108 $ & $f_{(0,2, 3, 2)}(t) = 40 t^{3} + 198 t^{2} + 301 t + 144 $ \\
$f_{(0,3, 3)}(t) = 15 t^{2} + 42 t + 28 $ & $f_{(0,0, 3, 3)}(t) = 75 t^{3} + 324 t^{2} + 444 t + 196 $ & $f_{(0,2, 3, 3)}(t) = 90 t^{3} + 399 t^{2} + 560 t + 252 $ \\
 & $f_{(0,1, 0, 1)}(t) = t^{3} + 8 t^{2} + 18 t + 12 $ & $f_{(0,3, 0, 1)}(t) = 3 t^{3} + 18 t^{2} + 34 t + 20 $ \\
 & $f_{(0,1, 0, 2)}(t) = 8 t^{3} + 42 t^{2} + 69 t + 36 $ & $f_{(0,3, 0, 2)}(t) = 16 t^{3} + 74 t^{2} + 111 t + 54 $ \\
 & $f_{(0,1, 0, 3)}(t) = 27 t^{3} + 120 t^{2} + 172 t + 80 $ & $f_{(0,3, 0, 3)}(t) = 45 t^{3} + 186 t^{2} + 252 t + 112 $ \\
 & $f_{(0,1, 1, 1)}(t) = t^{3} + 14 t^{2} + 36 t + 24 $ & $f_{(0,3, 1, 1)}(t) = 3 t^{3} + 28 t^{2} + 60 t + 36 $ \\
 & $f_{(0,1, 1, 2)}(t) = 8 t^{3} + 57 t^{2} + 108 t + 60 $ & $f_{(0,3, 1, 2)}(t) = 16 t^{3} + 95 t^{2} + 162 t + 84 $ \\
 & $f_{(0,1, 1, 3)}(t) = 27 t^{3} + 148 t^{2} + 240 t + 120 $ & $f_{(0,3, 1, 3)}(t) = 45 t^{3} + 222 t^{2} + 336 t + 160 $ \\
 & $f_{(0,1, 2, 1)}(t) = 4 t^{3} + 33 t^{2} + 68 t + 40 $ & $f_{(0,3, 2, 1)}(t) = 8 t^{3} + 55 t^{2} + 102 t + 56 $ \\
 & $f_{(0,1, 2, 2)}(t) = 18 t^{3} + 104 t^{2} + 175 t + 90 $ & $f_{(0,3, 2, 2)}(t) = 30 t^{3} + 156 t^{2} + 245 t + 120 $ \\
 & $f_{(0,1, 2, 3)}(t) = 48 t^{3} + 235 t^{2} + 354 t + 168 $ & $f_{(0,3, 2, 3)}(t) = 72 t^{3} + 329 t^{2} + 472 t + 216 $ \\
 & $f_{(0,1, 3, 1)}(t) = 9 t^{3} + 60 t^{2} + 110 t + 60 $ & $f_{(0,3, 3, 1)}(t) = 15 t^{3} + 90 t^{2} + 154 t + 80 $ \\
 & $f_{(0,1, 3, 2)}(t) = 32 t^{3} + 165 t^{2} + 258 t + 126 $ & $f_{(0,3, 3, 2)}(t) = 48 t^{3} + 231 t^{2} + 344 t + 162 $ \\
 & $f_{(0,1, 3, 3)}(t) = 75 t^{3} + 342 t^{2} + 490 t + 224 $ & $f_{(0,3, 3, 3)}(t) = 105 t^{3} + 456 t^{2} + 630 t + 280 $ \\ 
\multicolumn{3}{l}{}  \\
\multicolumn{3}{l}{$f_{(0, 4, 4, 4)}(t) = 280 t^{3} + 1100 t^{2} + 1404 t + 585 $} \\
\multicolumn{3}{l}{$f_{(0, 5, 5, 5)}(t) = 585 t^{3} + 2170 t^{2} + 2640 t + 1056 $} \\
\multicolumn{3}{l}{$f_{(0, 6, 6, 6)}(t) = 1056 t^{3} + 3774 t^{2} + 4446 t + 1729 $} \\
\multicolumn{3}{l}{$f_{(0, 2, 2, 2, 2)}(t) = 120 t^{4} + 924 t^{3} + 2380 t^{2} + 2520 t + 945 $} \\
\multicolumn{3}{l}{$f_{(0, 3, 3, 3, 3)}(t) = 945 t^{4} + 5610 t^{3} + 11946 t^{2} + 10920 t + 3640 $} \\
\multicolumn{3}{l}{$f_{(0, 2, 1, 0, 2)}(t) = 24 t^{4} + 188 t^{3} + 498 t^{2} + 549 t + 216 $} \\
\multicolumn{3}{l}{$f_{(0, 2, 0, 1, 3, 2)}(t) = 640 t^{5} + 6278 t^{4} + 21941 t^{3} + 35474 t^{2} + 27108 t + 7938 $} \\
\multicolumn{3}{l}{$f_{(0, 1, 1, 2, 0, 3, 3)}(t) = 16200 t^{6} + 166377 t^{5} + 665262 t^{4} + 1349600 t^{3} + 1481796 t^{2} + 841320 t + 194040 $} \\
\end{tabular}
}
\label{tab:fs}
\end{table}

Let us prove Case~\eqref{prop-case:f-poly-recur-1}. Note that if $s$ is of length $2$, we have shown earlier that the recursive computation gives the expected result. We then assume that $s = (\tilde{s}, u, v)$ and $u \neq 0$. In this case, we prove that there is a bijection between pure intervals of $s$ and some marked pure intervals of $(\tilde{s}, u+v)$ and $(\tilde{s}, u + v - 1)$. Let $T$ be an $s$-decreasing tree and $A$ a subset of tree-ascents of $T$. Let us first suppose that $(n-1, n) \not\in A$. In this case, $(T,A)$ is sent to a marked pure interval of $s' = (\tilde{s}, u+v)$. We construct $T'$ by \emph{merging} the nodes $n$ and $n-1$. In $T$, $n-1$ has $u+1$ subtrees which become subtrees of the root in $T'$, at the former position of $n-1$. The root is relabeled $n-1$ in $T'$ and has now $u+v+1$ subtrees (it has lost one subtree and gained $u+1$). Any tree-ascent of $T$, $(a,c)$ with $a < c < n-1$ is still a tree-ascent in $T'$. Similarly, if $(a,n-1)$ was a tree-ascent in $T$, it is also a tree-ascent in $T'$ (except that now $n-1$ is the root). Finally, if $(a,n)$ was a tree-ascent in $T$ and $a \neq n-1$, then $(a,n-1)$ is a tree-ascent in $T'$ (whether $a$ was a descendant of $n-1$ or not). As we have by hypothesis that $(n-1, n) \not\in A$, we can \emph{keep} the selected tree-ascents of $T$ in $T'$. We get $A' = \lbrace (a,c) \in A; a < c \leq n-1 \rbrace \cup \lbrace (a,n-1); (a,n) \in A \rbrace$ with $|A'| = |A|$. 

This operation is surjective. To obtain a bijection, we need to mark $u+1$ consecutive edges from the root of $T'$: they give the subtrees and insertion position of the node $n-1$ in $T$. All tree-ascents of $T'$ are still tree-ascents in $T$: a tree-ascent $(a,n-1)$ in $T'$ corresponds either to $(a,n-1)$ in $T$ if $a$ belongs to first $u$ marked edges, otherwise it transforms into a $(a,n)$ tree-ascent. We thus obtain all pure intervals of $s$ where $(n-1, n) \not\in A$. As there are $v+1$ way to mark the edges of an $s'$-decreasing tree, we obtain that $f_s(t) = (v+1)f_{s'}(t) + f_s''$ where $f_s''$ enumerates the pure intervals of $s$ where $(n-1, n)$ is in $A$. This bijection is illustrated for $s = (0, 0, 0, 2, 3)$ in the first column of Figure~\ref{fig:fpoly-recurs1}: we show the $4$ possible markings on a given pure interval of $(0,0,0,5)$ and their pre-images. 

If $(n-1,n) \in A$, we send $(T,A)$ to a pure interval of $s'' = (\tilde{s}, u+v-1)$. As $(n-1,n)$ is a tree-ascent, we know that the last subtree of $n-1$ in $T$ is empty (this is true because $u \neq 0$). We then construct~$T''$ by \emph{deleting} this last subtree and then \emph{merging} the nodes $n-1$ and $n$. The new root in $T''$ now has $u + v$ subtrees. The same arguments as earlier work for tree-ascents different than $(n-1,n)$ and we can define $A'' = \lbrace (a,c) \in A; a < c \leq n-1 \rbrace \cup \lbrace (a,n-1); (a,n) \in A \text{ and } a \neq n-1 \rbrace$. As we \emph{loose} the tree-ascent $(n-1, n)$, we have $|A''| = |A| - 1$. Besides, to obtain a bijection we now need to mark $u$ consecutive edges of the root but we cannot mark the last $u$ edges. Indeed, we need $(n-1, n)$ to be a tree-ascent of $T$ which implies that $n-1$ can never be in the last subtree of $n$. There are $v$ possible ways to mark the edges of an $s''$-decreasing tree, we obtain $f_s'' = vt f_{s''}$. We illustrate the bijection for $s=(0,0,0,2,3)$ on the second column of Figure~\ref{fig:fpoly-recurs1}: we show the $3$ possible markings on given pure interval of $(0,0,0,4)$ and their pre-images.
\end{proof}

%Looking at different examples, the collection of pure intervals of the $s$-weak order seems to have a nice geometric structure. We conjecture that it is always a polytopal complex, see Part~\ref{part_polytopal_conjectures}.
%%, meaning that each pure interval can be realized as a polytope,  whose faces are ``glued'' together. 
%The main result of the Part~\ref{part_one} of this paper is the following.

Looking at many different examples, we believe that the $s$-permutahedron can be realized as a polytopal complex for any weak composition $s$ (see Part~\ref{part_polytopal_conjectures}). Two essential conditions for this to hold are listed in the following theorem, which is the main result of Part~\ref{part_one} of this paper.

\begin{theorem}\label{sperm_combinatorial_complex}
For any weak composition $s$, the $s$-permutahedron $\Perm{s}$ is a combinatorial complex in the following sense. 
\begin{enumerate}
\item The face of a face is also a face: \\
If $I\in \Perm{s}$ is a face (a pure interval) and $J\subseteq I$ is a face (a pure interval) then $J\in \Perm{s}$.
\item The intersection of any two faces is also a face:\\
If $I,J\in \Perm{s}$ are two faces (pure intervals) then $I\cap J\in \Perm{s}$ is also a face (a pure interval).
\end{enumerate}
\end{theorem}

Part (1) of this theorem follows by definition. The intersection property in Part (2) is more involved and we need to develop some tools in order to prove it. First, we will give a complete characterization of pure intervals in Section~\ref{sec_characterization_pure_intervals} (Theorem~\ref{thm:pure-interval-char}), which will then be used to prove the intersection property in Section~\ref{sec_intersection_pure_intervals} (Theorem~\ref{thm:intersection}).

%%%%%%%%%%%%%%%%%%%%%
%%%%%%%%%%%%%%%%%%%%%
%%%%%%%%%%%%%%%%%%%%%
%%%%%%%%%%%%%%%%%%%%%
%%%%%%%%%%%%%%%%%%%%%
%%%%%%%%%%%%%%%%%%%%%

\section{Characterization of pure intervals}\label{sec_characterization_pure_intervals}
Although the definition of pure intervals is very simple and intuitive, we now aim to provide a characterization which is rather technical but conveniently useful for proofs (\Cref{thm:pure-interval-char}).
In order to do this, we need to introduce the notions of variations, essential variations, and minimal essential variations of an interval. 
Roughly speaking, the variations are the inversions that increase in the interval, while the minimal essential variations of a pure interval $[T_1,T_1+A]$ will characterize the tree-ascents $A$ (\Cref{thm_min_essential_variations_ascents}).   

\subsection{Variations and essential variations of an interval}
\begin{definition}[Variations]
\label{def:variations}
Let $T_1 \wole T_2$ be two $s$-decreasing trees for a given weak composition $s$. We say that a tree-inversion $(c,a)$ with $c > a$ \defn{varies} in the interval $[T_1,T_2]$ if $\card_{T_2}(c,a) > \card_{T_1}(c,a)$. In this case, we say that $(c,a)_v$ where $v = \card_{T_1}(c,a)$ is a \defn{variation} of $[T_1,T_2]$. We call $v$ the \defn{value} of the variation. The difference $\card_{T_2}(c,a) - \card_{T_1}(c,a)$ is the \defn{amplitude}. We sometimes omit the value $v$ and write that $(c,a)$ is a variation, meaning that there exists $v$ for which $(c,a)_v$ is a variation.

We say that $(c,a)_v$ is an \defn{essential variation} of $[T_1,T_2]$ if $(c,a)_v$ is a variation and there is no $b$ with $a < b < c$ such that $a$ belongs to a middle child of $b$ and $(c,b)$ varies. Finally, we say that an essential variation $(c,a)$ is \defn{minimal} if there is no $b$ with $a < b < c$ with $(b,a)$ an essential variation.

We write $\Var([T_1,T_2])$ the set of variations of the interval, and $\EVar(T_1,T_2)$ the set of essential variations.
\end{definition}

We present below two examples based on Figures~\ref{fig:pure-interval} and~\ref{fig:pure-interval10}. This can be computed using {\tt SageMath} as we show in~\cite{SageDemoII}.

\begin{example}\label{ex_variationsone}
On the pure interval of Figure~\ref{fig:pure-interval}, the variations with their respective values are $(3,1)_1$, $(3,2)_1$, $(4,1)_0$, $(4,2)_0$, $(4,3)_0$. They are all of amplitude $1$. Only $(3,2)_1$ and $(4,3)_0$ are essential variations. For example, $(4,2)_0$ is not an essential variation because $2$ belongs to the middle child of $3$ and $(4,3)$ varies. In this case, both essential variations are minimal and correspond to the tree-ascents in $A$ of the minimal tree. 
\end{example}

\begin{example}\label{ex_variationstwo}
The pure interval of Figure~\ref{fig:pure-interval10} has 15 variations, 10 of which are essential variations: $(4, 1)_0$, $(6, 2)_1$, $(9, 1)_0$, $(9, 4)_0$, $(9, 6)_1$, $(10, 1)_2$, $(10, 4)_2$, $(10, 5)_0$, $(10, 8)_0$, $(10, 9)_2$. For example, $(10,2)_2$ is a variation but not an essential variation. Besides, there are seven minimal essential variations: $(4, 1)_0$, $(6, 2)_1$, $(9, 4)_0$, $(9, 6)_1$, $(10, 5)_0$, $(10, 8)_0$, $(10, 9)_2$. For example, $(9, 1)_0$ is not minimal because $(4,1)_0$ is an essential variation. 
The seven minimal essential variations in this case also correspond to the tree-ascents in $A$ of the minimal tree. 
\end{example}

In \Cref{ex_variationsone,ex_variationstwo}, the minimal essential variations of both pure intervals $[T_1,T_1+A]$ correspond to the tree-ascents in $A$. This property holds in general and will be proved later on. 

\begin{theorem}[Minimal essential variations of pure intervals]
\label{thm_min_essential_variations_ascents}
The set of minimal essential variations of a pure interval $[T_1,T_1+A]$ is in correspondence with $A$.
More precisely, $(a,c)\in A$ if and only if there exists $v$ for which $(c,a)_v$ is a minimal essential variation.
\end{theorem}

\subsection{The Characterization Theorem}
In order to characterize the variation sets of pure intervals we need some further properties. 

\begin{lemma}
\label{lem:essential-var-middle-child}
For all $a < c$ such that $(c,a)_v$ is a variation but not an essential variation, there is a unique $b > a$ such that $(c,b)_v$ is an essential variation and $a$ belongs to a middle child of $b$.
\end{lemma}

\begin{proof}
Let us suppose that $(c,a)_v$ is a variation but not an essential variation. We know there is $b$ such that $a$ belongs to a middle child of $b$ and $(c,b)$ varies. We take such a $b$ to be maximal. Note that as $a$ is a descendant of $b$, then $\card_{T_1}(c,b) = \card_{T_1}(c,a) = v$. As $b$ is maximal, there is no $b' > b$ such that $(c,b')$ varies and $a$ belongs to a middle child of $b'$. This means that there is no $b' > b$ such that $(c,b')$ varies and $b$ belongs to a middle child of $b'$ (because $a$ is a descendant of $b$). Therefore, $(c,b)$ is an essential variation.

Now, let us suppose we have $c > b_1 > b_2 > a$ such that $(c,b_1)_v$ is an essential variation and $a$ belongs to a middle child of $b_1$. If $a$ is a descendant of $b_2$, then $b_2$ belongs to the same middle child of $b_1$ as $a$.
Therefore, $(c,b_2)$ is not an essential variation.
This proves the uniqueness part.
\end{proof}

For example, there are three non-essential variations on the pure interval of Figure~\ref{fig:pure-interval}. For $(3,1)_1$ the unique essential variation is $(3,2)_1$. For both $(4,2)_0$ and $(4,1)_0$, the unique essential variation is $(4,3)_0$. 

\begin{lemma}[Transitivity]
\label{lem:trans}
Let $[T_1, T_2]$ be an interval. Suppose that we have $c > b > a$ with $(c,b)$ and $(b,a)$ variations. Then $(c,a)$ also varies.
\end{lemma}

\begin{proof}
This is immediate. Indeed, we have by hypothesis $\card_{T_1}(b,a) < s(b)$, we obtain by planarity (Definition~\ref{def:multi-sets}) that $\card_{T_1}(c,a) \leq \card_{T_1}(c,b)$. As both $(c,b)$ and $(b,a)$ vary, we obtain $\card_{T_2}(c,b) > \card_{T_1}(c,b)$ and $\card_{T_2}(b,a) > \card_{T_1}(b,a) \geq 0$ which gives $\card_{T_2}(c,a) \geq \card_{T_2}(c,b) > \card_{T_1}(c,b) \geq \card_{T_1}(c,a)$ by transitivity.
\end{proof}

\begin{definition}[$+1$-interval]
\label{def:plus-one-interval}
An interval $[T_1,T_2]$ is said to be a \defn{$+1$-interval} if all variations have amplitude one, \emph{i.e.}, for all $a < b$ we have
\begin{equation}
\card_{T_1}(b,a) \leq \card_{T_2}(b,a) \leq \card_{T_1}(b,a) + 1.
\end{equation}
\end{definition}

\begin{lemma}
\label{lem:plus-one}
Every pure interval is a $+1$-interval.
\end{lemma}

\begin{proof}
Let $[T,T+A]$ be a pure interval and take $b > a$. Suppose that $\card_T(b,a) = v$ and $\card_{T+A}(b,a) = v + k$ with $k > 1$. Let $S$ be the multi-set of inversions such that $\card_S(c,a) = \card_T(c,a) + 1$ if $(a,c) \in A$ and otherwise $\card_S(c,a) = \card_T(c,a)$. In particular, $\tc{S}$ is the tree-inversion set of $T+A$. There is a transitivity path in $S$ 
\begin{equation*}
b = b_1 > b_2 > \dots > b_k = a
\end{equation*}
with $\card_S(b_1, b_2) = v + k$ and $\card_S(b_i, b_{i+1}) > 0$. This path does not exist as such in $\inv(T)$ because $\inv(T)$ is transitive and $\card_T(b,a) = v$. Besides, by definition of $S$, we have $\card_T(b_i,b_{i+1}) \leq \card_S(b_i,b_{i+1}) \leq \card_T(b_i,b_{i+1}) +1$. This gives us in particular
\begin{equation}
\card_T(b_1, b_2) \geq v + k - 1 > v
\end{equation}
because $k > 1$ by hypothesis. This implies that there exist $i$ such that $\card_T(b_i, b_{i+1}) = 0$. We take the minimal $i$ that satisfies this property ($b_i$ is as close of $b$ as possible). In particular, by transitivity, $\card_T(b,b_i) \geq v + k - 1 > v$. By definition of $S$, we have that $(b_{i+1}, b_i)$ is a tree-ascent of $T$ which belongs to~$A$. By Condition \eqref{cond:tree-ascent-desc} of Definition~\ref{def:tree-ascent}, it means that $b_{i+1}$ is a descendant of $b_i$. In particular, they belong to the same subtree of $b$ and we have $\card_T(b,b_{i+1}) = \card_T(b,b_i) > v$.

Now, if $j > i$ with $\card_T(b_j, b_{j+1}) = 0$, by taking the minimal $j$ satisfying this property, we prove in a similar way that $\card_T(b,b_{j+1}) > v$. By induction, we obtain that for all $i$ with  $\card_T(b_i, b_{i+1}) = 0$, we have $\card_T(b,b_{i+1}) > v$. Let $i'$ be the maximal value satisfying the property. We have that

\begin{equation*}
b = b_1 > b_{i'} > b_{i'+1} > \dots > b_k = a
\end{equation*}

is a transitive path in $\inv(T)$, which implies that $\card_T(b,a) > v$ and leads to a contradiction. 
\end{proof}

Lemma~\ref{lem:plus-one} is not an equivalence; see Figure~\ref{fig:not-pure-interval} for an example of a $+1$-interval that is not a pure interval. The main result of this section is the following characterization of pure intervals.

\begin{theorem}[Characterization Theorem of pure intervals]
\label{thm:pure-interval-char}
Let $T_1$ and $T_2$ be two $s$-decreasing trees for a given weak composition $s$, then $[T_1,T_2]$ is a pure interval if and only if it is a $+1$-interval and for all $a < b < c$
\begin{enumerate}
\item if $(c,a)_v \in \Var([T_1,T_2])$ and $(b,a)_w \in \Var([T_1,T_2])$ then $(c,b)_v \in \Var([T_1,T_2])$;
\label{cond:pure-interval-cb}
\item if $(c,a)_v \in \EVar([T_1,T_2])$ and $(c,b)_v \in \EVar([T_1,T_2])$ and $s(b) \neq 0$ then $(b,a)_0 \in \Var([T_1,T_2])$.
\label{cond:pure-interval-ba}
\end{enumerate}
%{\color{red}
%Then $T_2 = T_1 + A$ where $A$ is the set of couples $(a,b)$ such that $(b,a)$ is a \emph{minimal essential variation}. 
%}\cesar{put this sentence in a separate result?}
\end{theorem}

Note that for Condition~\eqref{cond:pure-interval-ba}, we require that $\card_{T_1}(c,a) = \card_{T_1}(c,b)$.

\begin{example}
You can check that the two conditions of \Cref{thm:pure-interval-char} are satisfied on the interval of Figure~\ref{fig:pure-interval10}. For~\eqref{cond:pure-interval-cb}, you can see for example that both $(10,4)_2$ and $(9,4)_0$ are variations and indeed $(10,9)_2$ is also a variation. Similarly, $(9,1)_0$ and $(4,1)_0$ are essential variations, and $(9,4)_0$ is a variation.

Besides \eqref{cond:pure-interval-ba} implies that as $(10,4)_0$ and $(10,9)_0$ are essential variations and $s(9) > 0$, then $(9,4)_0$ is a variation. Similarly, as $(9,1)_0$ and $(9,4)_0$ are essential variations and $s(4) > 0$, this implies that $(4,1)_0$ is a variation. Note that we also have that $(10,8)_0$ and $(10,5)_0$ are essential variations, but in this case $(8,5)_0$ is not because $s(8) = 0$. Finally, you can see that~\eqref{cond:pure-interval-ba} is not satisfied on variations: $(10,4)_2$ and $(10,3)_2$ are variations and $s(4) > 0$ but $(4,3)_0$ is not a variation. 
\end{example}

\begin{example}
On the other hand, Figure~\ref{fig:not-pure-interval} shows an interval, \emph{i.e.}, two trees $T_1 \wole T_2$, such that $[T_1, T_2]$ is a $+1$-interval but not a pure interval. In other words, $T_2$ is not obtained by increasing the cardinality of a subset of the tree-ascents of $T_1$. In this example, both conditions of \Cref{thm:pure-interval-char} are not satisfied. Indeed, for~\eqref{cond:pure-interval-cb}, $(4,1)_0$ and $(3,1)_1$ are variations but $(4,3)_0$ is not. For~\eqref{cond:pure-interval-ba}, $(5,4)_0$ and $(5,3)_0$ are essential variations and $s(4) > 0$ but $(4,3)_0$ is not a variation.
\end{example}

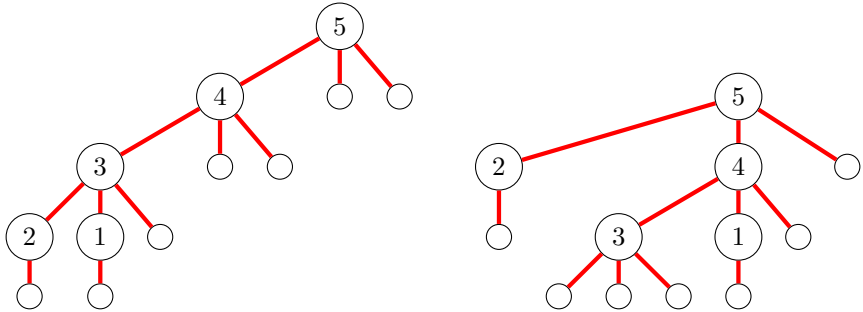
\begin{figure}[ht]
\begin{tabular}{cc}
{ \newcommand{\nodea}{\node[draw,circle] (a) {$5$}
;}\newcommand{\nodeb}{\node[draw,circle] (b) {$4$}
;}\newcommand{\nodec}{\node[draw,circle] (c) {$3$}
;}\newcommand{\noded}{\node[draw,circle] (d) {$2$}
;}\newcommand{\nodee}{\node[draw,circle] (e) {$$}
;}\newcommand{\nodef}{\node[draw,circle] (f) {$1$}
;}\newcommand{\nodeg}{\node[draw,circle] (g) {$$}
;}\newcommand{\nodeh}{\node[draw,circle] (h) {$$}
;}\newcommand{\nodei}{\node[draw,circle] (i) {$$}
;}\newcommand{\nodej}{\node[draw,circle] (j) {$$}
;}\newcommand{\nodeba}{\node[draw,circle] (ba) {$$}
;}\newcommand{\nodebb}{\node[draw,circle] (bb) {$$}
;}\begin{tikzpicture}[auto]
\matrix[column sep=.3cm, row sep=.3cm,ampersand replacement=\&]{
         \&         \&         \&         \&         \& \nodea  \&         \\ 
         \&         \&         \& \nodeb  \&         \& \nodeba \& \nodebb \\ 
         \& \nodec  \&         \& \nodei  \& \nodej  \&         \&         \\ 
 \noded  \& \nodef  \& \nodeh  \&         \&         \&         \&         \\ 
 \nodee  \& \nodeg  \&         \&         \&         \&         \&         \\
};

\path[ultra thick, red] (d) edge (e)
	(f) edge (g)
	(c) edge (d) edge (f) edge (h)
	(b) edge (c) edge (i) edge (j)
	(a) edge (b) edge (ba) edge (bb);
\end{tikzpicture}} &
{ \newcommand{\nodea}{\node[draw,circle] (a) {$5$}
;}\newcommand{\nodeb}{\node[draw,circle] (b) {$2$}
;}\newcommand{\nodec}{\node[draw,circle] (c) {$$}
;}\newcommand{\noded}{\node[draw,circle] (d) {$4$}
;}\newcommand{\nodee}{\node[draw,circle] (e) {$3$}
;}\newcommand{\nodef}{\node[draw,circle] (f) {$$}
;}\newcommand{\nodeg}{\node[draw,circle] (g) {$$}
;}\newcommand{\nodeh}{\node[draw,circle] (h) {$$}
;}\newcommand{\nodei}{\node[draw,circle] (i) {$1$}
;}\newcommand{\nodej}{\node[draw,circle] (j) {$$}
;}\newcommand{\nodeba}{\node[draw,circle] (ba) {$$}
;}\newcommand{\nodebb}{\node[draw,circle] (bb) {$$}
;}\begin{tikzpicture}[auto]
\matrix[column sep=.3cm, row sep=.3cm,ampersand replacement=\&]{
         \&         \&         \&         \& \nodea  \&         \&         \\ 
 \nodeb  \&         \&         \&         \& \noded  \&         \& \nodebb \\ 
 \nodec  \&         \& \nodee  \&         \& \nodei  \& \nodeba \&         \\ 
         \& \nodef  \& \nodeg  \& \nodeh  \& \nodej  \&         \&         \\
};

\path[ultra thick, red] (b) edge (c)
	(e) edge (f) edge (g) edge (h)
	(i) edge (j)
	(d) edge (e) edge (i) edge (ba)
	(a) edge (b) edge (d) edge (bb);
\end{tikzpicture}}
\end{tabular}
\caption{A $+1$-interval that is not a pure interval.}
\label{fig:not-pure-interval}
\end{figure}

%\subsection{Proof of the Characterization Theorem}

In the following two sections we will focus on the proof of the Characterization Theorem of pure intervals, \Cref{thm:pure-interval-char}. The techniques used to prove this result will also lead to a proof of \Cref{thm_min_essential_variations_ascents}, which characterizes the subset of tree-ascents of a pure interval in terms of its minimal essential variations.
%
%\label{cond:pure-interval-cb}
%\item if $(c,a)_v \in \EVar([T_1,T_2])$ and $(c,b)_v \in \EVar([T_1,T_2])$ and $s(b) \neq 0$ then $(b,a)_0 \in \Var([T_1,T_2])$.
%\label{cond:pure-interval-ba}

For the sake of clarity, we are going to give a name to the $+1$-intervals satisfying both Conditions~\eqref{cond:pure-interval-cb} and~\eqref{cond:pure-interval-ba} of \Cref{thm:pure-interval-char}. We call them \defn{pure-candidate intervals}. Our goal is to prove that pure-candidate intervals are the same as pure intervals. We separate this result into two parts: in \Cref{sec_purecandidate_pure} we show that pure-candidate intervals are pure, and in \Cref{sec_pure_purecandidate} we show that pure intervals are pure-candidate. 

\subsection{Proof Part 1: Pure-candidate intervals are pure intervals}\label{sec_purecandidate_pure}
\begin{proposition}
\label{prop:pure-candidate-to-pure}
If $[T_1,T_2]$ is a pure candidate interval, then it is a pure interval. More precisely, we have that
$[T_1,T_2] = [T_1, T_1+A]$ where $(a,c) \in A$ if and only if $(c,a)$ is a minimal essential variation of $[T_1, T_2]$.
\end{proposition}

We start by examining Condition~\eqref{cond:pure-interval-cb} and show some useful properties.

% Before proving \Cref{thm:pure-interval-char}, we first examine Condition~\ref{cond:pure-interval-cb} and show that it is equivalent to a condition on essential variations. We use the following lemma.

\begin{lemma}
\label{lem:var-desc}
Let $[T_1,T_2]$ be $+1$-interval satisfying Condition~\eqref{cond:pure-interval-cb}. We have that for all $a < c$, if $(c,a)$ varies then $a$ is a descendant of $c$ in $T_1$.
\end{lemma}

\begin{proof}
Suppose that it is not the case, in particular, $a$ cannot be a middle descendant of $c$ and so either $\card_{T_1}(c,a) = 0$ or $\card_{T_1}(c,a) = s(c)$. The second option is not possible because $(c,a)$ varies. So $\card_{T_1}(c,a) = 0$ and there exists $d > c$ such that $\card_{T_1}(d,a) < \card_{T_1}(d,c)$. In $T_2$, we have $\card_{T_2}(c,a) = \card_{T_1}(c,a) + 1 = 1$, which gives by transitivity that $\card_{T_2}(d,a) \geq \card_{T_2}(d,c) \geq \card_{T_1}(d,c) > \card_{T_1}(d,a)$. In particular, $(d,a)$ varies. We have $d > c > a$ with variations $(d,a)$ and $(c,a)$, using Condition~\eqref{cond:pure-interval-cb} of pure-candidate intervals, we obtain that $(d,c)$ also varies. But if $\card_{T_2}(d,c) > \card_{T_1}(d,c)$, the amplitude of the $(d,a)$ variation is greater than 1 and we reach a contradiction as we work in a $+1$-interval by hypothesis.
\end{proof}

\begin{lemma}
\label{lem:essential-double-var}
Let $[T_1,T_2]$ be $+1$-interval satisfying Condition~\eqref{cond:pure-interval-cb}. If $(c,a)_v$ and $(b,a)_w$ are variations and $(c,a)_v$ is an essential variation then $(c,b)_v$ is an essential variation.
\end{lemma}

\begin{proof}
By Condition~\eqref{cond:pure-interval-cb}, $(c,b)_v$ is a variation.  Using Lemma~\ref{lem:var-desc}, $a$ is a descendant of $b$ because $(b,a)$ is a variation. So if $b$ is a middle descendant of any $b'$ with $c > b' >b$, so is $a$. As $(c,a)$ is an essential variation, this implies that $(c,b)$ is also an essential variation.
\end{proof}

It is actually possible to prove that Condition~\eqref{cond:pure-interval-cb} is equivalent to the one of Lemma~\ref{lem:essential-double-var}, \emph{i.e.}, if the condition is true on essential variations then it is true for all variations. The proof is bit technical and not necessary at the moment.

We will prove Proposition~\ref{prop:pure-candidate-to-pure} in two steps: first, we show that minimal essential variations are indeed tree-ascents of $T_1$ (\Cref{prop:pure-candidate-minimal-ess-var}), then that every variation of the interval is obtained by transitivity from the minimal essential variations (\Cref{prop_transitive_min_essntial_variations}). 

\subsubsection{minimal essential variations are tree-ascents}

The following lemmas explore some useful properties of variations and essential variations in pure-candidate intervals.

\begin{lemma}[Weak transitivity of essential variations]
\label{lem:essential-var-weak-trans}
Let $[T_1, T_2]$ be a pure-candidate interval. Suppose that we have $c > b > a$ with $(c,b)_v$ and $(b,a)_0$ essential variations. Then $(c,a)_v$ is also an essential variation.
\end{lemma}

This is a version of Lemma~\ref{lem:trans} now specific to \emph{essential} variations. Note that Lemma~\ref{lem:trans} is not directly true for essential variations and that this version is not true in general: we need a pure-candidate interval. See for example the pure interval of Figure~\ref{fig:pure-interval10}. We have that $(10,4)_2$ and $(4,1)_0$ are essential variations, the lemma implies that $(10,1)_0$ is also an essential variation. On the other hand, $(10,9)_2$ and $(9,6)_1$ are essential variations but $(10,6)_2$ is not. On the example of Figure~\ref{fig:not-pure-interval} which is not a pure interval, we see that we have $(5,4)_0$ an essential variation as well as $(4,1)_0$ but $(5,1)_0$ is not an essential variation (it is a variation but $1$ is a middle child of $3$ and $(5,3)$ varies).

\begin{proof}
By Lemma~\ref{lem:trans}, we know that $(c,a)$ varies in $[T_1,T_2]$. More precisely, $\card_{T_1}(c,a) \leq v$ and $\card_{T_2}(c,a) > v$. Because $[T_1,T_2]$ is a pure-candidate interval, it is in particular a $+1$-interval and so $(c,a)_v$ is a variation of amplitude one. We need to prove that it is an essential variation. Suppose that it is not: there exists $b'$ with $c > b' > a$ such that $(c,b')_v$ is an essential variation and $a$ belongs to a middle child of $b'$. We have $\card_{T_1}(b,a) = 0$ so $a$ does not belong to a middle child of $b$ and $b \neq b'$. Suppose that $c > b > b'$. We have $(c,b)_v$ and $(c,b')_v$ essential variations and $s(b) \neq 0$ (because $(b,a)$ varies), by Condition~\eqref{cond:pure-interval-ba} on pure-candidate intervals, we obtain that $(b,b')_0$ is a variation, which contradicts the fact that $(b,a)$ is an essential variation. Now if $c > b' > b$, because $s(b') \neq 0$ (as $a$ belongs to a middle child of $b'$), we obtain that $(b',b)_0$ is a variation. By planarity (Definition~\ref{def:multi-sets}), this gives $\card_{T_1}(b,a) = s(b)$ which contradicts the fact that $(b,a)$ varies. 
\end{proof}

\begin{lemma}
\label{lem:middle-var-pure-cand}
Let $[T_1,T_2]$ be a pure-candidate interval. Suppose that $(c,a)_v$ is an essential variation. If $b$ is such that $c > b > a$ with $\card_{T_1}(c,b) = v$ and $\card_{T_1}(b,a) < s(b)$, then $(c,b)_v$ varies.
\end{lemma}

\begin{proof}
Suppose that $(c,b)$ does not vary. As $(c,a)$ varies, we have $\card_{T_2}(c,a) > \card_{T_1}(c,a)$. And because $(c,b)$ does not vary, we have $\card_{T_2}(c,b) = \card_{T_1}(c,b)$. We obtain that $\card_{T_2}(c,a) > \card_{T_2}(c,b)$ which implies by planarity (Definition~\ref{def:multi-sets}) that $\card_{T_2}(b,a) = s(b)$, \emph{i.e.}, $(b,a)$ varies. As $(c,a)$ and $(b,a)$ vary, Condition~\eqref{cond:pure-interval-cb} implies that $(c,b)$ also varies which leads to a contradiction. 
\end{proof}

\begin{lemma}
\label{lem:no-middle-child}
Let $[T_1,T_2]$ be a pure-candidate interval. Suppose that $(c,a)$ is an essential variation. Then, there is no $b$ such that $c > b > a$ and $a$ is a middle descendant of $b$.

In particular, if $b < c$ with $(b,a)$ a variation, then $(b,a)$ is an essential variation.
\end{lemma}

\begin{proof}
This is a direct consequence of Lemma~\ref{lem:middle-var-pure-cand}. If $a$ is a middle descendant of $b$, then $\card_{T_1}(c,b) = \card_{T_1}(c,a) = v$ and $\card_{T_1}(b,a) < s(b)$ which implies that $(c,b)$ varies and so $(c,a)$ is not an essential variation.

Now if $(b,a)$ is a variation but not an essential variation, this means that $a$ is a middle descendant of some $b'$ with $b > b' > a$ which contradicts the first part of the Lemma.
\end{proof}

In particular, Lemma~\ref{lem:no-middle-child} implies that the variation $(b,a)_0$ in Condition~\eqref{cond:pure-interval-ba} is always an essential variation. We can now prove the following proposition.

\begin{proposition}[Minimal essential variations are tree-ascents]
\label{prop:pure-candidate-minimal-ess-var}
Let $(c,a)_v$ be a minimal essential variation of a pure-candidate interval $[T_1, T_2]$. Then $(a,c)$ is a tree-ascent of $T_1$.
\end{proposition}

\begin{proof}

We check the conditions of Definition~\ref{def:tree-ascent}. Condition~\eqref{cond:tree-ascent-desc} is satisfied by Lemma~\eqref{lem:var-desc}. As $(c,a)$ varies, we have $\card_{T_1}(c,a) < s(c)$ and so Condition~\eqref{cond:tree-ascent-non-final} is satisfied.

Let us prove Condition~\eqref{cond:tree-ascent-middle}, \emph{i.e.}, if there is $b$ with $c > b > a$ and $a$ a descendant of $b$, then $a$ is a right descendant of $b$. Suppose that $\card_{T_1}(b,a) < s(b)$, by Lemma~\ref{lem:middle-var-pure-cand}, we have that $(c,b)_v$ is a variation. This is actually an essential variation because $a$ is a descendant of $b$ and $(c,a)_v$ is an essential variation. We then have essential variations $(c,a)_v$ and $(c,b)_v$ and we have supposed that $\card_{T_1}(b,a) < s(b)$ which implies that $s(b) > 0$. By Condition~\eqref{cond:pure-interval-ba} we have that $(b,a)_0$ is a variation. And by Lemma~\ref{lem:no-middle-child}, it is necessarily an essential variation. This contradicts the fact that $(c,a)$ is a minimal essential variation.

Let us prove Condition~\eqref{cond:tree-ascent-smaller}, \emph{i.e.}, if it exists, the strict right child of $a$ is empty. Suppose that it is not and that $a'$ is the root of its right subtree. As $s(a) > 0$, we have $\card_{T_2}(c,a') \geq \card_{T_2}(c,a) > \card_{T_1}(c,a) = \card_{T_1}(c,a')$ and so $(c,a')_v$ is a variation. Moreover, it is an essential variation because if $a'$ is a middle descendant of any node $b$, so is $a$. Condition~\eqref{cond:pure-interval-ba} then states that $(a,a')_0$ is a variation and so $\card_{T_1}(a,a') = 0$ which contradicts our initial statement. 
\end{proof}

\subsubsection{Transitive closure of minimal essential variations}

To prove that every variation is obtained by transitivity from the minimal essential variations, we introduce the fundamental notion of \emph{variation path}.

\begin{definition}[Variation path]
\label{def:variation-path}
Let $I = [T_1, T_2]$ be an interval and $(c,a)_v$ a variation. The \defn{variation path} between $c$ and $a$ is given by
\begin{equation*}
c > c_1 > \dots > c_k
\end{equation*}
where the set $\lbrace c_i \rbrace_{i=1}^k$ consists of all values $c_i \geq a$ such that $(c,c_i)_v$ is an essential variation and either $c_i = a$ or $\card_{T_1}(c_i,a) < s(c_i)$. 
\end{definition}

\begin{example}
\label{ex:var-path-pure}
Here are all variation paths for the variations of Figure~\ref{fig:pure-interval}:
\begin{itemize}
\item $(3,2)_1$: $3 > 2$;
\item $(3,1)_1$: $3 > 2$;
\item $(4,3)_0$: $4 > 3$;
\item $(4,2)_0$: $4 > 3$;
\item $(4,1)_0$: $4 > 3$.
\end{itemize}
\end{example}

\begin{example}
\label{ex:var-path-pure10}
Here are all variation paths for the variations of Figure~\ref{fig:pure-interval10}:
\begin{itemize}
\item $(10,8)_0$: $10 > 8$;
\item $(10,5)_0$: $10 > 5$;
\item $(10,9)_2$ and $(10,6)_2$ and $(10,2)_2$: $10 > 9$;
\item $(10,4)_2$ and $(10,3)_2$: $10 > 9 > 4$;
\item $(10,1)_2$: $10 > 9 > 4 > 1$;
\item $(9,4)_0$ and $(9,3)_0$: $9 > 4$;
\item $(9,1)_0$: $9 > 4 > 1$;
\item $(9,6)_1$ and $(9,2)_1$: $9 > 6$;
\item $(6,2)_1$: $6 > 2$;
\item $(4,1)_0$: $4 > 1$.
\end{itemize}
\end{example}

\begin{lemma}
\label{lem:variation-path-prop}
If $I = [T_1, T_2]$ is a pure-candidate interval, then the variation path of any variation $(c,a)_v$ satisfies the following properties
\begin{enumerate}[(i)]
\item for all $j > i$, then $(c_i,c_j)_0$ is an essential variation;
\label{cond:variation-path-0}
\item either $c_k = a$ or $a$ belongs to a middle child of $c_k$;
\label{cond:variation-path-middle-child}
\item $(c,c_1)$ is a minimal essential variation as well as $(c_i, c_{i+1})$ for all $i < k$. 
\label{cond:variation-path-minimal}
\end{enumerate}
\end{lemma}

\begin{proof}
Property~\eqref{cond:variation-path-0} is a consequence of Condition~\eqref{cond:pure-interval-ba} of pure-candidate intervals. Indeed, by definition of the variation path we have that $(c,c_i)_v$ and $(c,c_j)_v$ are essential variations and $s(c_i) > 0$ as $0 \leq \card_{T_1}(c_i,a) < s(c_i)$. This implies that $(c_i, c_j)_0$ is an essential variation (using Condition~\eqref{cond:pure-interval-ba} and Lemma~\ref{lem:no-middle-child}).

Property~\eqref{cond:variation-path-middle-child} is a direct consequence of Lemma~\ref{lem:essential-var-middle-child}. Indeed, suppose that $c_k \neq a$, \emph{i.e.}, $(c,a)$ is not an essential variation. Then, there is a unique $b > a$ such that $(c,b)_v$ is an essential variation with $a$ a middle child of $b$. In particular, $\card_{T_1}(b,a) < s(b)$ so $b \in \lbrace c_i \rbrace_{i=1}^k$. Now, if $b \neq c_k$, then $b = c_i$ with $i < k$. By Property~\eqref{cond:variation-path-0}, we have $\card_{T_1}(c_i,c_k) = 0$, which implies $\card_{T_1}(c_k,a) = s(c_k)$ and contradicts the definition of the variation path. 

We now prove Property~\eqref{cond:variation-path-minimal}. We know by definition that $(c,c_1)_v$ is an essential variation. Let us suppose that it is not minimal, \emph{i.e.}, there is $b$ with $c > b > c_1$ such that $(b,c_1)$ is an essential variation. Using Lemma~\ref{lem:essential-double-var}, we obtain that $(c,b)_v$ is an essential variation. As $b$ does not belong to the variation path, we have by definition that $\card_{T_1}(b,a) = s(b)$. We obtain a contradiction to the planarity condition on tree-inversions (see Definition~\ref{def:multi-sets}) on $b > c_1 > a$. Indeed, we should have either $\card_{T_1}(c_1,a) = s(c_1)$ which is impossible by definition of the variation path, or $\card_{T_1}(b,c_1) \geq \card_{T_1}(b,a) = s(b)$ which is impossible as we have supposed $(b,c_1)$ to be a variation. 

This gives us that $(c,c_1)$ is a minimal variation. Now, let us look at $(c_i,c_{i+1})$. First, see that if $i \neq k$, $(c_i,a)_0$ is a variation. Indeed, $(c_i,c_k)_0$ is an essential variation and either $c_k = a$ or $a$ is a middle child of $c_k$ so $(c_i,a)$ varies by transitivity. Now if $c_{i+1}$ is the first element of the variation path between $c_i$ and~$a$, we are done by applying our previous argument on this new variation path. We need to prove that this is always the case, \emph{i.e.}, if $b$ belongs to the variation path between $c_i$ and $a$, it also belongs to the variation path between $c$ and $a$. Let $b$ be such that $c_i > b > a$ with $(c_i,b)_0$ an essential variation and $\card_{T_1}(b,a) < s(b)$. We know that $(c,c_i)_v$ is an essential variation, then by weak transitivity of essential variations (see Lemma~\ref{lem:essential-var-weak-trans}), we obtain that $(c,b)_v$ is also an essential variation and so $b$ belongs to the variation path. 
\end{proof}

\begin{proposition}[Transitive closure of minimal essential variations]
\label{prop_transitive_min_essntial_variations}
Let $[T_1, T_2]$ be a pure-candidate interval and let $A$ be the set of couples $(a,c)$ such that $(c,a)$ is a minimal essential variation of $[T_1,T_2]$. Then $T_2 = T_1 + A$. 
\end{proposition}

\begin{proof}
We want to prove that for all $c > a$, we have $\card_{T_2}(c,a) = \card_{T_1 + A}(c,a)$. It is clear that $T_1 + A \leq T_2$. Indeed, let $S$ be the multi-set of inversions such that $\card_S(c,a) = \card_{T_1}(c,a) + 1$ if $(c,a)$ is a minimal essential variation of $[T_1, T_2]$ and $\card_S(c,a) = \card_{T_1}(c,a)$ otherwise ($\tc{S} = \inv(T+A)$). Clearly, as the minimal essential variations are a subset of the variations of $[T_1, T_2]$, we have $S \subseteq \inv(T_2)$ which implies that it is also the case for the transitive closure. In particular, if $(c,a)$ is not a variation, $\card_{T_2}(c,a) = \card_{T_1}(c,a) = \card_{T_1 + A}(c,a)$. Let us assume that $(c,a)_v$ is a variation in $[T_1, T_2]$ and take $c > c_1 > \dots > c_k \geq a$ the variation path of $(c,a)$. By Lemma~\ref{lem:variation-path-prop}, $(c,c_1)_v$ and $(c_i, c_{i+1})_0$ are minimal essential variations for all $i < k$. This means that $\card_{S}(c,c_1) = v +1$ and $\card_{S}(c_i,c_{i+1}) = 1$. Besides, if $c_k \neq a$, then $a$ is a middle descendant of $c_k$ and $\card_{T_1}(c_k,a) > 0$. By taking the transitive closure of the multi-set of inversions $S$, we obtain indeed $\card_{T_1 + A}(c,a) \geq v+1$. 
\end{proof}

Now we are ready to prove Proposition~\ref{prop:pure-candidate-to-pure}, which states that pure-candidate intervals are pure intervals.

\begin{proof}[Proof of Proposition~\ref{prop:pure-candidate-to-pure}]
Let $[T_1, T_2]$ be a pure-candidate interval and let $A$ be the set of couples $(a,c)$ such that $(c,a)$ is a minimal essential variation of $[T_1,T_2]$. 
By \Cref{prop_transitive_min_essntial_variations}, we have that $T_2 = T_1 + A$. 
Furthermore, we know that $A$ is subset of tree-ascents of $T_1$ by Proposition~\ref{prop:pure-candidate-minimal-ess-var}.
This proves that $[T_1, T_2]$ is a pure interval. 
\end{proof}

\subsection{Proof Part 2: Pure intervals are pure-candidate intervals}\label{sec_pure_purecandidate}
\begin{proposition}
\label{prop:pure-to-pure-candidate}
Let $[T,T+A]$ be a pure interval, then it is a pure-candidate interval.
\end{proposition}

In order to prove this proposition we need to show that the variations and essential variations of any interval $[T, T+A]$ satisfy Conditions~\eqref{cond:pure-interval-cb} and~\eqref{cond:pure-interval-ba} of \Cref{thm:pure-interval-char}. 
For this, it is usefull to characterize the variations and essential variations of a pure interval in terms of its \emph{ascent-paths}, a notion that we now define. 

\subsubsection{Ascent paths}
In the following, we adopt a notation for tree-ascents that is similar to the one we use for variations: we write $(a,c)_v$ a tree-ascent $(a,c)$ where $\card(c,a) = v$.

\begin{definition}[Ascent paths]
\label{def:ascent-path}
Let $[T,T+A]$ be a pure interval and $c > a$ with $\card_T(c,a) = v$. We say that there is an \defn{ascent-path} $c > c_1 > \dots > c_k \geq a$ if
\begin{enumerate}[(i)]
\item $(c_1,c)_v \in A$;
\label{cond:ascent-path-first}
\item $(c_i, c_{i-1})_0 \in A$ for all $1 < i \leq k$;
\label{cond:ascent-path-middle}
\item either $c_k = a$ or $a$ is a middle descendant of $c_k$.
\label{cond:ascent-path-last}
\end{enumerate}
\end{definition}

You can check on Examples~\ref{ex:var-path-pure} and~\ref{ex:var-path-pure10} that all variation paths are actually ascent paths. We are indeed going to prove that they are the same.

\begin{remark}
\label{rem:ascent-path}
As we have a series of tree-ascents, some properties of the ascent-paths are immediate
\begin{enumerate}[(i)]
\item $a$ is a descendant of $c$ in $T$;
\item for all $i$, $c > c_1 > \dots c_i$ is also an ascent-path, in particular $c_i$ is a descendant of $c$ with $\card_T(c,c_i) = v$;
\item for all $i < j$, $c_i > \dots > c_j$ is also an ascent-path, in particular $c_j$ is a descendant of $c_i$ with $\card_T(c_i, c_j) = 0$;
\item for all $i$, $c_i > \dots > c_k \geq a$ is also an ascent-path, in particular $a$ is a descendant of $c_i$ with $\card_T(c_i, a) = 0$.
\end{enumerate}
\end{remark}

\begin{lemma}
\label{lem:middle-variation-pure}
Let $c > b > a$ such that there is an ascent-path $c > c_1 > \dots c_k \geq a$ between $c$ and $a$ with $\card_T(c,a) = \card_T(c,b) = v$ (in particular, $v < s(c)$). Then, one of the following is true:
\begin{enumerate}[(i)]
\item $b = c_i$ for a certain $i$;
\label{case:middle-var-ci}
\item $b$ is a middle descendant of $c_i$ for a certain $i$;
\label{case:middle-var-middle}
\item $\card_T(b,a) = s(b)$.
\label{case:middle-var-end}
\end{enumerate}
\end{lemma}

\begin{proof}
Suppose that $b \neq c_i$ for all $i$. We first look at the case where $b > c_k$. Let $c_0 := c$ and $(c_{i+1}, c_i)$ with $0 \leq i < k$ be the tree-ascent surrounding $b$, \emph{i.e.}, $c_{i} > b > c_{i+1}$. Either $c_i = c$ and $\card_T(c,b) = \card_T(c_i, c_{i+1})$ by hypothesis or $\card_T(c_i, c_{i+1}) = 0$. If $\card_T(c_{i},b) = \card_T(c_i, c_{i+1})$, then Statement~\eqref{cond:inv-tree-ascent-middle} of Proposition~\ref{prop:tree-ascent-inversions} gives that $\card_T(b,c_{i+1}) = s(b)$ and we are in Case~\eqref{case:middle-var-end}. If $\card_T(c_{i},b) \neq \card_T(c_i, c_{i+1})$, then $c_i \neq c$ and $\card_T(c_i, b) > \card_T(c_i, c_{i+1}) = 0$. As $(c_i, c_{i-1})$ is a tree-ascent, either, $\card_T(c_i, b) < s(c_i)$ and we are in Case~\eqref{case:middle-var-middle} or Statement~\eqref{cond:inv-tree-ascent-smaller} of Proposition~\ref{prop:tree-ascent-inversions} gives $\card_T(c_{i-1},b) > \card_T(c_{i-1},c_i)$ (because the right child of $c_i$ is empty). This implies that $i-1 > 0$ and we can use the same reasoning until we reach an element $c_j$ with $0 < j \leq i$ such that $0 < \card_T(c_j,b) < s(c_j)$: \emph{i.e.}, $b$ is a middle descendant of $c_j$.  

In the case where $b < c_k$, if $(a,c_k)_0$ is a tree-ascent, then the previous case still applies. If not, $a$ is a middle child of $c_k$ and $(c_k, c_{k-1})$ is a tree-ascent. If $\card_T(c_k, b) = 0 < \card_T(c_k, a)$, we are in Case~\eqref{case:middle-var-end}. If $\card_T(c_k, b) < s(c_k)$, then $b$ is a middle descendant of $c_k$. If $\card_T(c_k, b) = s(c_k)$, we have $s(c_k) > 0$ (it has a middle child). We get that $\card_T(c_{k-1}, b) > \card_T(c_{k-1}, c_k)$ and we can apply the same reasoning as earlier to find $c_j$ such that $b$ is a middle descendant of $c_j$. 
\end{proof}

\begin{example}
Look again at Figure~\ref{fig:pure-interval10}. There is an ascent path between $10 > 9 > 4 > 1$ using the tree-ascents $(1,4)_0$, $(4,9)_0$, and $(9,10)_2$. We look at all $b$ with $10 > b > 1$ and $\card_{T_1}(10,b) = 2$. 
\begin{itemize}
\item $4$ and $9$ belong to the ascent-path and correspond to Case~\eqref{case:middle-var-ci};
\item $2$, $3$ et $6$ are middle descendants of either $4$ or $9$, they correspond to Case~\eqref{case:middle-var-middle};
\item $7$ corresponds to Case~\eqref{case:middle-var-end} as we have $\card_{T_1}(7,1) = 1 = s(7)$.
\end{itemize} 
\end{example}

%The following two lemmas characterize the variations and essential variations of a pure interval in terms of its ascent paths. 

%\begin{lemma}
%\label{lem:ascent-path-variation}
%If there is an ascent path $c > c_1 \dots > c_k \geq a$ with $\card_T(c,a) = v$ in a pure interval $[T,T+A]$, then $(c,a)_v$ is a variation of $[T,T+A]$. It is an essential variation if and only if $c_k = a$. 
%
%More specifically, for all $b$ such that $c > b > c_k$, then $c_k$ is never a middle descendant of $b$.
%\end{lemma}

\begin{lemma}
\label{lem:ascent-path-variation}
If $c > c_1 \dots > c_k \geq a$ is an ascent path in a pure interval $[T,T+A]$, with $\card_T(c,a) = v$, then $(c,a)_v$ is a variation of $[T,T+A]$. It is an essential variation if and only if $c_k = a$. 

More specifically, for all $b$ such that $c > b > c_k$, then $c_k$ is never a middle descendant of $b$.
\end{lemma}

%\begin{lemma}
%\label{lem:ascent-path-variation}
%If $c > c_1 \dots > c_k \geq a$ is an ascent path in a pure interval $[T,T+A]$, with $\card_T(c,a) = v$, then %Then 
%\begin{enumerate}
%\item $(c,a)_v$ is a variation of $[T,T+A]$.
%\item $(c,a)_v$ is an essential variation if and only if $c_k = a$. 
%More specifically, for all $b$ such that $c > b > c_k$, then $c_k$ is never a middle descendant of $b$.
%
%\end{enumerate}
%\end{lemma}

\begin{proof}
It is quite clear that $(c,a)$ varies. By hypothesis, $\card_T(c,a) = v$. As each step in $c > c_1 > \dots > c_k$, $(c_i, c_{i-1})$ is a tree-ascent of $A$, the cardinality increases and you get $\card_{T+A}(c,c_1) = v+1$ and $\card_{T+A}(c_i, c_{i+1}) = 1$. Besides, if $a \neq c_k$, then it is a middle descendant of $c_k$ meaning $\card_{T+A}(c_k,a) \geq \card_T(c_k,a) > 0$. By transitivity, $\card_{T+A}(c,a) \geq v+1$. 

If $a \neq c_k$, then $(c,c_k)$ varies and $(c,a)$ is not an essential variation. Reversely, assuming the second part of the lemma is true, we get that $(c,c_k)$ is always an essential variation. 

Now let $b$ be such that $c > b > c_k$. If $\card_{T_1}(c,b) \neq v$, then $c_k$ cannot be a middle descendant of $b$. We assume $\card_{T_1}(c,b) = v$ and look at the different possibilities of Lemma~\ref{lem:middle-variation-pure} on the ascent-path between $c$ and $c_k$. If $b = c_i$ for a certain $i$, then $\card_T(b,c_k) = 0$. If $b$ is a middle descendant of a certain $c_i$, then $c_k$ lies to the left of $b$ and $\card_T(b,c_k) = 0$. The only case left is $\card_T(b,c_k) = s(b)$. We get that $c_k$ is never middle descendant of $b$.
\end{proof}

\begin{lemma}
\label{lem:pure-var-path}
Let $(c,a)_v$ be a variation of a pure interval $[T,T+A]$, then the variation path $c > c_1 > \dots > c_k \geq a$ of $(c,a)_v$ as defined by Definition~\ref{def:variation-path} is an ascent-path.
\end{lemma}

\begin{proof}
Let $(c,a)_v$ be a variation of $[T,T+A]$. Let $S$ be the multi-set of inversions set such that $\card_S(c,a) = \card_T(c,a) + 1$ if $(a,c) \in A$ and $\card_S(c,a) = \card_T(c,a)$ otherwise, so $\tc{S}$ is the tree-inversion set of $T+A$. By definition, there is a transitivity path in $S$, $c = c_0 > c_1 > \dots > c_k > c_{k+1} = a$ such that $\card_S(c,c_1) = v + 1$ and $\card_S(c_i,c_{i-1}) > 1$. We choose the path so that $k$ is minimal and we prove that this is an ascent-path. In the following, we suppose that $k \geq 1$ as otherwise, $(a,c)$ is a tree-ascent of $A$ and there is nothing to prove. 

Let us look first at $(c,c_1)$. Suppose that we have $\card_T(c,c_1) = v + 1$ and look at the next step, $(c_1,c_2)$ (this always exists but sometimes $c_2 = a$). If $\card_T(c_1,c_2) > 0$, then $\card_T(c,c_2) \geq v +1$ and $c > c_2 > \dots > a$ is a shorter transitivity path. This implies that $\card_T(c_1,c_2) = 0$. As $\card_S(c_1,c_2) > 0$, it means that $(c_2,c_1)$ is a tree-ascent of $A$. By Condition~\eqref{cond:inv-tree-ascent-desc} of Proposition~\ref{prop:tree-ascent-inversions}, we obtain $\card_T(c,c_2) = \card_T(c,c_1) = v+1$ and again $c > c_2 > \dots \geq c_k$ is a shorter transitivity path. This means that $\card_T(c,c_1) = v$ and that $(c_1,c) \in A$.

A similar reasoning can be made for $(c_{i-1}, c_i)$ for all $1 < i \leq k$. We have $c_{i-1} > c_i > c_{i+1}$. If $\card_T(c_{i-1}, c_i) > 0$, then $\card_T(c_{i}, c_{i+1}) = 0$ otherwise one can build a shorter transitivity path skipping $c_{i}$. Then $\card_T(c_{i-1}, c_{i}) > 0$ implies that $(c_{i+1}, c_{i})$ is a tree-ascent of $A$ and $\card_T(c_{i-1},c_{i}) = \card_T(c_{i-1},c_{i}) = 0$ otherwise, again, $c_{i}$ could be skipped in the transitivity path. The only case left is $(c_k, a)$. We know that $\card_S(c_k,a) > 0$. If $\card_T(c_k,a) = 0$, this means that $(a,c_k)_0$ is a tree-ascent of $A$ and $c > c_1 > \dots > c_{k+1} = a$ is an ascent-path. If $\card_T(c_k,a) > 0$, as $(c_k,c_{k-1})$ is a tree-ascent, the strict right child of $c_k$ is empty, \emph{i.e.} $a$ is a middle child of $c_k$. 

We have proved that having a variation $(c,a)$ implies to have an ascent-path between $c$ and $a$, in particular, $a$ has to be a descendant of $c$. There is left to prove that this ascent-path is indeed the variation path of $(c,a)_v$. Note that for all $i$, $c > c_1 > \dots > c_i$ is an ascent-path as well and by Lemma~\ref{lem:ascent-path-variation} we obtain that $(c,c_i)_v$ is an essential variation. For all $i < k$, we have $\card_T(c_i,a) = 0 < s(c_i)$ (there is a variation $(c_{i+1},c_i)$ so $s(c_i) > 0$) and so $c_i$ belongs to the variation path. Now either $(c_k,a)_0$ is a variation or $a$ is a middle child of $c_k$, so in both cases, $\card_T(c_k,a) < s(c_k)$ and $c_k$ also belongs to the variation path. We need to prove that the variation path consists only of those elements. Let $(c,b)_v$ be an essential variation such that $c > b > a$ and $\card_T(b,a) < s(b)$. We use Lemma~\ref{lem:middle-variation-pure}. Case~\eqref{case:middle-var-end} is forbidden by hypothesis. Case~\eqref{case:middle-var-middle} is not possible because $(c,b)_v$ is an essential variation. The only possibility left is Case~\eqref{case:middle-var-ci}: $b$ is one of the $c_i$.
\end{proof}

\begin{proposition}
\label{prop:ascent-var-paths}
Variation paths and ascent paths of a pure interval are the same.
\end{proposition}
\begin{proof}
This is a direct consequence of Lemma~\ref{lem:pure-var-path}. The variation path is unique for each variation by definition. The ascent path (if it exists) between two values $c$ and $a$ is also unique: indeed if $a \neq c_k$, then the choice for $c_k$ is unique by Lemma~\ref{lem:essential-var-middle-child} and then for $1 \leq i \leq k$, there is unique tree-ascent $(c_i, *)$. Each ascent path corresponds to a variation and we have proved in Lemma~\ref{lem:pure-var-path} that the variation path of this variation is the ascent path.
\end{proof}

\subsubsection{Variations and essential variations of pure intervals}
As a straight forward consequence of \Cref{lem:ascent-path-variation,lem:pure-var-path}, we also obtain the following characterization of variations and essential variations of pure intervals.

\begin{proposition}
\label{prop_variations_essvariations_pureintervals}
%$(c,a)_v$ is a variation of a pure interval $[T,T+A]$ if and only if there is an ascent-path $c > c_1 > \dots > c_k \geq a$.
Let $[T,T+A]$ be a pure interval. Then, 
 $(c,a)_v$ is a variation if and only if there is an ascent-path $c > c_1 > \dots > c_k \geq a$.
  It is an essential variation if and only if $c_k = a$. 
\end{proposition}

Using this, we are now ready to prove that pure intervals are pure-candidate intervals.

\begin{proof}[Proof of Proposition~\ref{prop:pure-to-pure-candidate}]
Let $T$ be an $s$-decreasing tree and $A$ a subset of the tree-ascents of $T$. 
We will prove that the pure interval $[T,T+A]$ is a pure-candidate interval, meaning that it satisfies 
Conditions~\eqref{cond:pure-interval-cb} and~\eqref{cond:pure-interval-ba} of \Cref{thm:pure-interval-char}. 

We first prove that $[T,T+A]$ satisfies Condition \eqref{cond:pure-interval-cb}. Let $(c,a)_v$ and $(b,a)_w$ be variations of $[T,T+1]$ with $c > b > a$. Lemma~\ref{lem:pure-var-path} states that there is an ascent-path between $c$ and $a$ and an ascent-path between $b$ and $a$. In particular, $a$ is a descendant of both $c$ and $b$ which implies $\card_{T_1}(c,b) = \card_{T_1}(c,a) = v$. We now use Lemma~\ref{lem:middle-variation-pure} on $c > b > a$ and the ascent-path $c > c_1 > \dots > c_k \geq a$. The three possible cases give
\begin{enumerate}[(i)]
\item $b = c_i$ for a certain $i$: then $(c,b)_v$ is a variation.
\item $b$ is a middle descendant of a certain $c_i$: $(c,b)_v$ is variation by transitivity because $(c,c_i)_v$ is a variation.
\item $\card_T(b,a) = s(b)$: this is impossible because $(b,a)$ varies.
\end{enumerate}

Now, let us prove Condition~\eqref{cond:pure-interval-ba}. Suppose that there exists $c > b > a$ with $(c,a)_v$ and $(c,b)_v$ essential variations for some $v$ and $s(b) > 0$ and $(b,a)_0$ is not a variation. We choose $b$ such that the distance between $c$ and $b$ is minimal. In particular, if $c > b_1 > \dots > b_{k'} = b$ is the variation path between $c$ and $b$, we have that $(b_i, a)_0$ is a variation for $i <k'$ . Indeed, $c > b_i > b > a$ and $(c,b_i)_v$ is an essential variation and $s(b_i) > 0$ (as $(b_{i+1}, b_i)$ is a tree-ascent). We now look at the ascent-path $c > c_1 > \dots > c_k = a$ between $c$ and $a$ and again use Lemma~\ref{lem:middle-variation-pure} on $c > b > a$. If $b = c_i$ for some $i$, then there is an ascent-path $b = c_i > \dots > c_k = a$ and by Lemma~\ref{lem:ascent-path-variation}, we have an essential variation $(b,a)_0$. As $(c,b)$ is an essential variation, $b$ cannot be a middle descendant of any $c_i$. Only the last case remains where $\card_T(b,a) = s(b) > 0$. We write, $c_0 := c$. The ascent-path between $c$ and $b$ gives us that $(b,b_{k'-1})$ is a tree-ascent.  As $s(b) > 0$, then $\card_T(b,a) = s(b)$ implies that $\card_T(b_{k'-1}, a) > \card_T(b_{k'-1}, b)$ by Statement~\eqref{cond:inv-tree-ascent-smaller} of Proposition~\ref{prop:tree-ascent-inversions}. Either $b_{k'-1} = c$ and this contradicts $\card_T(c,a) = \card_T(b,a)$ or $(b_{k'-1}, a)_0$ is an essential variation by minimality of the distance $c - b$ and again we reach a contradiction. 
\end{proof}

\subsubsection{Proof of Theorems~\ref{thm:pure-interval-char} and~\ref{thm_min_essential_variations_ascents}}

We now have all the ingredients to prove Theorems~\ref{thm:pure-interval-char} and~\ref{thm_min_essential_variations_ascents}. 

\begin{proof}[Proof of Theorem~\ref{thm:pure-interval-char}]
A pure candidate interval is by definition an interval satisfying the two conditions in Theorem~\ref{thm:pure-interval-char}. So, we need to show that pure intervals and pure candidate intervals are the same. This was shown in Proposition~\ref{prop:pure-candidate-to-pure} and Proposition~\ref{prop:pure-to-pure-candidate}.
\end{proof}

\begin{proof}[Proof of Theorem~~\ref{thm_min_essential_variations_ascents}]
Let $[T,T+A]$ be a pure interval. By Proposition~\ref{prop:pure-to-pure-candidate}, it is also a pure candidate interval. So, we can apply Proposition~\ref{prop:pure-candidate-to-pure} to deduce that $(a,c) \in A$ if and only if $(c,a)$ is a minimal essential variation.
\end{proof}

\subsection{Properties of variations of pure intervals}
Now that we have proved that pure-candidate intervals and pure intervals are the same: all the properties of variations / essential variations that we have shown to prove either Proposition~\ref{prop:pure-candidate-to-pure} or Proposition~\ref{prop:pure-to-pure-candidate} are satisfied in pure intervals. In particular each variation is given by a certain variation path (from Definition~\ref{def:variation-path}), which is also an ascent-path (Definition~\ref{def:ascent-path}). The specific properties of variations in pure intervals are crucial in our next section. We regroup some of them in this proposition for clarity. We call them the \emph{middle variation} properties as they all concern the variation of $(c,b)$ depending on properties of $(c,a)$ where $c > b > a$.

\begin{proposition}[Middle variation properties]
\label{prop:middle-b}
Let $[T,T+A]$ be a pure interval and $c > b > a$. The following properties hold.
\begin{enumerate}[(1)]
\item If $(c,a)_v$ is a variation, and $\card_T(c,b) = v$ and $\card_T(b,a) < s(b)$, then then $(c,b)_v$ is a variation.
\label{case-prop:a-before-b}
\item If $(c,a)_v$ is a variation and $a$ is a middle descendant of $b$, then $(c,b)_v$ is a variation.
\label{case-prop:middle-var}
\item If $(c,a)_v$ is an essential variation, then $a$ is not a middle descendant of $b$.
\label{case-prop:ev-no-middle}
\item If $(c,a)_v$ is an essential variation and $(b,a)_w$ a variation, then $(b,a)_w$ and $(c,b)_v$ are essential variations.
\label{case-prop:double-var}
\item If $(c,b)_v$ is a variation with $s(b) > 0$ and $\card_T(b,a) = s(b)$ and $\card_T(c,a) = v$, then there exists $b'$ with $c > b' > b$ and $a$ is a middle descendant of $b'$ and $(c,b')$ a variation.
\label{case-prop:a-middle-desc}
\item If $(c,a)_v$ is an essential variation and $(c,b)_v$ is a variation, then $\card_T(b,a) = 0$.
\label{case-prop:ca-ev-ba-0}
\item If $(c,a)_v$ is a variation and $\card_T(c,b) = v$ and $\card_T(b,a) = 0$ and $s(b) > 0$, then either $(b,a)$ is a variation or there is $b'$ with $c > b' > b$ with $b$ a middle descendant of $b'$ and $(c,b')_v$ a variation.
\label{case-prop:ba-var}
\item If $(c,a)_v$ is a variation and $(c,b)_v$ is an essential variation with $s(b) > 0$ and  $\card_T(b,a) = 0$ then $(b,a)$ is a variation.
\label{case-prop:cb-ev-ba-var}
\end{enumerate}
\end{proposition}

\begin{proof}
Property~\eqref{case-prop:a-before-b} is Lemma~\ref{lem:middle-var-pure-cand}. It is also a consequence of Lemma~\ref{lem:middle-variation-pure}. This implies Property~\eqref{case-prop:middle-var}: if $a$ is a middle descendant of $b$, then $\card_T(c,b) = v$ and $\card_T(b,a) < s(b)$. Then Property~\eqref{case-prop:ev-no-middle} is a consequence of~\eqref{case-prop:middle-var}: if $(c,b)$ varies and $a$ is a middle descendant of $b$, then $(c,a)$ is not an essential variation. Property~\ref{case-prop:double-var} is Lemma~\ref{lem:essential-double-var} with the addition that $(b,a)$ is also an essential variation which is a consequence in particular of~\ref{case-prop:a-before-b}.

We prove Property~\eqref{case-prop:a-middle-desc}. Suppose that we have $(c,b)_v$ a variation, $s(b) > 0$,  $\card_T(b,a) = s(b)$  and $a$ is never a middle descendant of any $b'$ with $c > b' > b$ such that $(c,b')$ varies. We prove that it implies $\card_T(c,a) > v$. We look at the ascent-path $c = c_0 > c_1 > \dots > c_k \geq b$ between $c$ and $b$. Either $c_k = b$ and by hypothesis $\card_T(c_k,a) = s(c_k) > 0$. Or $b$ is middle descendant of $c_k$ and as we have $\card_T(b,a) = s(b) > 0$, we get by transitivity $\card_T(c_k,a) \geq \card_T(c_k,b) > 0$. As $(c,c_k)$ varies, by hypothesis $a$ cannot be a middle descendant of $c_k$, so this gives again $\card_T(c_k,a) = s(c_k) > 0$. As we have an ascent-path, $(c_k, c_{k-1})$ is a tree-ascent. Using Statement~\eqref{cond:inv-tree-ascent-smaller} of Proposition~\ref{prop:tree-ascent-inversions}, we obtain $\card_T(c_{k-1},a) > \card_T(c_{k-1},c_k)$. Either $c_{k-1} = c$ and we are done, or as $a$ cannot be a middle descendant of $c_{k-1}$, we get $\card_T(c_{k-1},a) = s(c_{k-1})$. And $s(c_{k-1}) > 0$ because $(c_{k-1},c_k)$ varies. Besides, $(c_{k-1},c_{k-2})$ is a tree-ascent so we apply the same reasoning until we reach $c_0 = c$. This implies Property~\eqref{case-prop:ca-ev-ba-0}. Indeed, $(c,a)_v$ is an essential variation so $a$ cannot be a middle descendant of $b$ (Property~\eqref{case-prop:ev-no-middle}) so either $\card_T(b,a) = s(0)$ or $\card_T(b,a) = s(b)$. But in this last case, Property~\eqref{case-prop:a-middle-desc} tells us that $a$ is a middle descendant of $b'$ which is forbidden.

Property~\eqref{case-prop:ba-var} is a direct consequence of Lemma~\ref{lem:middle-variation-pure}. As $s(b) > 0$ and $\card_T(b,a) = 0$, only cases~\eqref{case:middle-var-ci} and~\eqref{case:middle-var-middle} are possible. In case~\eqref{case:middle-var-ci}, $b = c_i$ for $c_i$ an element of the variation path. If $i < k$, then $(b,a)$ varies as in Remark~\ref{rem:ascent-path}. If $i = k$, then $a$ is a middle descendant of $b$ and so $\card_T(b,a) > 0$. In case~\eqref{case:middle-var-middle}, we have that $b$ is a middle descendant of some other node as described in the property. This implies Propery~\eqref{case-prop:cb-ev-ba-var}: as $(c,b)$ is an essential variation, $b$ cannot be a middle descendant of any $b'$ and then we have that $(b,a)$ is a variation.
\end{proof}

\section{Intersection of pure intervals}\label{sec_intersection_pure_intervals}

\subsection{The Intersection Theorem}\label{sec_intersection_theorem}
The goal is to prove the following theorem.

\begin{theorem}
\label{thm:intersection}
The intersection of two pure intervals is a pure interval.
\end{theorem}

\begin{example}
Figure~\ref{fig:s022-complex} shows all the pure intervals for the $s$-weak lattice with $s=(0,2,2)$. The lattice is drawn such that each pure interval corresponds to a cell. The dimension of the cell is the number of selected tree-ascents in the pure interval. In this case, the intersection of two pure intervals correspond to the geometric intersection of the cells. In Figure~\ref{fig:inter-pure}, we show two cells of dimension two intersecting in a cell of dimension one and two cells with an empty intersection.
\end{example}

\begin{figure}
\begin{equation*}
\begin{aligned}\scalebox{.6}{{ \newcommand{\nodea}{\node[draw,circle,fill=white] (a) {$3$}
;}\newcommand{\nodeb}{\node[draw,circle,fill=white] (b) {$$}
;}\newcommand{\nodec}{\node[draw,circle,fill=red!50] (c) {$2$}
;}\newcommand{\noded}{\node[draw,circle,fill=red!50] (d) {$1$}
;}\newcommand{\nodee}{\node[draw,circle,fill=white] (e) {$$}
;}\newcommand{\nodef}{\node[draw,circle,fill=white] (f) {$$}
;}\newcommand{\nodeg}{\node[draw,circle,fill=white] (g) {$$}
;}\newcommand{\nodeh}{\node[draw,circle,fill=white] (h) {$$}
;}\begin{tikzpicture}[auto]
\matrix[column sep=0.1cm, row sep=.3cm,ampersand replacement=\&]{
         \&         \& \nodea  \&         \&         \\ 
 \nodeb  \&         \& \nodec  \&         \& \nodeh  \\ 
         \& \noded  \& \nodef  \& \nodeg  \&         \\ 
         \& \nodee  \&         \&         \&         \\
};

\path[ultra thick, red] (d) edge (e)
	(c) edge (d) edge (f) edge (g)
	(a) edge (b) edge (c) edge (h);
\end{tikzpicture}}}\end{aligned}
 \cap 
\begin{aligned}\scalebox{.6}{{ \newcommand{\nodea}{\node[draw,circle,fill=white] (a) {$3$}
;}\newcommand{\nodeb}{\node[draw,circle,fill=white] (b) {$$}
;}\newcommand{\nodec}{\node[draw,circle,fill=red!50] (c) {$2$}
;}\newcommand{\noded}{\node[draw,circle,fill=white] (d) {$$}
;}\newcommand{\nodee}{\node[draw,circle,fill=red!50] (e) {$1$}
;}\newcommand{\nodef}{\node[draw,circle,fill=white] (f) {$$}
;}\newcommand{\nodeg}{\node[draw,circle,fill=white] (g) {$$}
;}\newcommand{\nodeh}{\node[draw,circle,fill=white] (h) {$$}
;}\begin{tikzpicture}[auto]
\matrix[column sep=0.1cm, row sep=.3cm,ampersand replacement=\&]{
         \&         \& \nodea  \&         \&         \\ 
 \nodeb  \&         \& \nodec  \&         \& \nodeh  \\ 
         \& \noded  \& \nodee  \& \nodeg  \&         \\ 
         \&         \& \nodef  \&         \&         \\
};

\path[ultra thick, red] (e) edge (f)
	(c) edge (d) edge (e) edge (g)
	(a) edge (b) edge (c) edge (h);
\end{tikzpicture}}}\end{aligned}
= 
\begin{aligned}\scalebox{.6}{{ \newcommand{\nodea}{\node[draw,circle,fill=white] (a) {$3$}
;}\newcommand{\nodeb}{\node[draw,circle,fill=white] (b) {$$}
;}\newcommand{\nodec}{\node[draw,circle,fill=red!50] (c) {$2$}
;}\newcommand{\noded}{\node[draw,circle,fill=white] (d) {$$}
;}\newcommand{\nodee}{\node[draw,circle,fill=white] (e) {$1$}
;}\newcommand{\nodef}{\node[draw,circle,fill=white] (f) {$$}
;}\newcommand{\nodeg}{\node[draw,circle,fill=white] (g) {$$}
;}\newcommand{\nodeh}{\node[draw,circle,fill=white] (h) {$$}
;}\begin{tikzpicture}[auto]
\matrix[column sep=0.1cm, row sep=.3cm,ampersand replacement=\&]{
         \&         \& \nodea  \&         \&         \\ 
 \nodeb  \&         \& \nodec  \&         \& \nodeh  \\ 
         \& \noded  \& \nodee  \& \nodeg  \&         \\ 
         \&         \& \nodef  \&         \&         \\
};

\path[ultra thick, red] (e) edge (f)
	(c) edge (d) edge (e) edge (g)
	(a) edge (b) edge (c) edge (h);
\end{tikzpicture}}}\end{aligned}
\end{equation*}
\begin{equation*}
\begin{aligned}\scalebox{.6}{{ \newcommand{\nodea}{\node[draw,circle,fill=white] (a) {$3$}
;}\newcommand{\nodeb}{\node[draw,circle,fill=white] (b) {$$}
;}\newcommand{\nodec}{\node[draw,circle,fill=red!50] (c) {$2$}
;}\newcommand{\noded}{\node[draw,circle,fill=red!50] (d) {$1$}
;}\newcommand{\nodee}{\node[draw,circle,fill=white] (e) {$$}
;}\newcommand{\nodef}{\node[draw,circle,fill=white] (f) {$$}
;}\newcommand{\nodeg}{\node[draw,circle,fill=white] (g) {$$}
;}\newcommand{\nodeh}{\node[draw,circle,fill=white] (h) {$$}
;}\begin{tikzpicture}[auto]
\matrix[column sep=0.1cm, row sep=.3cm,ampersand replacement=\&]{
         \&         \& \nodea  \&         \&         \\ 
 \nodeb  \&         \& \nodec  \&         \& \nodeh  \\ 
         \& \noded  \& \nodef  \& \nodeg  \&         \\ 
         \& \nodee  \&         \&         \&         \\
};

\path[ultra thick, red] (d) edge (e)
	(c) edge (d) edge (f) edge (g)
	(a) edge (b) edge (c) edge (h);
\end{tikzpicture}}}\end{aligned}
 \cap 
\begin{aligned}\scalebox{.6}{\input{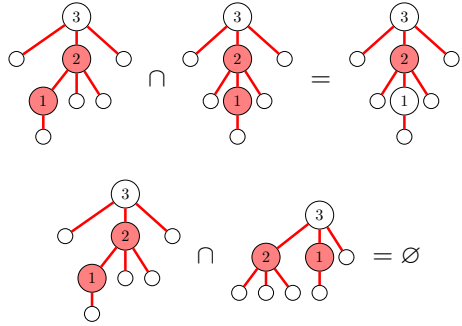}}\end{aligned}
= 
\varnothing
\end{equation*}
\caption{Example of intersections of pure intervals for $s=(0,2,2)$}
\label{fig:inter-pure}
\end{figure}

First, we prove that the intersection of two intervals is always an interval. This is actually a general result on lattices and the proof is immediate.

\begin{lemma}
\label{lem:lattice-interval-intersection}
Let $L$ be a lattice and $I_1 = [x_1, y_1]$, $I_2 = [x_2, y_2]$ two intervals of $L$. Then, $I_1 \cap I_2 \neq \emptyset$ if and only if $x_1 \join x_2 \leq y_1 \meet y_2$, and in this case we have $I_1 \cap I_2 = [x_1 \join x_2, y_1 \meet y_2]$. 
\end{lemma}

\begin{proof}
Suppose that $I_1 \cap I_2 \neq \emptyset$, then there is $x \in I_1 \cap I_2$. We have $x_1 \leq x$ and $x_2 \leq x$, which implies $x_1 \join x_2 \leq x$. Similarly, $x \leq y_1 \meet y_2$. We obtain $x_1 \join x_2 \leq x \leq y_1 \meet y_2$, \emph{i.e.}, $I_1 \cap I_2 \subseteq [x_1 \join x_2, y_1 \meet y_2]$. 

Now suppose $x_1 \join x_2 \leq y_1 \meet y_2$ and take $x \in [x_1 \join x_2, y_1 \meet y_2]$. By definition, we have $x_1 \leq x_1 \join x_2 \leq x \leq y_1 \meet y_2 \leq y_1$, so $x \in I_1$. Similarly, $x \in I_2$. This gives $I_1 \cap I_2 \neq \emptyset$ and $[x_1 \join x_2, y_1 \meet y_2] \subseteq I_1 \cap I_2$.
\end{proof}

So we know that by intersecting two pure intervals, we always obtain an interval. We will prove that the variations of the intersection satisfy the conditions of \Cref{thm:pure-interval-char}. For this we need to understand what are the variations and essential variations of the intersection. 

\subsection{Variations of the intersection}
As a first step, we prove the following lemma.

%\viviane{Viviane wants to add something here about a modification of intersection stability: 
%If two pure intervals intersect then 
%$$\inv(X)=\inv(T)\cup \inv(T')$$
%$$\inv(Y)=\inv(T+A)\cap \inv(T'+A')$$
%Question: If no transitivity is needed (i.e. these two equations are satisfied) and $X\leq Y$ can we deduce that the pure intervals intersect?
%}

\begin{lemma}[intersection stability]
\label{lem:intersection-stability}
Let $[T,T+A]$ and $[T',T'+A']$ be two pure intervals with a non-empty intersection. Let $X = T \join T'$. Suppose that there is $b > a$ with $\card_T(b,a) = \card_{T'}(b,a) = v$. Then $\card_X(b,a) = v$.
\end{lemma}

\begin{proof}
Let us suppose that there exist $a < b$ with $\card_T(b,a) = \card_{T'}(b,a) = v$ and $\card_X(b,a) > v$. We choose~$a$ such that $|b-a|$ is minimal. As the intersection of $[T,T+A]$ and $[T',T'+A']$ is non-empty, in particular, $X \in [T,T+A]$: the inversion $(b,a)$ varies in $[T,T+A]$ and using the $+1$ property of Lemma~\ref{lem:plus-one}, we obtain $\card_X(b,a) = v+1$. Let $S = \inv(T) \cup \inv(T')$ such that $\inv(X) = \tc{S}$. There is a transitivity path in $S$
\begin{equation*}
b = b_1 > \dots > b_k = a.
\end{equation*}
with $\card_S(b_1,b_2) = v+1$ and $\card_S(b_i,b_{i+1}) > 0$. We choose the transitivity path such that $k$ is minimal. Note that $k > 2$ because  $\card_T(b,a) = \card_{T'}(b,a) = v$.  Suppose that $b_3 \neq a$. If either $\card_T(b_1,b_3) = v + 1$ or $\card_{T'}(b_1,b_3) = v+1$, then $k$ is not minimal. By transitivity, we have that $\card_X(b_1,b_3) = v + 1$, and the +1 property gives us $\card_T(b_1,b_3) = \card_{T'}(b_1,b_3) = v$. This contradicts the minimality of $|b-a|$. 

We now have $b > b_2 > a$ a transitivity path in $S$ which does not exist either in $T$ nor $T'$. Without loss of generality and using the +1 property, we can assume that
\begin{align}
\card_T(b,b_2) &= v + 1 \\
\card_T(b_2, a) &= 0 \\
\card_{T'}(b,b_2) &= v \\
\card_{T'}(b_2, a) &= 1.
\end{align}
We have $X \in [T,T+A]$ and $\card_X(b_2,a) = 1 > \card_T(b_2,a)$, which means that $(b_2,a)$ varies in $[T,T+A]$. We also have $\card_X(b,a) > \card_T(b,a)$ by hypothesis, so $(b,a)$ varies in $[T,T+A]$. Using Condition~\eqref{cond:pure-interval-cb} of \Cref{thm:pure-interval-char}, we obtain that $(b,b_2)$ also varies in $[T,T+A]$ so $\card_{T+A}(b,b_2) = v+2$. By transitivity, this gives $\card_{T+A}(b,a) = v+2$ which contradicts the fact that $\card_T(b,a) = v$ by the $+1$ property. 
\end{proof}

\begin{proposition}[Variation intersection]
\label{prop:var-intersection}
Let $[T,T+A]$ and $[T',T'+A']$ be two pure intervals such that they have a non empty intersection $[X,Y]$ where $X = T \join T'$ and $Y = (T+A) \meet (T'+A')$, then 
\begin{equation}
\Var([X,Y]) = \Var([T,T+A]) \cap \Var([T',T'+A']).
\end{equation}

Note that the intersection is taken considering variations with their values, \emph{i.e.}, $(b,a)_v \in  \Var([T,T+A]) \cap \Var([T',T'+A'])$ if and only if $(b,a)_v \in  \Var([T,T+A])$ and $(b,a)_v \in  \Var([T',T'+A'])$.
\end{proposition}

\begin{proof}
Let $b > a$, be such that $(b,a)_v$ is a variation of both $[T,T+A]$ and $[T',T'+A']$. In particular, $\card_T(b,a) = \card_{T'}(b,a) = v$ and $\card_{T+A}(b,a) = \card_{T'+A'}(b,a) = v+1$. Using intersection stability of Lemma~\ref{lem:intersection-stability}, we obtain $\card_X(b,a) = v$. Now, let $S$ be the multi-set of inversions obtained by increasing by one the cardinality of $(b,a)$ in $\inv(T)$. By definition, $\inv(T) \subseteq \inv(T+A)$. Besides, as the intersection is non-empty, $T \wole T' + A'$ which means $\inv(T) \subseteq \inv(T' + A')$. We also have $\card_{T+A}(b,a) = \card_{T'+A'}(b,a) = v+1 = \card_S(b,a)$ and this gives us $S \subseteq \inv(T+A)$ and $S \subseteq \inv(T'+A')$. As $Y = (T+A) \meet (T'+A')$, this gives that $S \subseteq \inv(Y)$ and so $\card_Y(b,a) \geq v +1$: $(b,a)_v$ is a variation of $[X,Y]$.

Conversely, if $(b,a)$ does not vary in $T$. We have $\card_T(b,a) = \card_{T+A}(b,a) = v$ and because $T \leq X \leq Y \leq T +A$, this implies that $\card_X(b,a) = \card_Y(b,a) = v$ and $(b,a)$ is not a variation of $[X,Y]$. 
\end{proof}

Proposition~\ref{prop:var-intersection} can be summarized in a sentence: the variations of the intersection are the intersections of the variations. Note that this is true only if the intersection is non-empty. Besides, this does not extend to \emph{essential} variations. Indeed, in this case, the intersection of essential variations is only included in the set of essential variations of the intersection. 

We present two examples also computed in our {\tt SageMath} demo worksheet~\cite{SageDemoII}.

\begin{example}
Figure~\ref{fig:inter-pure} shows two examples of intersection of pure intervals for $s=(0,2,2)$. On the first one, the intersection is not empty. The variations of the two pure intervals are respectively $\lbrace (3,2)_1, (3,1)_1, (2,1)_0 \rbrace$ and $\lbrace (3,2)_1, (3,1)_1, (2,1)_1 \rbrace$. The variations of the intersection are $\lbrace (3,2)_1, (3,1)_1 \rbrace$. In this case, the intersection has only one essential variation, $(3,2)_1$ which is the only essential variation found in both pure intervals. 

The second intersection of Figure~\ref{fig:inter-pure} is an empty one but we see that it does not imply that the intersection of the variations is empty. Indeed, in this case $(3,1)_1$ is a variation in both pure intervals.
\end{example}

\begin{example}
Figure~\ref{fig:inter-pure-10} shows an intersection of pure intervals of size $10$. For clarity, we have written the pure intervals along with their maximal trees. You can check that the minimal tree of the intersection is the join of the two minimal trees and it is smaller than the meet of the two maximal trees. The variations of the intersection are $\lbrace (9, 4)_0, (9, 3)_0, (9, 1)_0, (10, 6)_2, (10, 2)_2, (6, 2)_1 \rbrace$. This is indeed the intersection of the variations in both intervals. However, the intersection has three essential variations: $\lbrace (9, 4)_0, (10, 6)_2, (6, 2)_1 \rbrace$. We see that $(9,4)_0$ and $(6,2)_1$ are essential variations in both pure intervals (they are the only ones present in both) but $(10,6)_2$ is not (it is an essential variation only in the second one). 
\end{example}

\begin{figure}
\begin{center}
\begin{tabular}{ccc}
\scalebox{.6}{{ \newcommand{\nodea}{\node[draw,circle] (a) {$10$}
;}\newcommand{\nodeb}{\node[draw,circle,fill=red!50] (b) {$8$}
;}\newcommand{\nodec}{\node[draw,circle,fill=red!50] (c) {$5$}
;}\newcommand{\noded}{\node[draw,circle] (d) {$$}
;}\newcommand{\nodee}{\node[draw,circle] (e) {$$}
;}\newcommand{\nodef}{\node[draw,circle] (f) {$$}
;}\newcommand{\nodeg}{\node[draw,circle,fill=red!50] (g) {$9$}
;}\newcommand{\nodeh}{\node[draw,circle] (h) {$7$}
;}\newcommand{\nodei}{\node[draw,circle] (i) {$$}
;}\newcommand{\nodej}{\node[draw,circle,fill=red!50] (j) {$4$}
;}\newcommand{\nodeba}{\node[draw,circle,fill=red!50] (ba) {$1$}
;}\newcommand{\nodebb}{\node[draw,circle] (bb) {$$}
;}\newcommand{\nodebc}{\node[draw,circle] (bc) {$3$}
;}\newcommand{\nodebd}{\node[draw,circle] (bd) {$$}
;}\newcommand{\nodebe}{\node[draw,circle] (be) {$$}
;}\newcommand{\nodebf}{\node[draw,circle] (bf) {$$}
;}\newcommand{\nodebg}{\node[draw,circle] (bg) {$$}
;}\newcommand{\nodebh}{\node[draw,circle,fill=red!50] (bh) {$6$}
;}\newcommand{\nodebi}{\node[draw,circle] (bi) {$$}
;}\newcommand{\nodebj}{\node[draw,circle,fill=red!50] (bj) {$2$}
;}\newcommand{\nodeca}{\node[draw,circle] (ca) {$$}
;}\newcommand{\nodecb}{\node[draw,circle] (cb) {$$}
;}\newcommand{\nodecc}{\node[draw,circle] (cc) {$$}
;}\newcommand{\nodecd}{\node[draw,circle] (cd) {$$}
;}\begin{tikzpicture}[auto]
\matrix[column sep=.3cm, row sep=.3cm,ampersand replacement=\&]{
         \&         \&         \&         \& \nodea  \&         \&         \&         \&         \&         \&         \\ 
         \& \nodeb  \&         \& \nodef  \&         \&         \&         \&         \& \nodeg  \&         \& \nodecd \\ 
         \& \nodec  \&         \&         \& \nodeh  \&         \&         \&         \& \nodebh \& \nodecc \&         \\ 
 \noded  \&         \& \nodee  \& \nodei  \&         \&         \& \nodej  \& \nodebi \& \nodebj \& \nodecb \&         \\ 
         \&         \&         \&         \& \nodeba \&         \& \nodebc \& \nodebg \& \nodeca \&         \&         \\ 
         \&         \&         \&         \& \nodebb \& \nodebd \& \nodebe \& \nodebf \&         \&         \&         \\
};

\path[ultra thick, red] (c) edge (d) edge (e)
	(b) edge (c)
	(ba) edge (bb)
	(bc) edge (bd) edge (be) edge (bf)
	(j) edge (ba) edge (bc) edge (bg)
	(h) edge (i) edge (j)
	(bj) edge (ca)
	(bh) edge (bi) edge (bj) edge (cb)
	(g) edge (h) edge (bh) edge (cc)
	(a) edge (b) edge (f) edge (g) edge (cd);
\end{tikzpicture}}} & 
\scalebox{.6}{{ \newcommand{\nodea}{\node[draw,circle] (a) {$10$}
;}\newcommand{\nodeb}{\node[draw,circle] (b) {$$}
;}\newcommand{\nodec}{\node[draw,circle] (c) {$8$}
;}\newcommand{\noded}{\node[draw,circle] (d) {$5$}
;}\newcommand{\nodee}{\node[draw,circle] (e) {$$}
;}\newcommand{\nodef}{\node[draw,circle] (f) {$$}
;}\newcommand{\nodeg}{\node[draw,circle] (g) {$7$}
;}\newcommand{\nodeh}{\node[draw,circle] (h) {$$}
;}\newcommand{\nodei}{\node[draw,circle] (i) {$$}
;}\newcommand{\nodej}{\node[draw,circle] (j) {$9$}
;}\newcommand{\nodeba}{\node[draw,circle] (ba) {$$}
;}\newcommand{\nodebb}{\node[draw,circle] (bb) {$4$}
;}\newcommand{\nodebc}{\node[draw,circle] (bc) {$$}
;}\newcommand{\nodebd}{\node[draw,circle] (bd) {$3$}
;}\newcommand{\nodebe}{\node[draw,circle] (be) {$1$}
;}\newcommand{\nodebf}{\node[draw,circle] (bf) {$$}
;}\newcommand{\nodebg}{\node[draw,circle] (bg) {$$}
;}\newcommand{\nodebh}{\node[draw,circle] (bh) {$$}
;}\newcommand{\nodebi}{\node[draw,circle] (bi) {$$}
;}\newcommand{\nodebj}{\node[draw,circle] (bj) {$6$}
;}\newcommand{\nodeca}{\node[draw,circle] (ca) {$$}
;}\newcommand{\nodecb}{\node[draw,circle] (cb) {$$}
;}\newcommand{\nodecc}{\node[draw,circle] (cc) {$2$}
;}\newcommand{\nodecd}{\node[draw,circle] (cd) {$$}
;}\begin{tikzpicture}[auto]
\matrix[column sep=.3cm, row sep=.3cm,ampersand replacement=\&]{
         \&         \&         \&         \& \nodea  \&         \&         \&         \&         \&         \&         \&         \\ 
 \nodeb  \&         \& \nodec  \&         \& \nodeg  \&         \&         \& \nodej  \&         \&         \&         \&         \\ 
         \&         \& \noded  \& \nodeh  \&         \& \nodei  \& \nodeba \& \nodebb \&         \&         \& \nodebj \&         \\ 
         \& \nodee  \&         \& \nodef  \&         \&         \& \nodebc \& \nodebd \& \nodebi \& \nodeca \& \nodecb \& \nodecc \\ 
         \&         \&         \&         \&         \&         \& \nodebe \& \nodebg \& \nodebh \&         \&         \& \nodecd \\ 
         \&         \&         \&         \&         \&         \& \nodebf \&         \&         \&         \&         \&         \\
};

\path[ultra thick, red] (d) edge (e) edge (f)
	(c) edge (d)
	(g) edge (h) edge (i)
	(be) edge (bf)
	(bd) edge (be) edge (bg) edge (bh)
	(bb) edge (bc) edge (bd) edge (bi)
	(cc) edge (cd)
	(bj) edge (ca) edge (cb) edge (cc)
	(j) edge (ba) edge (bb) edge (bj)
	(a) edge (b) edge (c) edge (g) edge (j);
\end{tikzpicture}}} \\
\scalebox{.6}{{ \newcommand{\nodea}{\node[draw,circle,fill=white] (a) {$10$}
;}\newcommand{\nodeb}{\node[draw,circle,fill=white] (b) {$5$}
;}\newcommand{\nodec}{\node[draw,circle,fill=white] (c) {$$}
;}\newcommand{\noded}{\node[draw,circle,fill=white] (d) {$$}
;}\newcommand{\nodee}{\node[draw,circle,fill=red!50] (e) {$8$}
;}\newcommand{\nodef}{\node[draw,circle,fill=red!50] (f) {$7$}
;}\newcommand{\nodeg}{\node[draw,circle,fill=white] (g) {$$}
;}\newcommand{\nodeh}{\node[draw,circle,fill=white] (h) {$$}
;}\newcommand{\nodei}{\node[draw,circle,fill=white] (i) {$9$}
;}\newcommand{\nodej}{\node[draw,circle,fill=red!50] (j) {$4$}
;}\newcommand{\nodeba}{\node[draw,circle,fill=white] (ba) {$$}
;}\newcommand{\nodebb}{\node[draw,circle,fill=red!50] (bb) {$3$}
;}\newcommand{\nodebc}{\node[draw,circle,fill=red!50] (bc) {$1$}
;}\newcommand{\nodebd}{\node[draw,circle,fill=white] (bd) {$$}
;}\newcommand{\nodebe}{\node[draw,circle,fill=white] (be) {$$}
;}\newcommand{\nodebf}{\node[draw,circle,fill=white] (bf) {$$}
;}\newcommand{\nodebg}{\node[draw,circle,fill=white] (bg) {$$}
;}\newcommand{\nodebh}{\node[draw,circle,fill=white] (bh) {$$}
;}\newcommand{\nodebi}{\node[draw,circle,fill=red!50] (bi) {$6$}
;}\newcommand{\nodebj}{\node[draw,circle,fill=white] (bj) {$$}
;}\newcommand{\nodeca}{\node[draw,circle,fill=red!50] (ca) {$2$}
;}\newcommand{\nodecb}{\node[draw,circle,fill=white] (cb) {$$}
;}\newcommand{\nodecc}{\node[draw,circle,fill=white] (cc) {$$}
;}\newcommand{\nodecd}{\node[draw,circle,fill=white] (cd) {$$}
;}\begin{tikzpicture}[auto]
\matrix[column sep=0.1cm, row sep=.3cm,ampersand replacement=\&]{
         \&         \&         \&         \&         \&         \& \nodea  \&         \&         \&         \&         \&         \&         \\ 
         \& \nodeb  \&         \&         \& \nodee  \&         \&         \&         \&         \& \nodei  \&         \&         \& \nodecd \\ 
 \nodec  \&         \& \noded  \&         \& \nodef  \&         \&         \& \nodej  \&         \& \nodebh \&         \& \nodebi \&         \\ 
         \&         \&         \& \nodeg  \&         \& \nodeh  \& \nodeba \& \nodebb \& \nodebg \&         \& \nodebj \& \nodeca \& \nodecc \\ 
         \&         \&         \&         \&         \&         \& \nodebc \& \nodebe \& \nodebf \&         \&         \& \nodecb \&         \\ 
         \&         \&         \&         \&         \&         \& \nodebd \&         \&         \&         \&         \&         \&         \\
};

\path[ultra thick, red] (b) edge (c) edge (d)
	(f) edge (g) edge (h)
	(e) edge (f)
	(bc) edge (bd)
	(bb) edge (bc) edge (be) edge (bf)
	(j) edge (ba) edge (bb) edge (bg)
	(ca) edge (cb)
	(bi) edge (bj) edge (ca) edge (cc)
	(i) edge (j) edge (bh) edge (bi)
	(a) edge (b) edge (e) edge (i) edge (cd);
\end{tikzpicture}}} & 
\scalebox{.6}{{ \newcommand{\nodea}{\node[draw,circle] (a) {$10$}
;}\newcommand{\nodeb}{\node[draw,circle] (b) {$5$}
;}\newcommand{\nodec}{\node[draw,circle] (c) {$$}
;}\newcommand{\noded}{\node[draw,circle] (d) {$$}
;}\newcommand{\nodee}{\node[draw,circle] (e) {$$}
;}\newcommand{\nodef}{\node[draw,circle] (f) {$9$}
;}\newcommand{\nodeg}{\node[draw,circle] (g) {$8$}
;}\newcommand{\nodeh}{\node[draw,circle] (h) {$7$}
;}\newcommand{\nodei}{\node[draw,circle] (i) {$$}
;}\newcommand{\nodej}{\node[draw,circle] (j) {$$}
;}\newcommand{\nodeba}{\node[draw,circle] (ba) {$4$}
;}\newcommand{\nodebb}{\node[draw,circle] (bb) {$$}
;}\newcommand{\nodebc}{\node[draw,circle] (bc) {$$}
;}\newcommand{\nodebd}{\node[draw,circle] (bd) {$3$}
;}\newcommand{\nodebe}{\node[draw,circle] (be) {$$}
;}\newcommand{\nodebf}{\node[draw,circle] (bf) {$1$}
;}\newcommand{\nodebg}{\node[draw,circle] (bg) {$$}
;}\newcommand{\nodebh}{\node[draw,circle] (bh) {$$}
;}\newcommand{\nodebi}{\node[draw,circle] (bi) {$$}
;}\newcommand{\nodebj}{\node[draw,circle] (bj) {$6$}
;}\newcommand{\nodeca}{\node[draw,circle] (ca) {$$}
;}\newcommand{\nodecb}{\node[draw,circle] (cb) {$$}
;}\newcommand{\nodecc}{\node[draw,circle] (cc) {$2$}
;}\newcommand{\nodecd}{\node[draw,circle] (cd) {$$}
;}\begin{tikzpicture}[auto]
\matrix[column sep=.3cm, row sep=.3cm,ampersand replacement=\&]{
         \&         \&         \&         \& \nodea  \&         \&         \&         \&         \&         \&         \&         \\ 
         \& \nodeb  \&         \& \nodee  \&         \&         \&         \& \nodef  \&         \&         \& \nodebj \&         \\ 
 \nodec  \&         \& \noded  \&         \& \nodeg  \&         \&         \& \nodeba \& \nodebi \& \nodeca \& \nodecb \& \nodecc \\ 
         \&         \&         \&         \& \nodeh  \&         \& \nodebb \& \nodebc \& \nodebd \&         \&         \& \nodecd \\ 
         \&         \&         \& \nodei  \&         \& \nodej  \&         \& \nodebe \& \nodebf \& \nodebh \&         \&         \\ 
         \&         \&         \&         \&         \&         \&         \&         \& \nodebg \&         \&         \&         \\
};

\path[ultra thick, red] (b) edge (c) edge (d)
	(h) edge (i) edge (j)
	(g) edge (h)
	(bf) edge (bg)
	(bd) edge (be) edge (bf) edge (bh)
	(ba) edge (bb) edge (bc) edge (bd)
	(f) edge (g) edge (ba) edge (bi)
	(cc) edge (cd)
	(bj) edge (ca) edge (cb) edge (cc)
	(a) edge (b) edge (e) edge (f) edge (bj);
\end{tikzpicture}}} \\ \hline
Intersection & \\
\scalebox{.6}{{ \newcommand{\nodea}{\node[draw,circle,fill=white] (a) {$10$}
;}\newcommand{\nodeb}{\node[draw,circle,fill=white] (b) {$5$}
;}\newcommand{\nodec}{\node[draw,circle,fill=white] (c) {$$}
;}\newcommand{\noded}{\node[draw,circle,fill=white] (d) {$$}
;}\newcommand{\nodee}{\node[draw,circle,fill=white] (e) {$8$}
;}\newcommand{\nodef}{\node[draw,circle,fill=white] (f) {$$}
;}\newcommand{\nodeg}{\node[draw,circle,fill=white] (g) {$9$}
;}\newcommand{\nodeh}{\node[draw,circle,fill=white] (h) {$7$}
;}\newcommand{\nodei}{\node[draw,circle,fill=white] (i) {$$}
;}\newcommand{\nodej}{\node[draw,circle,fill=red!50] (j) {$4$}
;}\newcommand{\nodeba}{\node[draw,circle,fill=white] (ba) {$$}
;}\newcommand{\nodebb}{\node[draw,circle,fill=white] (bb) {$3$}
;}\newcommand{\nodebc}{\node[draw,circle,fill=white] (bc) {$1$}
;}\newcommand{\nodebd}{\node[draw,circle,fill=white] (bd) {$$}
;}\newcommand{\nodebe}{\node[draw,circle,fill=white] (be) {$$}
;}\newcommand{\nodebf}{\node[draw,circle,fill=white] (bf) {$$}
;}\newcommand{\nodebg}{\node[draw,circle,fill=white] (bg) {$$}
;}\newcommand{\nodebh}{\node[draw,circle,fill=white] (bh) {$$}
;}\newcommand{\nodebi}{\node[draw,circle,fill=red!50] (bi) {$6$}
;}\newcommand{\nodebj}{\node[draw,circle,fill=white] (bj) {$$}
;}\newcommand{\nodeca}{\node[draw,circle,fill=red!50] (ca) {$2$}
;}\newcommand{\nodecb}{\node[draw,circle,fill=white] (cb) {$$}
;}\newcommand{\nodecc}{\node[draw,circle,fill=white] (cc) {$$}
;}\newcommand{\nodecd}{\node[draw,circle,fill=white] (cd) {$$}
;}\begin{tikzpicture}[auto]
\matrix[column sep=0.1cm, row sep=.3cm,ampersand replacement=\&]{
         \&         \&         \&         \& \nodea  \&         \&         \&         \&         \&         \&         \&         \\ 
         \& \nodeb  \&         \& \nodee  \&         \&         \&         \& \nodeg  \&         \&         \&         \& \nodecd \\ 
 \nodec  \&         \& \noded  \& \nodef  \& \nodeh  \&         \&         \& \nodebh \&         \& \nodebi \&         \&         \\ 
         \&         \&         \& \nodei  \&         \& \nodej  \&         \&         \& \nodebj \& \nodeca \& \nodecc \&         \\ 
         \&         \&         \&         \& \nodeba \& \nodebb \& \nodebg \&         \&         \& \nodecb \&         \&         \\ 
         \&         \&         \&         \& \nodebc \& \nodebe \& \nodebf \&         \&         \&         \&         \&         \\ 
         \&         \&         \&         \& \nodebd \&         \&         \&         \&         \&         \&         \&         \\
};

\path[ultra thick, red] (b) edge (c) edge (d)
	(e) edge (f)
	(bc) edge (bd)
	(bb) edge (bc) edge (be) edge (bf)
	(j) edge (ba) edge (bb) edge (bg)
	(h) edge (i) edge (j)
	(ca) edge (cb)
	(bi) edge (bj) edge (ca) edge (cc)
	(g) edge (h) edge (bh) edge (bi)
	(a) edge (b) edge (e) edge (g) edge (cd);
\end{tikzpicture}}} & 
\scalebox{.6}{{ \newcommand{\nodea}{\node[draw,circle] (a) {$10$}
;}\newcommand{\nodeb}{\node[draw,circle] (b) {$5$}
;}\newcommand{\nodec}{\node[draw,circle] (c) {$$}
;}\newcommand{\noded}{\node[draw,circle] (d) {$$}
;}\newcommand{\nodee}{\node[draw,circle] (e) {$8$}
;}\newcommand{\nodef}{\node[draw,circle] (f) {$$}
;}\newcommand{\nodeg}{\node[draw,circle] (g) {$9$}
;}\newcommand{\nodeh}{\node[draw,circle] (h) {$7$}
;}\newcommand{\nodei}{\node[draw,circle] (i) {$$}
;}\newcommand{\nodej}{\node[draw,circle] (j) {$$}
;}\newcommand{\nodeba}{\node[draw,circle] (ba) {$4$}
;}\newcommand{\nodebb}{\node[draw,circle] (bb) {$$}
;}\newcommand{\nodebc}{\node[draw,circle] (bc) {$3$}
;}\newcommand{\nodebd}{\node[draw,circle] (bd) {$1$}
;}\newcommand{\nodebe}{\node[draw,circle] (be) {$$}
;}\newcommand{\nodebf}{\node[draw,circle] (bf) {$$}
;}\newcommand{\nodebg}{\node[draw,circle] (bg) {$$}
;}\newcommand{\nodebh}{\node[draw,circle] (bh) {$$}
;}\newcommand{\nodebi}{\node[draw,circle] (bi) {$$}
;}\newcommand{\nodebj}{\node[draw,circle] (bj) {$6$}
;}\newcommand{\nodeca}{\node[draw,circle] (ca) {$$}
;}\newcommand{\nodecb}{\node[draw,circle] (cb) {$$}
;}\newcommand{\nodecc}{\node[draw,circle] (cc) {$2$}
;}\newcommand{\nodecd}{\node[draw,circle] (cd) {$$}
;}\begin{tikzpicture}[auto]
\matrix[column sep=.3cm, row sep=.3cm,ampersand replacement=\&]{
         \&         \&         \&         \& \nodea  \&         \&         \&         \&         \&         \&         \&         \\ 
         \& \nodeb  \&         \& \nodee  \&         \&         \&         \& \nodeg  \&         \&         \& \nodebj \&         \\ 
 \nodec  \&         \& \noded  \& \nodef  \& \nodeh  \&         \&         \& \nodeba \& \nodebi \& \nodeca \& \nodecb \& \nodecc \\ 
         \&         \&         \& \nodei  \&         \& \nodej  \& \nodebb \& \nodebc \& \nodebh \&         \&         \& \nodecd \\ 
         \&         \&         \&         \&         \&         \& \nodebd \& \nodebf \& \nodebg \&         \&         \&         \\ 
         \&         \&         \&         \&         \&         \& \nodebe \&         \&         \&         \&         \&         \\
};

\path[ultra thick, red] (b) edge (c) edge (d)
	(e) edge (f)
	(h) edge (i) edge (j)
	(bd) edge (be)
	(bc) edge (bd) edge (bf) edge (bg)
	(ba) edge (bb) edge (bc) edge (bh)
	(g) edge (h) edge (ba) edge (bi)
	(cc) edge (cd)
	(bj) edge (ca) edge (cb) edge (cc)
	(a) edge (b) edge (e) edge (g) edge (bj);
\end{tikzpicture}}}
\end{tabular}
\end{center}
\caption{Example of an intersection of pure intervals in size $10$}
\label{fig:inter-pure-10}
\end{figure}
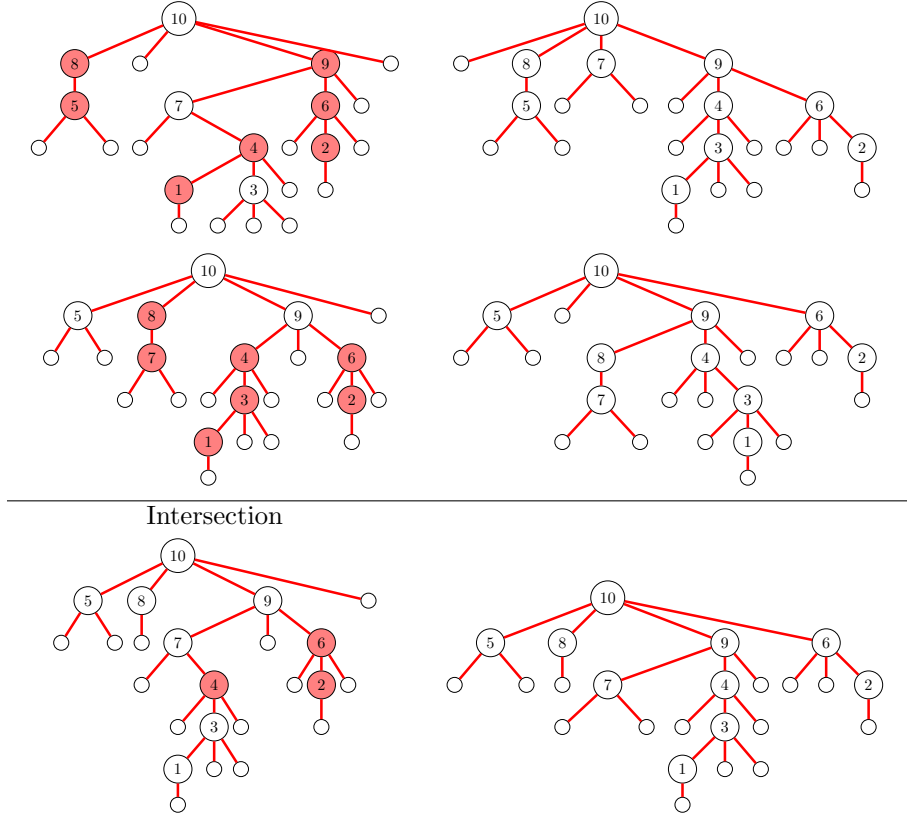

\subsection{Essential variations of the intersection}
To prove that the essential variations satisfy the conditions of \Cref{thm:pure-interval-char}, we need to characterize the variations of the two intervals that give essential variations in the intersection. We will do this using the following compatibility notion.

\begin{definition}
\label{def:compatible-var}
Let $[T,T+A]$ and $[T',T'+A']$ be two pure intervals with a non-empty intersection~$[X,Y]$. We say that $(c,a)_v$ is a \defn{compatible} variation of the intersection if $(c,a)_v$ is a variation in both $[T,T+A]$ and $[T',T'+A']$ and for all $b$ with $c > b > a$ such that $a$ is a middle descendant of $b$ in~$T'$ (resp. $T$) then $\card_T(b,a) = s(b)$ (resp. $\card_{T'}(b,a) = s(b)$).
\end{definition}

Note that if $(c,a)_v$ is an essential variation of both $[T,T+A]$ and $[T',T'+A']$ then it is compatible because by Property~\eqref{case-prop:ev-no-middle} of Proposition~\ref{prop:middle-b}, $a$ is never a middle descendant of any $b$, $c > b > a$ in neither $T$ nor $T'$.

\begin{proposition}
\label{prop:compatible-var}
Let $[T,T+A]$ and $[T',T'+A']$ be two pure intervals with a non-empty intersection~$[X,Y]$. Then $(c,a)_v$ is an essential variation of $[X,Y]$ if and only if, $(c,a)_v$ is a compatible variation of the intersection.
\end{proposition}

\begin{example}
Let us call $T$ and $T'$ respectively the two minimal trees of the pure intervals shown in Figure~\ref{fig:inter-pure-10}.
The variation $(10,6)_2$ is the only compatible variation which is not an essential variation in both $T$ and $T'$: in $T'$, $(10,6)_2$ is an essential variation so $6$ is never a middle descendant of any $b$ with $10 > b > 6$. In $T$, $6$ is a middle descendant of $9$ and we check that $\card_{T'}(9,6) = 2 = s(9)$. It is indeed an essential variation of the intersection.

On the other hand, $(9,1)_0$ is an essential variation of $T$ and a variation in $T'$. In $T'$, $1$ is a middle descendant of $3$ and $4$ but we have $\card_T(3,1) = 0 < s(3)$ and $\card_T(4,1) = 0 < s(4)$ so it is not a compatible variation. Indeed, we see that $(9,1)_0$ is a variation of the intersection but not an essential variation. 
\end{example}

\begin{proof}
Let us first suppose that a variation $(c,a)_v$ is compatible for the intersection and show that it is an essential variation of the intersection~$[X,Y]$. By definition, it is a variation of both $[T,T+A]$ and $[T',T'+A']$ so by Proposition~\ref{prop:var-intersection}, it is a variation of $[X,Y]$. Now, take $b$ with~$c > b > a$. If $\card_T(b,a) = \card_{T'}(b,a) = 0$, then by intersection stability (Lemma~\ref{lem:intersection-stability}), $\card_X(b,a) = 0$. Now if $a$ is a middle descendant of $b$ in $T'$, the compatibility implies that $\card_T(b,a) = s(b)$ and so $\card_X(b,a) = s(b)$. Similarly, if $a$ is a middle descendant of $b$ in $T$, the compatibility gives $\card_X(b,a) \geq \card_{T'}(b,a) = s(b)$. So $a$ is never a middle descendant of $b$ in $X$ which implies that $(c,a)_v$ is an essential variation.

We want to show the reverse implication, namely: the essential variations consists only of the compatible variations. Let us choose $(c,a)_v$ such that it is an essential variation of $[X,Y]$. Proposition~\ref{prop:var-intersection} tells us that $(c,a)_v$ is a variation for both $[T,T+A]$ and $[T',T'+A']$. Let us suppose that there exists $b$, with $c > b > a$ and $a$ is a middle descendant of $b$ in $T'$. Remember that by Statement~\eqref{case-prop:middle-var} of Proposition~\ref{prop:middle-b}, this means that $(c,b)_v$ varies in $[T',T'+A']$.

Suppose first that $a$ is still a middle descendant of $b$ in $X$. We know that $\card_X(c,a) = v$. As $a$ is a descendant of $b$, this implies that $\card_X(c,b) = v$. We know $\card_T(c,b) \leq \card_X(c,b)$. If $\card_T(c,b) < v$, as $\card_T(c,a) = v$ this implies in particular that $\card_T(b,a) = s(b)$ which is not possible as we have $\card_X(b,a) < s(b)$, so $\card_T(c,b) = v$. As $\card_T(b,a) < s(b)$, we can apply Property~\eqref{case-prop:a-before-b} of Proposition~\ref{prop:middle-b} and we get that $(c,b)_v$ is a variation of $[T,T+A]$. As $(c,b)_v$ is also a variation of $[T',T'+A']$, using Proposition~\ref{prop:var-intersection}, we obtain that $(c,b)_v$ is a variation of $[X,Y]$ and then $(c,a)_v$ is not an essential variation of $[X,Y]$.

This means that $a$ is no longer a middle descendant of $b$ in $X$. As the cardinality of $(b,a)$ can only increase, this means $\card_X(b,a) = s(b)$. By the $+1$ property (Lemma~\ref{lem:plus-one}), $\card_{T'}(b,a) = s(b) - 1$ and $s(b) - 1 \leq \card_T(b,a) \leq s(b)$. But if $\card_T(b,a) = s(b) - 1$, the intersection stability (Lemma~\ref{lem:intersection-stability}) tells us that $\card_X(b,a) = s(b) - 1$ which is not possible. We have proved that if $a$ is a middle descendant of $b$ in $T'$, then $\card_T(b,a) = s(b)$. Symmetrically, we prove that if $a$ is a middle descendant of $b$ in $T$, then $\card_{T'}(b,a) = s(b)$.
\end{proof}

\begin{remark}
\label{rem:ba-inter-variation}
Suppose that $(c,a)$ is an essential variation of an intersection $[X,Y]$ of two pure intervals $[T,T+A]$ and $[T',T'+A']$, then for all $b$ such that $a$ is a middle descendant of $b$ in $T$ with $c > b > a$, we have $\card_{T'}(b,a) = s(b)$ whereas $\card_T(b,a) < s(b)$. This means in particular that $\card_X(b,a) = s(b)$ and implies
\begin{enumerate}
\item by the $+1$ property $\card(b,a)$ can vary by at most $1$ and so $\card_T(b,a) = s(b) - 1$ ($a$ is in the before last subtree of $b$);
\item $(b,a)_{s(b) - 1}$ is a variation of $[T,T+A]$. 
\end{enumerate}
\end{remark}

\begin{proposition}
\label{prop:union-essential-var}
Let $[T,T+A]$ and $[T',T'+A']$ be two pure intervals with a non-empty intersection~$[X,Y]$. If $(c,a)_v$ is an essential variation of $[X,Y]$, then $(c,a)_v$ is an essential variation of either $[T,T+A]$ or $[T',T'+A']$.
\end{proposition}

\begin{proof}
We need to prove that if $(c,a)_v$ is compatible, then it is an essential variation for one of the intervals. Suppose that it is not. We have $b$ and $b'$ such that $a$ is a middle descendant of $b$ in $T$ and $b'$ in $T'$. The compatibility implies that $\card_T(b',a) = s(b')$ and $\card_{T'}(b,a) = s(b)$ so in particular $b \neq b'$. Beside, we can chose $b$ (resp. $b'$) to be minimal, \emph{i.e.}, $a$ is not a middle descendant in $T$ (resp. $T')$ of any node of value smaller than $b$ (resp. $b'$). Without loss of generality, we suppose $b' < b$.  In $T'$, we have $\card_{T'}(b,a) = s(b)$ and $a$ is a descendant of $b'$ so $\card_{T'}(b,b') = s(b)$. Using Remark~\ref{rem:ba-inter-variation}, $\card_T(b,a) = s(b) - 1$. Now the compatibility gives us $\card_T(b',a) = s(b')$ which implies that $\card_T(b,b') \leq \card_T(b,a)$. By the $+1$ property, we obtain $\card_T(b,b') = s(b) - 1$. Moreover, both $(b,a)$ and $(b,b')$ vary as their cardinality increases in $T'$. They are essential variations because we have chosen $b$ to be minimal. By Condition~\eqref{cond:pure-interval-ba} of pure intervals on $b > b' > a$, as $s(b') > 0$, then $\card_T(b',a) = 0$ which contradicts $\card_T(b',a) = s(b')$.
\end{proof}

\subsection{Proof of the Intersection Theorem}
Now that we characterized the variations and the essential variations of the intersection, we are ready to prove the Intersection Theorem~\ref{thm:intersection}.

\begin{proof}[Proof of Theorem~\ref{thm:intersection}]
Let $[T,T+A]$ and $[T',T'+A']$ be two pure intervals with a non-empty intersection~$[X,Y]$. We prove that the variations and essential variations of $[X,Y]$ satisfy the conditions of \Cref{thm:pure-interval-char}.

We start with Condition~\eqref{cond:pure-interval-cb}. We take $c > b > a$ such that $(c,a)_v$ and $(b,a)_w$ are variations of $[X,Y]$. By \Cref{prop:var-intersection}, the variations of $[X,Y]$ are the intersection of the variations of $[T, T+A]$ and $[T', T'+A']$. So $(c,a)_v$ and $(b,a)_w$ are variations in both $[T,T+A]$ and $[T',T+A']$. These are pure intervals which implies that $(c,b)_v$ is a variation in both. By \Cref{prop:var-intersection}, $(c,b)_v$ is a variation of~$[X,Y]$.

We now prove Condition~\eqref{cond:pure-interval-ba}. We take $c > b > a$ such that $(c,a)_v$ and $(c,b)_v$ are essential variations of $[X,Y]$ and suppose $s(b) \neq 0$. By \Cref{prop:compatible-var}, $(c,a)_v$ and $(c,b)_v$ are compatible variations of $[T, T+A]$ and $[T', T'+A']$. We need to show that $(b,a)_0$ is a variation of $[X,Y]$, \emph{i.e.}, that it is a variation in both  $[T, T+A]$ and $[T', T'+A']$.

Let us first prove that $\card_T(b,a) = \card_{T'}(b,a) = 0$. We have proved in Proposition~\ref{prop:union-essential-var} that $(c,a)_v$ is an essential variation of at least one of the intervals. Without loss of generality, we suppose that $(c,a)_v$ is an essential variation of $[T,T+A]$. Using Statement~\eqref{case-prop:ca-ev-ba-0} of Proposition~\ref{prop:middle-b}, we obtain that $\card_T(b,a) = 0$. Suppose that $\card_{T'}(b,a) > 0$. If $a$ is a middle descendant of $b$, then as $(c,a)_v$ is compatible, this means $\card_T(b,a) = s(b)$ which contradicts our previous conclusion that $\card_T(b,a) = 0 < s(b)$. So $\card_{T'}(b,a) = s(b)$. We use Statement~\eqref{case-prop:a-middle-desc} of Proposition~\ref{prop:middle-b}: $a$ is a middle descendant of some $b'$ in $T'$ with $c > b' >b$. We choose $b'$ to be minimal so that $a$ is not a middle descendant of any other node of value between $b$ and $b'$. As $\card_{T'}(b,a) = s(b)$, we have $\card_{T'}(b',b) \leq \card_{T'}(b',a) < s(b')$. As $(c,a)$ is a compatible variation, in $T$ we have $\card_T(b',a) = s(b')$ and as $\card_T(b,a) = 0$ this gives $\card_T(b',b) \geq \card_T(b',a) = s(b')$. So $(b',b)$ is a variation of $[T',T'+A']$ as well as $(b',a)$. We can apply again Statement~\eqref{case-prop:a-middle-desc} on $b' > b > a$ which contradicts the minimality of $b'$. 

We prove now that $(b,a)$ varies in $[T,T+A]$. If $(c,b)_v$ is an essential variation, then it is the case by Statement~\eqref{case-prop:cb-ev-ba-var} of Proposition~\ref{prop:middle-b}. If not, then there exists $b'$ with $c > b' > b$ and $b$ is a middle descendant of $b'$ in $T$. We choose $b'$ to be minimal, so $b$ is not a middle descendant of any other node of value between $b$ and $b'$. Because $(c,a)$ is a compatible variation, this implies that $\card_{T'}(b',b) = s(b')$ and, moreover, by Proposition~\ref{prop:union-essential-var} $(c,b)_v$ is an essential variation of $[T',T'+A']$. Using Statement~\eqref{case-prop:cb-ev-ba-var}, we obtain that $(b,a)$ varies in $T'$. In particular, $a$ is a descendant of $b$ in $T'$ and $\card_{T'}(b',a) = s(b')$. In $T$, we have $\card_T(b,a) = 0$, and so $\card_T(b',a) \leq \card_T(b',b) < s(b')$. These cardinalities are higher in $T'$ so it implies that both $(b',a)$ and $(b',b)$ are variations of $[T,T+A]$ and by the $+1$ property, we have $\card_T(b',b) = \card_T(b',a) = s(b') - 1$. Besides, the minimality of $b'$ gives that $(b',b)$ is an essential variation and by Statement~\eqref{case-prop:cb-ev-ba-var} of the middle variations Proposition~\ref{prop:middle-b} on $b'>b>a$, we obtain that $(b,a)$ varies.

In this last proof, we have not made any assumptions on the variations of $[T,T+A]$ beside being compatible, especially we have not assumed that $(c,a)$ was an essential variation. So the proof applies in the same way to $[T',T'+A']$ and $(b,a)_0$ is also a variation of $[T',T'+A']$. This gives that $(b,a)_0$ is indeed a variation of $[X,Y]$.
\end{proof}

\begin{corollary}
\label{cor:intersection}
Let $[T,T+A]$ and $[T',T'+A']$ be two pure intervals with non-empty intersection. Then their intersection is the pure interval $[T'',T''+A'']$ where $T''=T \join T'$ and $A''$ is the set of couples $(a,c)$ such that $(c,a)$ is a minimal compatible variation of the two intervals, \emph{i.e.}, there is no $b$ with $a < b < c$ such that $(b,a)$ is a compatible variation.
\end{corollary}

\begin{proof}
This is a consequence of Proposition~\ref{prop:pure-candidate-to-pure} and Proposition~\ref{prop:compatible-var}.
\end{proof}

\part{The $s$-associahedron}
\label{part_two}

\section{The $s$-Tamari lattice (background)}

We recall the definition of the $s$-Tamari lattice and its connection with the $\nu$-Tamari lattice of Pr\'eville-Ratelle and Viennot, as shown in~\cite{CP22}. We refer to~\cite{CP22} for more detailed explanations. 

\subsection{The $s$-Tamari lattice}
\begin{definition}
\label{def:s-tam-tree}
An $s$-decreasing tree $T$ is called an \defn{$s$-Tamari tree} if for any $a<b<c$ the number of~$(c,a)$ inversions is less than or equal to the number of~$(c,b)$ inversions: 
\[
\card_T(c,a)\leq \card_T(c,b).
\]
In terms of the tree this means that the node labels in $T^c_i$ are smaller than all the labels in $T^c_j$ for $i<j$.
The multi set of inversions of an $s$-Tamari tree is called an \defn{$s$-Tamari inversion set}.

The \defn{$s$-Tamari poset} is the restriction of the $s$-weak order to the set of $s$-Tamari trees. 
\end{definition}

The $s$-Tamari trees play the role of $231$-avoiding permutations of the $s$-weak order, and the number of $s$-Tamari trees may be regarded as the \defn{$s$-Catalan number}. 

An example of the $s$-Tamari lattice for $s=(0,2,2)$ is illustrated on the left of Figure~\ref{fig:bij-stam-nutrees}. 
The $s$-Catalan number in this case is 12 (a Fuss-Catalan number), which counts the number of elements of the lattice. 

\begin{theorem}[{\cite[Theorem~2.2]{CP22}}]
\label{thm:stam-sublattice}
The $s$-Tamari poset is a sublattice of the $s$-week order. In particular, it is a lattice.
\end{theorem}

\begin{theorem}[{\cite[Theorem~2.20]{CP22}}]\label{thm_sTam_quotientLattice}
If $s$ contains no zeros (except at the first position which is irrelevant), then the $s$-Tamari lattice is a quotient lattice of the $s$-weak order.
\end{theorem}

The cover relations of the $s$-Tamari lattice can be described as certain rotations on $s$-Tamari trees. 

\begin{definition}
\label{def:Tamari-ascent}
Let $T$ be an $s$-Tamari tree of some weak composition $s$. We say that~$(a,c)$ with $a < c$ is a \defn{Tamari-ascent} of $T$ if $a$ is a non-right child of $c$.
It was shown in~\cite[Lemma~2.23]{CP22} that if~$(a,c)$ is a Tamari-ascent of $T$ then by increasing the cardinality of $(c,a)$ by one in $\inv(T)$ and taking the transitive closure, we obtain an $s$-Tamari inversion set. This operation is called an \defn{$s$-Tamari rotation} of $T$.
\end{definition}

\begin{theorem}[{\cite[Theorem~2.25]{CP22}}]\label{thm_sTamari_covering_relations}
The cover relations of the $s$-Tamari lattice are in correspondence with $s$-Tamari rotations.
\end{theorem}

\subsection{The $\nu$-Tamari lattice}
The $\nu$-Tamari lattice is another lattice structure introduced in~\cite{PrevilleRatelleViennot}, which can be defined in terms of a family of combinatorial object called $\nu$-trees~\cite{ceballos_vtamarivtrees_2018}.

Let $\nu$ be a lattice path on the plane consisting of a finite number of north and east unit steps, which are represented by the letters $N$ and $E$.
We denote by $A_\nu$ the set of lattice points weakly above $\nu$ inside the smallest rectangle containing $\nu$. We say that two points $p,q \in A_\nu$ are \defn{$\nu$-incompatible} if and only if $p$ is southwest or northeast to $q$ and the south-east corner of the smallest rectangle containing $p$ and $q$ is weakly above~$\nu$. Otherwise, we say that $p$ and $q$ are \defn{$\nu$-compatible}.

\begin{definition}
A $\nu$-tree is a maximal collection of pairwise $\nu$-compatible elements in $A_\nu$.
\end{definition}

Although $\nu$-trees are just sets of lattice points above $\nu$, they can be regarded as planar binary trees in the classical sense, by connecting consecutive nodes (or lattice points) in the same row or column~\cite{ceballos_vtamarivtrees_2018}. 
The result is a planar binary tree with a root a the top left corner of~$A_\nu$.

Let $\vT$ be a $\nu$-tree and $p,r\in \vT$ be two elements which do not lie in the same row or same column. We denote by $p\square r$ the smallest rectangle containing~$p$ and~$r$, and write~$p\llcorner r$~(resp. $p\urcorner r$) for the lower left corner (resp. upper right corner) of~$p\square r$.  

Let $p,q,r\in \vT$ be such that $q=p\llcorner r$ and no other elements in $\vT$ besides $p,q,r$ lie in~$p\square r$. 
The node~$q$ is called a \defn{$\nu$-ascent} of $T$, 
and the \defn{$\nu$-tree rotation} of $\vT$ at $q$ is defined as the set $\vT'=\bigl(\vT\setminus\{q\})\cup\{q'\}$, where~\mbox{$q'=p\urcorner r$}.  
    As shown in~\cite[Lemma~2.10]{ceballos_vtamarivtrees_2018}, the the rotation of a $\nu$-tree is also a $\nu$-tree. 

\begin{definition}
The \defn{$\nu$-Tamari lattice} is the poset on $\nu$-trees whose cover relations are given by $\nu$-tree rotations.  
\end{definition}

An example of this lattice for $\nu=NEENEEN$ is shown on the right of Figure~\ref{fig:bij-stam-nutrees}.

\subsection{The $s$-Tamari lattice and the $\nu$-Tamari lattice are isomorphic}\label{sec_sTamarivTamari}
Let $s=(s(1),\dots,s(n))$ be a weak composition and~$\nu(s)$ be the lattice path $\nu(s):=NE^{s(1)}\dots NE^{s(n)}$. We denote by $\reverse s$ the reversed sequence $(s(n),\dots,s(2),s(1))$. 
In~\cite{CP22}, we showed that the $s$-Tamari lattice and the~$\nu(\reverse s)$-Tamari lattice are isomorphic. The bijection relating these two lattices is given as follows.

Let $T$ be an $s$-Tamari tree. We use the reverse preorder transversal to label the nodes of $T$ with the numbers $0,1,\dots ,n$, such that a node $x$ has label $i$ if the number of internal nodes traversed striktly before $x$ is equal to $i$. We denote by $\streestovtrees(T)$ the unique $\nu$-tree containing as many nodes at height~$i$ as there are nodes in $T$ with label $i$. Such $\nu$-tree can be uniquely constructed using the right flushing algorithm from~\cite{ceballos_vtamarivtrees_2018}:
denote by $h_i$ the number of nodes that need to be added at height~$i$. We add these nodes from bottom to top, from right to left, avoiding forbidden positions. The forbidden positions are those above a node that is not the left most node in a row. 
See Figure~\ref{fig:bij-stam-nutrees} for an illustration.

\begin{figure}[htbp]
\begin{center}
\begin{tabular}{cc}
\includegraphics[height= 8.5cm]{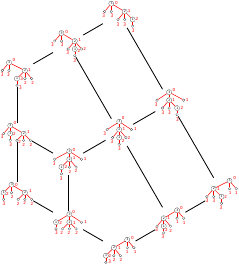}
&
\includegraphics[height= 8.5cm]{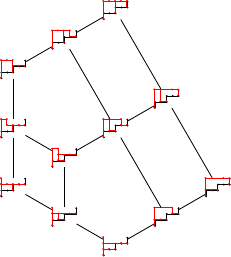}
\end{tabular}
\caption{Bijection between the $s$-Tamari lattice and the~\mbox{$\nu(\protect\reverse{s})$-Tamari} lattice for $s=(0,2,2)$. (Figure 16 of~\cite{CP22})}
\label{fig:bij-stam-nutrees}
\end{center}
\end{figure}

\begin{proposition}[{\cite[Corollary 2.34]{CP22}}]\label{prop_sTam_vTam_trees}
The map $\streestovtrees$ is an isomorphism between the $s$-Tamari lattice and the~$\nu(\reverse s)$-Tamari lattice.
\end{proposition}

\section{The $s$-associahedron}
In this section we extend the $s$-Tamari lattice to a full polyhedral complex that we call the $s$-associahedron, and show that it coincides with the $\nu$-associahedron introduced by Ceballos, Padrol, and Sarmiento in~\cite{CeballosPadrolSarmiento-geometryNuTamari}. 

\subsection{The $s$-associahedron}

Let $T$ be an $s$-Tamari tree and $A$ be a (possibly empty) subset of Tamari ascents of $T$.  
Similarly as in the case of the $s$-weak order, 
we denote by $T+A$ the $s$-Tamari tree whose inversion set is obtained by increasing all cardinality $\card_T(c,a)$ where $(a,c) \in A$ by one  and then taking the transitive closure. 

\begin{definition}
\label{def_pure_sTamari_intervals}
Let $T_1 \wole T_2$ be two $s$-Tamari trees for a given weak composition $s$. We say that the interval $[T_1, T_2]$ is a \defn{pure $s$-Tamari interval} if $T_2 = T_1 + A$ with $A$ a subset of Tamari-ascents of $T_1$. 
\end{definition}

Examples of pure $s$-Tamari intervals are shown in Figures~\ref{fig_pure_sTamari_intervals_one_two} and~\ref{fig_pure_sTamari_intervals_three}.

\begin{figure}[ht]
\includegraphics[width=0.4\textwidth]{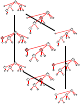}
\includegraphics[width=0.55\textwidth]{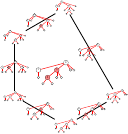}
\caption{Examples of 2-dimensional pure $s$-Tamari intervals.}
\label{fig_pure_sTamari_intervals_one_two}
\end{figure}

\begin{figure}[h!]
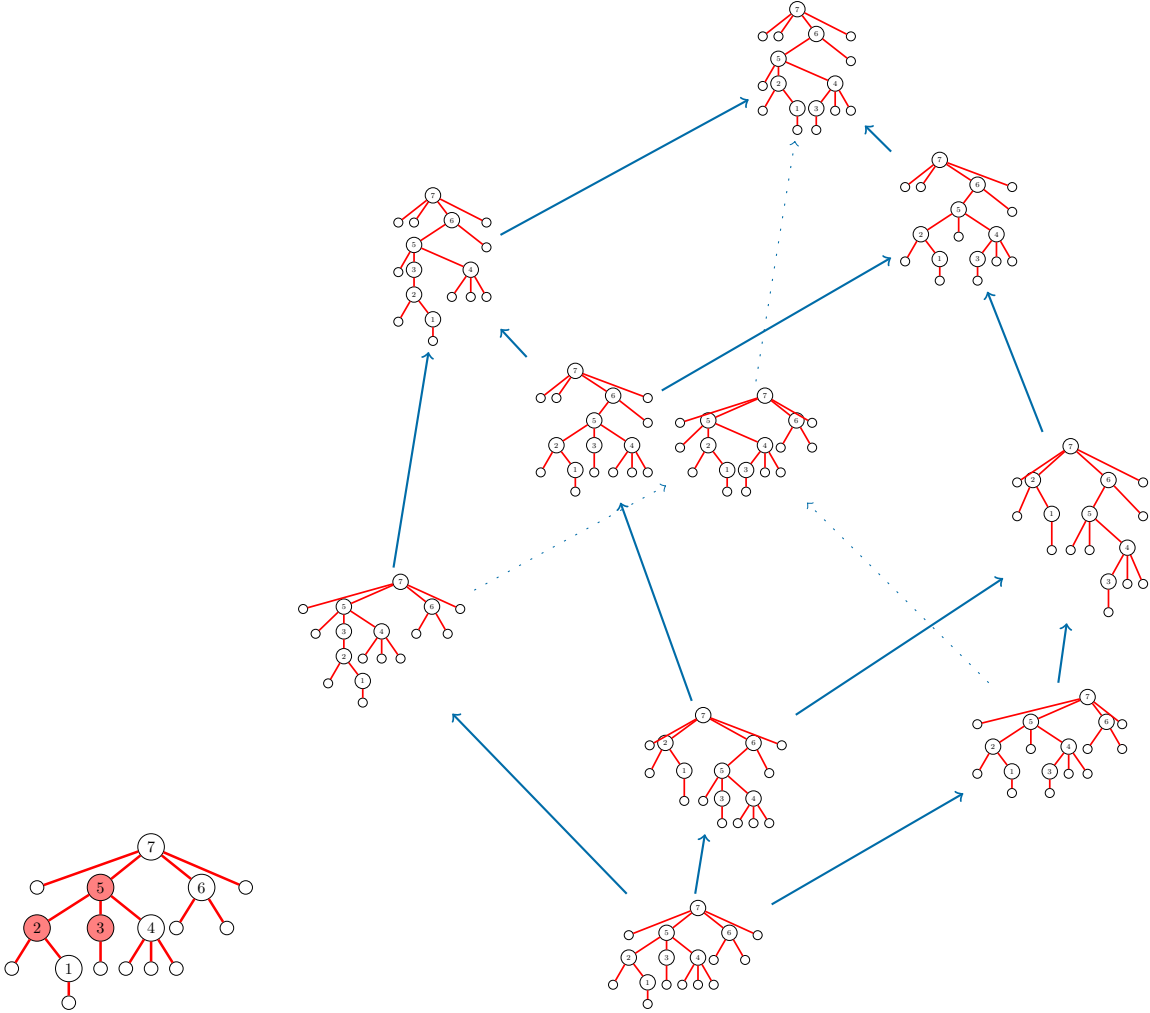

\begin{tabular}{cc}
\scalebox{.6}{{ \newcommand{\nodea}{\node[draw,circle] (a) {$7$}
;}\newcommand{\nodeb}{\node[draw,circle] (b) {$$}
;}\newcommand{\nodec}{\node[draw,circle, fill=red!50] (c) {$5$}
;}\newcommand{\noded}{\node[draw,circle, fill=red!50] (d) {$2$}
;}\newcommand{\nodee}{\node[draw,circle] (e) {$$}
;}\newcommand{\nodef}{\node[draw,circle] (f) {$1$}
;}\newcommand{\nodeg}{\node[draw,circle] (g) {$$}
;}\newcommand{\nodeh}{\node[draw,circle, fill=red!50] (h) {$3$}
;}\newcommand{\nodei}{\node[draw,circle] (i) {$$}
;}\newcommand{\nodej}{\node[draw,circle] (j) {$4$}
;}\newcommand{\nodeba}{\node[draw,circle] (ba) {$$}
;}\newcommand{\nodebb}{\node[draw,circle] (bb) {$$}
;}\newcommand{\nodebc}{\node[draw,circle] (bc) {$$}
;}\newcommand{\nodebd}{\node[draw,circle] (bd) {$6$}
;}\newcommand{\nodebe}{\node[draw,circle] (be) {$$}
;}\newcommand{\nodebf}{\node[draw,circle] (bf) {$$}
;}\newcommand{\nodebg}{\node[draw,circle] (bg) {$$}
;}\begin{tikzpicture}[every node/.style={inner sep = 3pt}]
\matrix[column sep=.1cm, row sep=.3cm,ampersand replacement=\&]{
         \&         \&         \&         \&         \& \nodea  \&         \&         \&         \&         \\ 
         \&  \nodeb \&         \& \nodec  \&         \&         \&         \& \nodebd \&         \& \nodebg \\ 
         \& \noded  \&         \& \nodeh  \&         \& \nodej  \& \nodebe \&         \& \nodebf \&         \\ 
 \nodee  \&         \& \nodef  \& \nodei  \& \nodeba \& \nodebb \& \nodebc \&         \&         \&         \\ 
         \&         \& \nodeg  \&         \&         \&         \&         \&        \&         \&         \\
};

\path[ultra thick, red] (f) edge (g)
	(d) edge (e) edge (f)
	(h) edge (i)
	(j) edge (ba) edge (bb) edge (bc)
	(c) edge (d) edge (h) edge (j)
	(bd) edge (be) edge (bf)
	(a) edge (b) edge (c) edge (bd) edge (bg);
\end{tikzpicture}}} &
\begin{tikzpicture}%
	[x={(-5cm, 5cm)},
	y={(1cm, 6cm)},
	z={(12cm, 7cm)},
	scale=0.4,
	back/.style={loosely dotted, thin},
	edge/.style={color=blue!95!black, thick},
	facet/.style={fill=red!95!black,fill opacity=0.800000},
	vertex/.style={inner sep=1pt,anchor=base},
	vertexback/.style={inner sep=1pt,anchor=base}]
\scriptsize

\coordinate (frontpentagon1) at (0,0,0);
\coordinate (frontpentagon2) at (0,1,0);
\coordinate (frontpentagon3) at (1,2,0);
\coordinate (frontpentagon4) at (2,2,0);
\coordinate (frontpentagon5) at (2,0,0);

\coordinate (backpentagon1) at (0,0,1);
\coordinate (backpentagon2) at (0,1,1);
\coordinate (backpentagon3) at (1,2,1);
\coordinate (backpentagon4) at (2,2,1);
\coordinate (backpentagon5) at (2,0,1);

\node[vertex] (nfront1) at (frontpentagon1) {
\scalebox{.4}{\input{figures/dtrees/prism_n1}}
};
\node[vertex] (nfront2) at (frontpentagon2) {
\scalebox{.4}{\input{figures/dtrees/prism_n2}}
};
\node[vertex] (nfront3) at (frontpentagon3) {
\scalebox{.4}{\input{figures/dtrees/prism_n3}}
};
\node[vertex] (nfront4) at (frontpentagon4) {
\scalebox{.4}{\input{figures/dtrees/prism_n4}}
};
\node[vertex] (nfront5) at (frontpentagon5) {
\scalebox{.4}{\input{figures/dtrees/prism_n5}}
};

\node[vertex] (nback1) at (backpentagon1) {
\scalebox{.4}{\input{figures/dtrees/prism_b1}}
};
\node[vertex] (nback2) at (backpentagon2) {
\scalebox{.4}{\input{figures/dtrees/prism_b2}}
};
\node[vertex] (nback3) at (backpentagon3) {
\scalebox{.4}{\input{figures/dtrees/prism_b3}}
};
\node[vertex] (nback4) at (backpentagon4) {
\scalebox{.4}{\input{figures/dtrees/prism_b4}}
};
\node[vertexback] (nback5) at (backpentagon5) {
\scalebox{.4}{\input{figures/dtrees/prism_b5}}
};

%% Drawing edges in front
%% List of edges in front
\def\listEdgesFront{nfront1/nfront2, nfront2/nfront3, nfront3/nfront4, nfront1/nfront5, nfront5/nfront4, nback1/nback2, nback2/nback3, nback3/nback4, nfront1/nback1, nfront2/nback2, nfront3/nback3, nfront4/nback4}
\foreach \x/\y in \listEdgesFront{
%    \draw[edge] (\x) -- (\y);
    \draw[edge,->] (\x) -- (\y);
}

% Drawing edges in the back
% List of edges in the back
\def\listEdgesBack{nback1/nback5, nback5/nback4, nfront5/nback5};
\foreach \x/\y in \listEdgesBack{
%    \draw[edge,back] (\x) -- (\y);
    \draw[edge,back,->] (\x) -- (\y);
}

\end{tikzpicture}
\end{tabular}
\caption{Example of a 3-dimensional pure $s$-Tamari interval.}
\label{fig_pure_sTamari_intervals_three}
\end{figure}

The following lemma is the analog of Lemma~\ref{lem_maximal_equal_join} for the $s$-Tamari lattice.

\begin{lemma}\label{lem_maximal_equal_join_stamari}
Let $T_1$ and $T_2$ be two $s$-Tamari trees such that 
$T_2 = T_1+A$ with $A$ a subset of Tamari-ascents of $T_1$. Then $T_2$ can be obtained as the join
$
T_2 = \bigvee_{a\in A} (T_1+a)
$ 
over all $s$-Tamari trees $T_1+a$ obtained from $T_1$ by rotating a Tamari-ascent~$a\in A$.  
\end{lemma}

\begin{proof}
The proof is similar to the one of Lemma~\ref{lem_maximal_equal_join}. Indeed even though the definition of tree-ascents differs from the definition of Tamari-ascents, the \emph{rotation} is the same: increase the cardinality of the inversion and take the transitive closure. Besides, as $s$-Tamari is a sublattice of the $s$-weak order, the definition of the join is the same as in the $s$-weak order (by union and transitive closure of tree-inversions). Then all the arguments for proving Lemma~\ref{lem_maximal_equal_join} still work.
\end{proof}

%Let $T$ be an $s$-Tamari tree and $A$ be a (possibly empty) subset of Tamari-ascents of $T$. For simplicity, we will denote by $T+A$ the $s$-Tamari tree with inversion set $\tc{(\inv(T)+A)}$.  
%
%\begin{definition}[The $s$-associahedron]
%The \defn{$s$-associahedron} $\Asso{s}$ is a polyhedral complex whose faces are pairs $(T,A)$ where $T$ is an $s$-Tamari tree and $A$ is a subset of Tamari-ascents of $T$. The dimension of $(T,A)$ is equal to $|A|$. In particular, 
%\begin{enumerate}
%\item the vertices of $\Asso{s}$ are $s$-Tamari trees $T$, and
%\item two $s$-Tamari trees are connected by an edge if and only if they are related by an $s$-Tamari rotation. 
%\end{enumerate} 
%The face $(T,A)$ is contained in $(T',A')$ if and only if $[T,T+A]\subseteq [T',T'+A']$ as intervals in the $s$-Tamari lattice.
%\end{definition}

\begin{definition}[The $s$-associahedron]\label{def_sAssociahedron}
The \defn{$s$-associahedron} $\Asso{s}$ is the collection of pure $s$-Tamari intervals $[T,T+A]$.
Here, $T$ denotes an $s$-Tamari tree and $A$ a subset of Tamari-ascents of $T$. 
The \defn{dimension} of $[T,T+A]$ is said to be equal to $|A|$. In particular, 
\begin{enumerate}
\item the vertices of $\Asso{s}$ are $s$-Tamari trees $T$, and
\item two $s$-Tamari trees are connected by an edge if and only if they are related by an $s$-Tamari rotation. 
\end{enumerate} 
We refer to pure $s$-Tamari intervals $[T,T+A]$ as \defn{faces} of $\Asso{s}$, and say that one face is contained in another if the containment holds as intervals in the $s$-Tamari lattice. 
\end{definition}

Figure~\ref{fig_associahedron_s022} illustrates an example of the s-associahedron $\Asso{0, 2, 2}$. As we can see, it is a polytopal complex whose faces are labeled by pure $s$-Tamari intervals, and whose edge graph is the Hasse diagram of the s-Tamari lattice. Its $f$-polynomial is
\begin{equation*}
12 + 16 t + 5 t^2.
\end{equation*}

Indeed, there are $12$ $s$-Tamari trees (faces of dimension $0$), $16$ edges (faces of dimension $1$) and $5$ pure $s$-Tamari intervals of dimension $2$, which correspond to the $5$ polygons. 

The following proposition is straightforward from the definition.

\begin{proposition}\label{prop:f-poly_stamari}
The $f$-polynomial of the $s$-associahedron $\Asso{s}$ is given by
\[
\sum_{T} (1+t)^{\tasc(T)},
\]
where the sum ranges over all $s$-Tamari trees $T$ and $\tasc(T)$ denotes the number of $s$-Tamari ascents of~$T$. 
\end{proposition}

\begin{proof}
The proof is exactly the same as the proof of Proposition~\ref{prop:f-poly}.
Let 
\[
\sum_{T} (1+t)^{\tasc(T)}
= f_0+f_1t+f_2t^2+\dots.
\]
where the sum runs over all $s$-Tamari trees $T$. We need to show that $f_k$ counts the number of $k$-dimensional faces of $\Perm{s}$. 
This follows from the fact that every subset of ascents $A$ of $T$, of size $k$, contributes a $t^k$ to the term $(1+t)^{\tasc(T)}$. 
\end{proof}

The constant term $f_0$ of this polynomial is the number of elements of the $s$-weak order, which is the \defn{$s$-Catalan number}. Moreover, 
the coefficients of $t^k$ in the polynomial $\sum_{T} t^{\tasc(T)}$ may be regarded as $s$-generalizations of the Narayana numbers. The \defn{$s$-Narayana number} $\operatorname{Nar}_s(k)$ counts the number of $s$-Tamari trees with exactly $k$ Tamari ascents. They have been already considered in~\cite[Section~4]{CeballosPadrolSarmiento-geometryNuTamari} and~\cite{ceballos_signature_2018} in this general set up, and in~\cite{ArmstrongRhoadesWilliams2013} for the special case of signatures arising in rational Catalan combinatorics. 

\subsection{The $\nu$-associahedron}\label{sec_nu_associahedron}
In~\cite{CeballosPadrolSarmiento-geometryNuTamari}, Ceballos, Padrol and Sarmiento proved that the Hasse diagram of the $\nu$-Tamari lattice can be realized as the edge graph of a polyhedral complex induced by an arrangement of tropical hyperplanes. They named this polyhedral complex the $\nu$-associahedron and gave a complete characterization of its faces in terms of certain combinatorial objects called ``covering~$(I,\overline J)$-forrests", see~\cite[Definition~5.1 and Theorem~5.2]{CeballosPadrolSarmiento-geometryNuTamari}. We present here an equivalent description of their $\nu$-associahedron phrased in terms of $\nu$-trees, which is more convenient for our purposes. 
%This equivalent description follows from~\cite{ceballos_vtamarivtrees_2018,CeballosPadrolSarmiento-geometryNuTamari}.  

Let $\nu$ be a lattice path consisting of north an east unit steps. Recall that a $\nu$-tree is a maximal $\nu$-compatible collection of elements in $A_\nu$ (lattice points weakly above $\nu$), and that the $\nu$-Tamari lattice is the lattice of $\nu$-trees whose covering relation is given by $\nu$-Tamari rotations. This lattice was extended to a full simplicial complex in~\cite{CeballosPadrolSarmiento-geometryNuTamari} called the \defn{$\nu$-Tamari complex}. Its faces, which we call \defn{$\nu$-faces}, are $\nu$-compatible collections of elements in  $A_\nu$ ordered by inclusion~\cite{ceballos_vtamarivtrees_2018}. A $\nu$-face $\vF$ is said to be \defn{covering} if it contains the root (top left corner in $A_\nu$) and at least one point in each row and each column. The covering $\nu$-faces are exactly the interior faces of the $\nu$-Tamari complex, see~\cite[Lemma~4.3]{CeballosPadrolSarmiento-geometryNuTamari}.   

\begin{definition}
The \defn{$\nu$-associahedron} $\Asso{\nu}$ is the polyhedral complex of interior faces of the~\mbox{$\nu$-Tamari} complex, ordered by reversed inclusion. 
Equivalently, $\Asso{\nu}$ is the polyhedral complex of covering~\mbox{$\nu$-faces}, ordered by reversed inclusion.  
The \defn{dimension of a covering $\nu$-face} $\vF$ is $\ell(\nu)+1-|\vF|$, where $\ell(\nu)$ is the length of $\nu$.  
In particular,
 \begin{enumerate}
\item the vertices of $\Asso{\nu}$ are $\nu$-trees, and
\item two $\nu$-trees are connected by an edge if and only if they are related by a $\nu$-Tamari rotation. 
\end{enumerate} 
\end{definition}

An example of the $\nu$-associahedron for $\nu=NEENEEN$ is shown in Figure~\ref{fig_vassociahedron_NEENEEN}. 
Note that in this case, it is isomorphic to the $s$-associahedron for $s=(0,2,2)$ illustrated in Figure~\ref{fig_associahedron_s022}.

\begin{figure}[htb]
\includegraphics[width=0.8\textwidth]{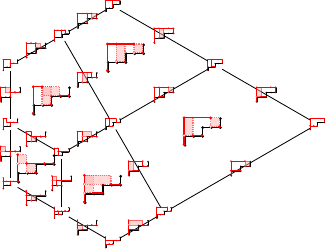}
\caption{The $\nu$-associahedron for $\nu=NEENEEN$.}
\label{fig_vassociahedron_NEENEEN}
\end{figure}

Before matching the definitions of $s$-associahedra and $\nu$-associahedra we need the following results.

\begin{lemma}
\label{lem_vfaces1}
Covering $\nu$-faces $\vF$ are in bijection with pairs $(\vT,\vA)$ such that  $\vT$ is a $\nu$-tree and~$\vA$ is a (possibly empty) subset of $\nu$-ascents of $\vT$. The bijection is determined by $\vF=\vT\smallsetminus \vA$ and~$\dim(\vF)=|\vA|$.
\end{lemma}

\begin{proof}
If $\vT$ is a $\nu$-tree and $a$ is a $\nu$-ascent of $T$, then the $\nu$-face $T\smallsetminus \{a\}$ is contained in two facets ($T$ and the rotation of $T$ at $a$). Therefore, $T\smallsetminus \{a\}$ is an interior face of the $\nu$-Tamari complex. Similarly, if~$A$ is a subset of $\nu$-ascents of $T$, the $\nu$-face $T\smallsetminus A$ is also an interior face. Thus, $\vF=T\smallsetminus A$ is a covering $\nu$-face.
By definition, its dimension is  equal to $\dim(\vF)=\ell(\nu)+1-|\vT\smallsetminus \vA|=\ell(\nu)+1-|\vT|+|\vA|=|\vA|$.

It remains to show that every covering $\nu$-face $\vF$ can be written uniquely as $\vF=\vT\smallsetminus \vA$ for some $\nu$-tree~$\vT$ and a subset $\vA$ of $\nu$-ascents of $\vT$. We prove this using the connection between $\nu$-Tamari lattices and subword complexes presented in~\cite{ceballos_vtamarivtrees_2018}.

Let $\vF$ be a covering $\nu$-face. 
By~\cite[Theorem~5.5]{ceballos_vtamarivtrees_2018}, the $\nu$-Tamari lattice is the increasing flip graph of a suitably chosen subword complex $\mathcal{SC}(Q_\nu,\pi_\nu)$. 
It is not required to understand subword complexes in this proof, but to know that under this correspondence, $\nu$-trees correspond to facets of the subword complex and $\nu$-rotations to increasing flips.   
The link of a face in a subword complex is also a subword complex itself, see e.g.~\cite{knutson_subword_2004}. 
Therefore, the restriction of the $\nu$-Tamari lattice to the set of $\nu$-trees containing $\vF$ is the increasing flip graph of another suitable subword complex $\mathcal{SC}(Q_{\vF,\nu,},\pi_\nu)$. 
Furthermore, it is known that the increasing flip graph of a subword complex has a unique source and a unique sink~\cite[Proposition~4.8]{PilaudStump-ELlabelings}. 
In our language, this means that there is a unique $\nu$-tree $\vT_\vF^{\min}$ (resp. $\vT_\vF^{\max}$) containing $\vF$ such that every ``flippable" element of $\vT_\vF^{\min}$ (resp. $\vT_\vF^{\max}$) that is not in~$\vF$ is ``increasingly flippable" (resp. ``decreasingly flippable"), that is a $\nu$-ascent (resp. ``$\nu$-descent"). 
Moreover, since $\vF$ is an interior face, then every element of $\vA=\vT_\vF^{\min}\smallsetminus \vF$ is flippable. Thus, $\vA$ is a subset of $\nu$-ascents of $\vT_\vF^{\min}$. Since $\vF=\vT_\vF^{\min}\smallsetminus \vA$, this finishes our proof.
\end{proof}

For $a\in \vA$ we denote by $\vT_a$ the $\nu$-tree obtained from $\vT$ by applying a $\nu$-Tamari rotation at the~\mbox{$\nu$-ascent}~$a$, and by $\vT+\vA$ the join of the set $\{\vT_a: a\in\vA\}$. The interval $[\vT,\vT+\vA]$ of the $\nu$-Tamari lattice is called a \defn{pure $\nu$-Tamari interval}.  
An example is illustrated in Figure~\ref{fig_pure_vTamari_interval}.

\begin{figure}[htb]
\begin{tabular}{cc}
\scalebox{.4}{\begin{tikzpicture}
\draw[line width = 4] (0,0) -- (0,1);
\draw[line width = 4] (0,1) -- (1,1);
\draw[line width = 4] (1,1) -- (2,1);
\draw[line width = 4] (2,1) -- (3,1);
\draw[line width = 4] (3,1) -- (3,2);
\draw[line width = 4] (3,2) -- (4,2);
\draw[line width = 4] (4,2) -- (4,3);
\draw[line width = 4] (4,3) -- (5,3);
\draw[line width = 4] (5,3) -- (6,3);
\draw[line width = 4] (6,3) -- (6,4);
\draw[line width = 4] (6,4) -- (7,4);
\draw[line width = 4] (7,4) -- (8,4);
\draw[line width = 4] (8,4) -- (8,5);
\draw[line width = 4] (8,5) -- (8,6);
\draw[line width = 4] (8,6) -- (9,6);
\draw[line width = 4] (9,6) -- (9,7);
\draw[fill, radius=0.15] (0,0) circle;
\draw[fill, radius=0.15] (0,1) circle;
\draw[fill, radius=0.15] (1,1) circle;
\draw[fill, radius=0.15] (2,1) circle;
\draw[fill, radius=0.15] (3,1) circle;
\draw[fill, radius=0.15] (3,2) circle;
\draw[fill, radius=0.15] (4,2) circle;
\draw[fill, radius=0.15] (4,3) circle;
\draw[fill, radius=0.15] (5,3) circle;
\draw[fill, radius=0.15] (6,3) circle;
\draw[fill, radius=0.15] (6,4) circle;
\draw[fill, radius=0.15] (7,4) circle;
\draw[fill, radius=0.15] (8,4) circle;
\draw[fill, radius=0.15] (8,5) circle;
\draw[fill, radius=0.15] (8,6) circle;
\draw[fill, radius=0.15] (9,6) circle;
\draw[fill, radius=0.15] (9,7) circle;

\draw[red, dashed, line width = 4, fill=red!30, opacity = .5] (1.9,2.1) rectangle (0.9,7.1);
\draw[red, dashed, line width = 4, fill=red!30, opacity = .5] (5.9,4.1) rectangle (4.9,5.1);
\draw[red, dashed, line width = 4, fill=red!30, opacity = .5] (4.9,5.1) rectangle (0.9,7.1);
\draw[red, fill=white, radius=0.15] (0.9, 2.1) circle;
\draw[red, fill=white, radius=0.15] (4.9, 4.1) circle;
\draw[red, fill=white, radius=0.15] (0.9, 5.1) circle;
\draw[red, fill, radius=0.15] (-0.1,0.1) circle;
\draw[red, fill, radius=0.15] (2.9,1.1) circle;
\draw[red, fill, radius=0.15] (1.9,1.1) circle;
\draw[red, fill, radius=0.15] (3.9,2.1) circle;
\draw[red, fill, radius=0.15] (1.9,2.1) circle;
\draw[red, fill, radius=0.15] (5.9,3.1) circle;
\draw[red, fill, radius=0.15] (7.9,4.1) circle;
\draw[red, fill, radius=0.15] (6.9,4.1) circle;
\draw[red, fill, radius=0.15] (5.9,4.1) circle;
\draw[red, fill, radius=0.15] (4.9,5.1) circle;
\draw[red, fill, radius=0.15] (8.9,6.1) circle;
\draw[red, fill, radius=0.15] (8.9,7.1) circle;
\draw[red, fill, radius=0.15] (0.9,7.1) circle;
\draw[red, fill, radius=0.15] (-0.1,7.1) circle;

\draw[red, line width = 4] (-0.1,0.1) -- (-0.1,7.1);
\draw[red, line width = 4] (-0.1,7.1) -- (0.9,7.1);
\draw[red, line width = 4] (0.9,7.1) -- (8.9,7.1);
\draw[red, line width = 4] (1.9,1.1) -- (1.9,2.1);
\draw[red, line width = 4] (1.9,1.1) -- (2.9,1.1);
\draw[red, line width = 4] (1.9,2.1) -- (3.9,2.1);
\draw[red, line width = 4] (5.9,3.1) -- (5.9,4.1);
\draw[red, line width = 4] (5.9,4.1) -- (6.9,4.1);
\draw[red, line width = 4] (6.9,4.1) -- (7.9,4.1);
\draw[red, line width = 4] (8.9,6.1) -- (8.9,7.1);
\end{tikzpicture}} &
\input{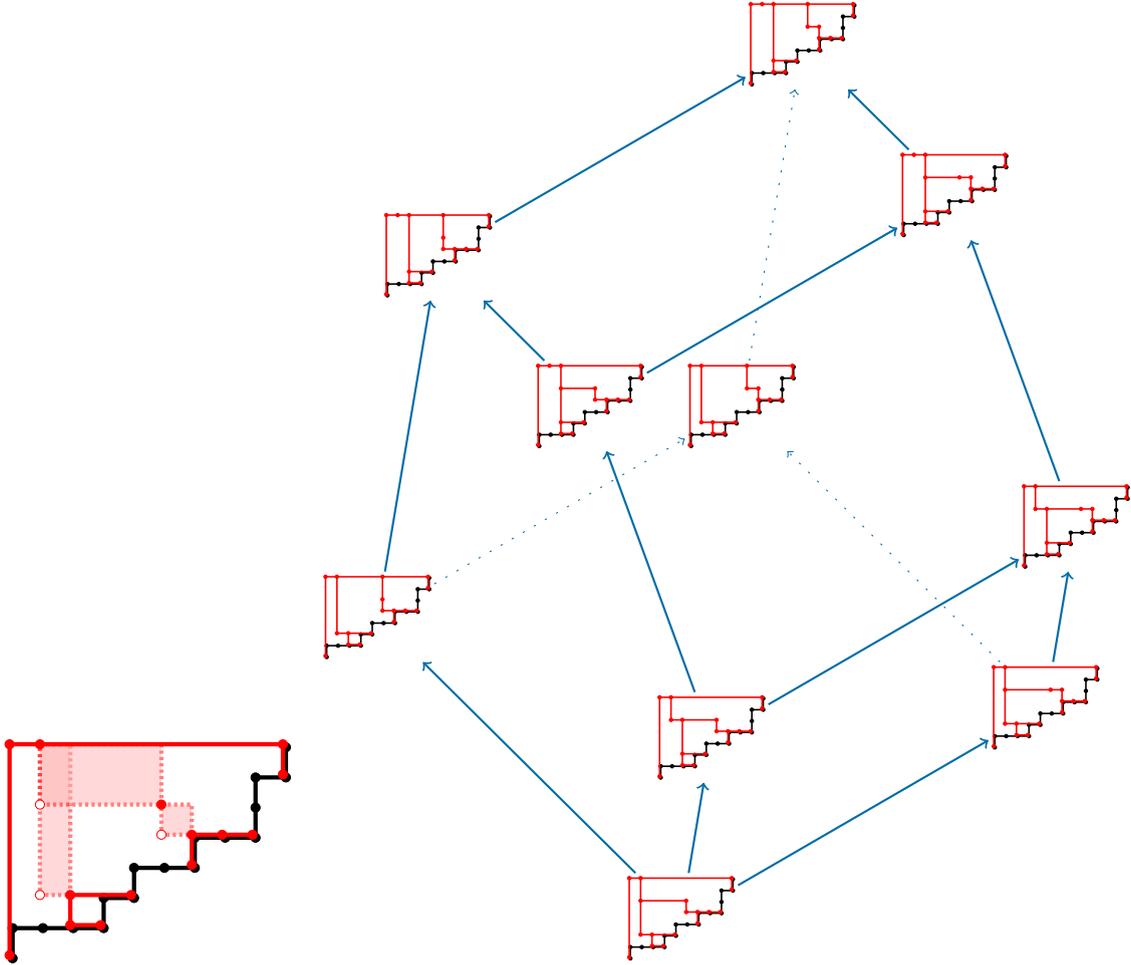}
\end{tabular}
\caption{Example of a pure $\nu$-Tamari interval}
\label{fig_pure_vTamari_interval}
\end{figure}

As we have seen in Section~\ref{sec_characterization_pure_intervals}, the classification of pure intervals in the $s$-weak order is rather involved and technical, and one might expect a similar situation to happen for the classification of pure $s$-Tamari intervals in the $s$-Tamari lattice. However, if $\nu=\nu(\reverse s)$, we know by Proposition~\ref{prop_sTam_vTam_trees} and Lemma~\ref{lem_maximal_equal_join_stamari} that pure $s$-Tamari intervals are mapped to pure $\nu$-Tamari intervals under the bijection~$\streestovtrees$ from Section~\ref{sec_sTamarivTamari}.
For instance, the pure \mbox{$s$-Tamari} interval in Figure~\ref{fig_pure_sTamari_intervals_three} is mapped to the pure $\nu$-Tamari interval in Figure~\ref{fig_pure_vTamari_interval}.
Moreover, the following proposition gives a nice and simple classification of pure $\nu$-Tamari intervals, as the sets of $\nu$-trees containing covering $\nu$-faces. 

\begin{proposition}\label{prop_pure_nuTamari_intervals}
Let $\vT$ is a $\nu$-tree and~$\vA$ is a (possibly empty) subset of $\nu$-ascents of $\vT$. 
If $\vF=\vT\smallsetminus \vA$ is the corresponding covering $\nu$-face,  
then
\begin{align}\label{eq_pure_nuTamari_interval}
[\vT,\vT+\vA] =
\left\{
\nu\text{-trees } \vT':\,  \vF\subseteq \vT'
\right\}.
\end{align}
\end{proposition}

\begin{proof}
Using the same subword complex techniques from the proof of Lemma~\ref{lem_vfaces1}, the set of $\nu$-trees containing $\vF$ is in correspondence with the set of facets of the subword complex $\mathcal{SC}(Q_{\vF,\nu,},\pi_\nu)$, and \mbox{$\nu$-rotations} correspond to increasing flips. 
As argued above, this set of facets has a unique minimal element $\vT_\vF^{\min}$ and a unique maximal element $\vT_\vF^{\max}$. Therefore,
\begin{align}\label{eq_pure_nuTamari_interval_two}
[\vT_\vF^{\min},\vT_\vF^{\max}] =
\left\{
\nu\text{-trees } \vT':\,  \vF\subseteq \vT'
\right\}.
\end{align}
By construction, we have that $\vT=\vT_\vF^{\min}$ and $\vA=\vT\smallsetminus \vF$. It remains to show that $\vT+\vA=\vT_\vF^{\max}$.

By~\cite[Lemma~4.4 or Proposition~5.16]{CeballosPadrolSarmiento-geometryNuTamari}, the interval in Equation~\eqref{eq_pure_nuTamari_interval_two} is a product of classical Tamari lattices. Its minimal element is $\vT$, and it contains the $\nu$-Tamari trees $\vT+a$ for $a\in \vA$ (because we rotate an element $a$ that is not in $\vF$). Since the dimension $\dim(\vF)=|\vA|$, the interval  in Equation~\eqref{eq_pure_nuTamari_interval_two} contains all the covers of the minimal element. Furthermore, the join of the covers of the minimal element in a classical Tamari lattice is the top element of the lattice, and this property is preserved by taking products of Tamari lattices. 
Therefore, $\vT_\vF^{\max}$ is the join of the set $\{\vT_a: a\in\vA\}$, which is equal to $\vT+\vA$ by definition. 
\end{proof}

Figure~\ref{fig_pure_vTamari_interval} shows an example of a pure $\nu$-Tamari interval. The $\nu$-ascents $\vA$ of the $\nu$-tree $\vT$ are the circled nodes in the figure, and the covering $\nu$-face $\vF=\vT\smallsetminus\vA$ is the set of filled red nodes. The $\nu$-trees in the interval $[\vT,\vT+\vA]$ are exactly the $\nu$-trees containing these filled red nodes. 

\begin{corollary}
\label{cor_vfaces2}
Let $[\vT,\vT+\vA]$ and $[\vT',\vT'+\vA']$ be two pure $\nu$-Tamari intervals, and $\vF=\vT\smallsetminus\vA$ and $\vF'=\vT'\smallsetminus\vA'$ be their corresponding covering $\nu$-faces. Then
$[\vT,\vT+\vA]\subseteq [\vT',\vT'+\vA']$ if and only if $\vF\supseteq\vF'$.
\end{corollary}

\begin{proof}
Assume $\vF\supseteq\vF'$. Since every $\nu$-tree containing $\vF$ contains $\vF'$, then by Equation~\ref{eq_pure_nuTamari_interval} we have $[\vT,\vT+\vA]\subseteq [\vT',\vT'+\vA']$. This proves the backward direction. 

For the forward direction, assume that  $\vF\nsupseteq\vF'$. We aim to show that $[\vT,\vT+\vA]\nsubseteq [\vT',\vT'+\vA']$. 
Let~$f'\in \vF'\smallsetminus \vF$, and $\widetilde \vT$ be a $\nu$-tree containing $\vF$. We consider two cases. 

Case 1: $f'\notin \widetilde \vT$. In this case, $\widetilde \vT$ does not contain $\vF'$, and so  $\widetilde \vT\in [\vT,\vT+\vA]$ but $\widetilde \vT\notin [\vT',\vT'+\vA']$ as we wanted to show.

Case 2: $f'\in \widetilde \vT$. Since $\vF\subseteq \widetilde \vT$ is an interior face and $f'\notin \vF$, then $f'$ is flippable in $\widetilde \vT$. Flipping it, we obtain a new $\nu$-tree $\vT^*=\widetilde \vT \smallsetminus \{f'\} \cup \{f^*\}$. This new tree satisfies $\vF\subseteq \vT^*$ but   $\vF'\nsubseteq \vT^*$. 
Thus, $ \vT^*\in [\vT,\vT+\vA]$ but $\vT^*\notin [\vT',\vT'+\vA']$ as we wanted.
\end{proof}

\subsection{The $s$-associahedron and the $\nu$-associahedron are isomorphic}

The bijection $\streestovtrees$ between $s$-Tamari trees and $\nu(\reverse s)$-trees described in Section~\ref{sec_sTamarivTamari} extends naturally to a bijection $\overline \streestovtrees$ between the faces of the $s$-associahedron and the faces of the $\nu(\reverse s)$-associahedron. For each pair $(T,A)$ of an $s$-Tamari tree $T$ and a subset $A$ of Tamari-ascents of $T$, we can associate a pair $(\vT,\vA)$ of a $\nu$-tree $\vT=\streestovtrees(T)$ and a subsets $\vA$ of $\nu$-assents of $\vT$ corresponding to $A$. We denote by $\overline \streestovtrees$ the map that sends the pure $s$-Tamari interval $[T,T+A]$ to the covering $\nu$-face $\vF=\vT\smallsetminus \vA$. 
%Lemma~\ref{lem_vfaces1} and Corollary~\ref{lem_vfaces2} imply the following result.

\begin{theorem}
\label{thm:nu-ass-s-ass}
The map $\overline \streestovtrees$ is an isomorphism between $\Asso{s}$ and $\Asso{\nu(\reverse s)}$.
\end{theorem}

\begin{proof}
By Lemma~\ref{lem_vfaces1}, the map $\overline \streestovtrees$ is a bijection between pure $s$-Tamari intervals and covering $\nu(\reverse s)$-faces. So, $\overline \streestovtrees$ is a bijection between the faces of the $s$-associahedron and the faces of the $\nu(\reverse s)$-associahedron. 
We need to show that this map preserves their order relation. 

Let $[T,T+A]$ and $[T',T'+A']$ be two pure $s$-Tamari intervals. Consider the corresponding pure $\nu$-Tamari intervals $[\vT,\vT+\vA]$ and $[\vT',\vT'+\vA']$ determined by the bijection $\streestovtrees$, and let $\vF=\vT\smallsetminus \vA$ and $\vF'=\vT'\smallsetminus \vA'$ be the corresponding covering $\nu$-faces. 
Since $\streestovtrees$ is order preserving, 
$[T,T+A]\subseteq [T',T'+A']$ if and only if $[\vT,\vT+\vA]\subseteq [\vT',\vT'+\vA']$. 
Combining this with Corollary~\ref{cor_vfaces2}, we get that 
$[T,T+A]\subseteq [T',T'+A']$ if and only if $\vF\supseteq\vF'$ as we wanted.
\end{proof}

\part{Polytopal conjectures}\label{part_polytopal_conjectures}

\section{Polytopal complex and polytopal subdivision realizations}\label{sec_polytopal_conjectures}

We proved in Theorem~\ref{sperm_combinatorial_complex} that the $s$-permutahedron is what we call a \defn{combinatorial complex}, i.e. a collection of cells or faces (pure intervals) satisfying two properties: (1) it is closed under taking faces and (2) any two faces intersect properly. These two conditions are necessary conditions for being a \defn{polytopal complex}. In order to be a polytopal complex, we need the further property that all the cells of the complex can be geometrically realized as (convex) polytopes. This condition is stated in the following conjecture. 

\begin{conjecture}[Polytopality of pure intervals]
\label{conj:pure-polytopal}
For any weak composition $s$, the pure intervals of the $s$-weak order are polytopal in the following sense:
\begin{enumerate}
\item The inclusion poset of pure intervals contained in a pure interval $[T,T+A]$ is the face lattice of some polytope $P$ of dimension $|A|$. \label{item_pure_polytopal_one}  
\item The Hasse diagram of the restriction of the $s$-weak order to $[T,T+A]$ is the edge graph of $P$.\label{item_pure_polytopal_two}
\end{enumerate}
\end{conjecture}

Item~\eqref{item_pure_polytopal_two} in this conjecture is not really necessary because it follows from Item~\eqref{item_pure_polytopal_one}. However, we include it here because it is a nice property that we would like to highlight.

We have strong reasons to believe that this conjecture is true. On one side, it has been recently proven in the case where $s$ does not contain any zeros~\cite{GMPT23}, as we explain in Section~\ref{sec:flows}. On the other side, we have an empirical polytopal construction of each pure interval as a generalized permutahedron which we call an \defn{ascentope}. On Figure~ \ref{fig:pure-interval-ascentope}, you see on the right the ascentope corresponding to the pure interval. The code to compute the ascentopes is available on~\cite{SageDemoII}.
 Computational experiments show that our ascentopes have the right properties with examples of dimensions up to $9$. We plan to study these objects in future work.

Our next conjecture is even more general and was also proven when $s$ does not contain any zeros~\cite{GMPT23}, see Section~\ref{sec:flows}. All of our figures seem to indicate that the Hasse diagram of the $s$-weak order seems to be realizable as the edge graph of a polytopal subdivision of a polytope, whose faces are in correspondence with pure intervals. This polytope should be combinatorially isomorphic to the zonotope  
\begin{equation}\label{eq_zonotope}
Z(s) = \sum_{1\leq i < j \leq n} s(j)\Delta_{ij},
\end{equation}

where $\Delta_{ij}=\conv\{e_i,e_j\}\subset \mathbb R^n$. In particular, if $s$ has no zeros (except possibly for~$s(1)$) then $Z(s)$ is combinatorially an $(n-1)$-dimensional permutahedron.

\begin{figure}[htbp]
\begin{center}
\scalebox{1}{
\begin{tabular}{cccc}
\includegraphics[height=3.5cm]{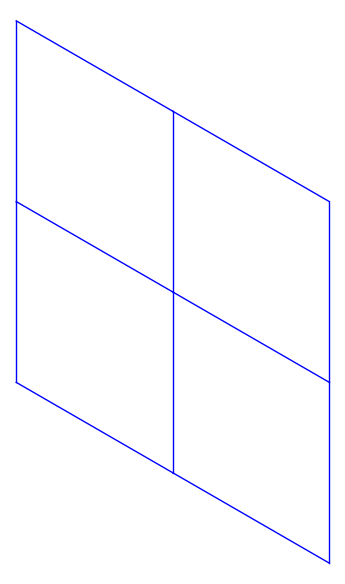} &
\includegraphics[height=3.5cm]{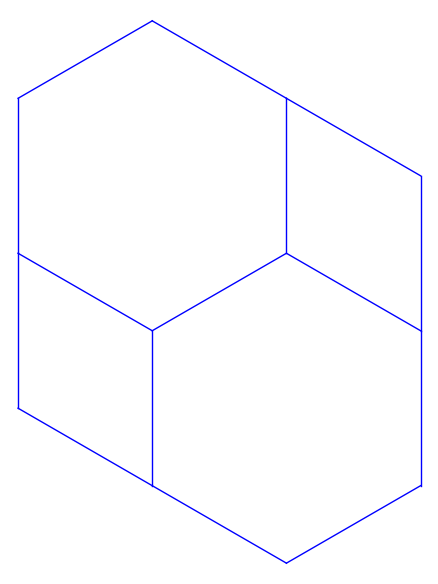} &
\includegraphics[height=3.5cm]{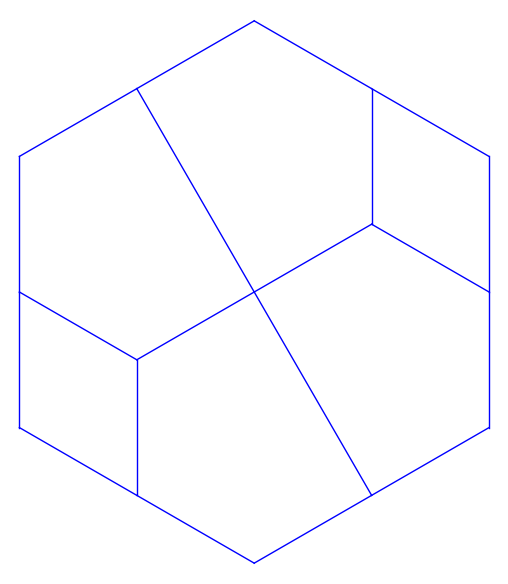} &
\includegraphics[height=3.5cm]{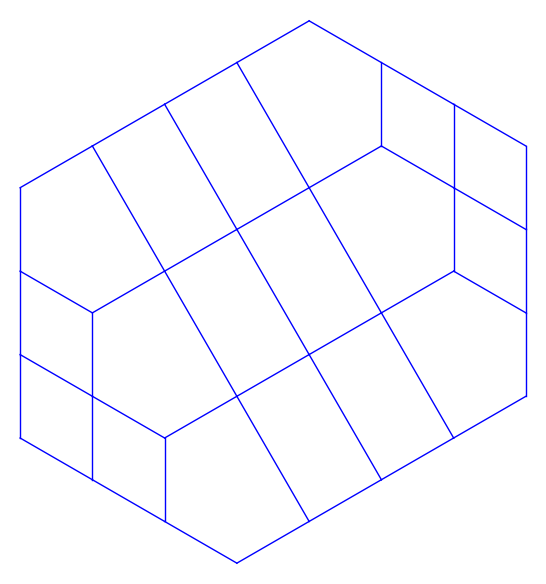} \\
$s = (0,0,2)$ & $s = (0,1,2)$ & $s = (0,2,2)$ & $s = (0,4,3)$   
\end{tabular}
}
\caption{Some geometric realizations of $s$-permutahedra in dimension~2.}
\label{fig:s-perm_dim2}
\end{center}
\end{figure}

\begin{figure}[htbp]
\begin{center}
\scalebox{1}{
\begin{tabular}{ccc}
\includegraphics[height=3.8cm]{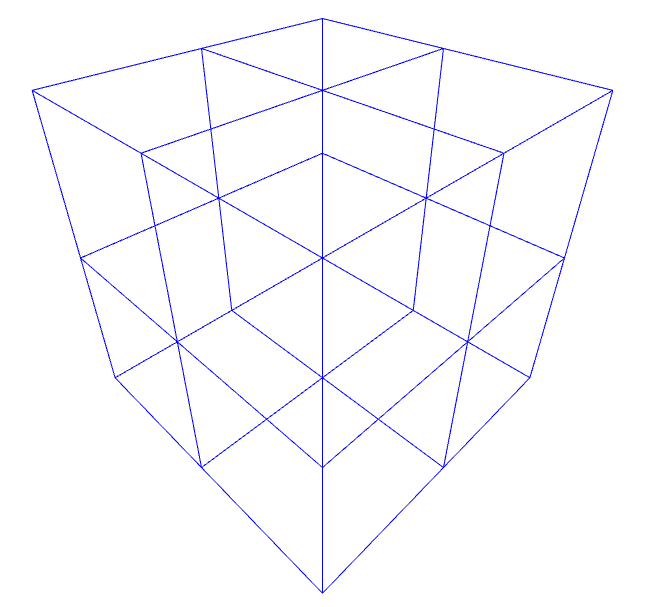} &
\includegraphics[height=4.2cm]{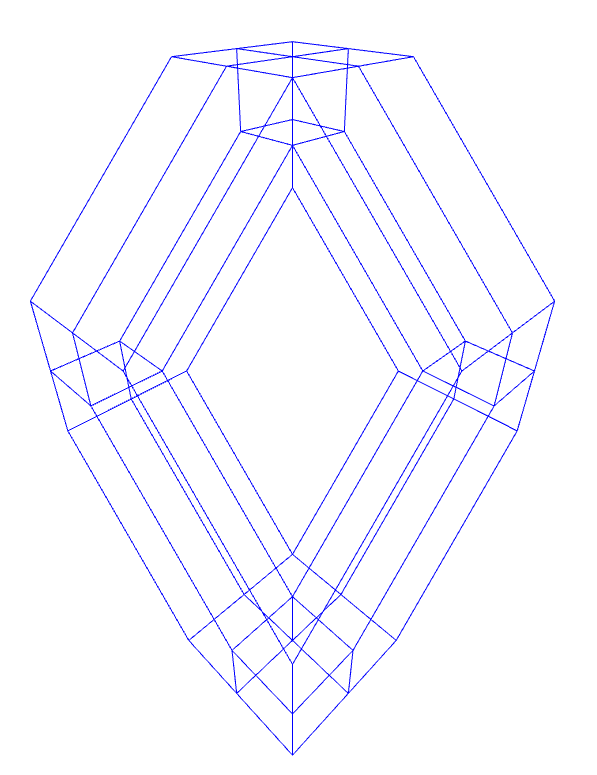} &
\includegraphics[height=3.5cm]{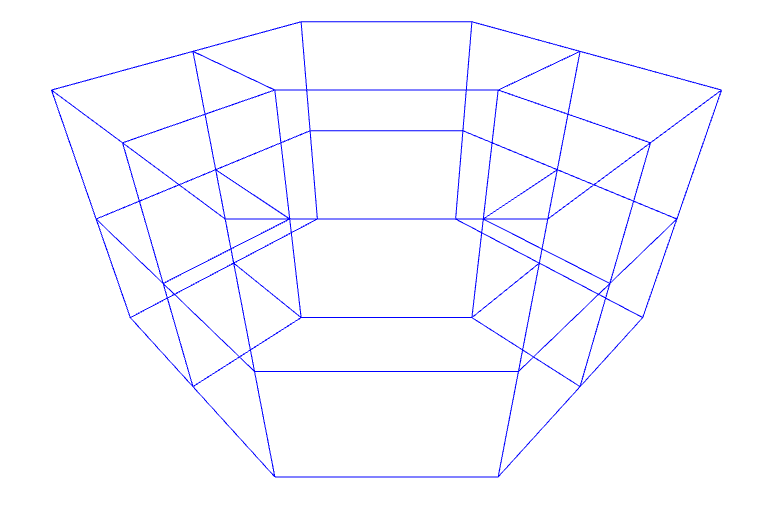} \\
$s = (0,0,0,2)$ & $s = (0,0,1,2)$, & $s = (0,1,0,2)$ \\
\includegraphics[height=4.2cm]{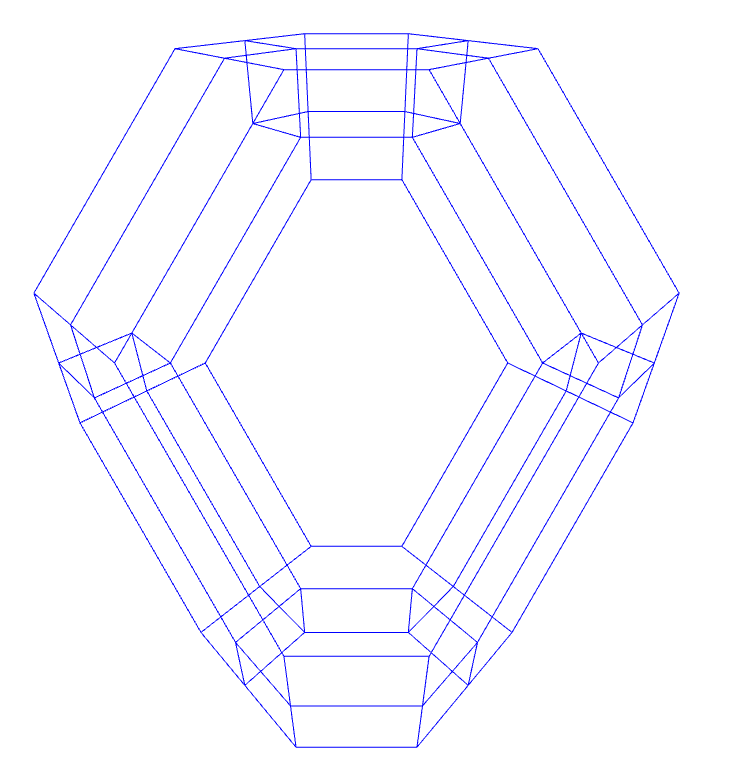} &
\includegraphics[height=4.2cm]{s0222} &
\includegraphics[height=4.2cm]{s0232} \\
$s = (0,1,1,2)$ & $s = (0,2,2,2)$, & $s = (0,2,3,2)$ \\
\includegraphics[height=4.2cm]{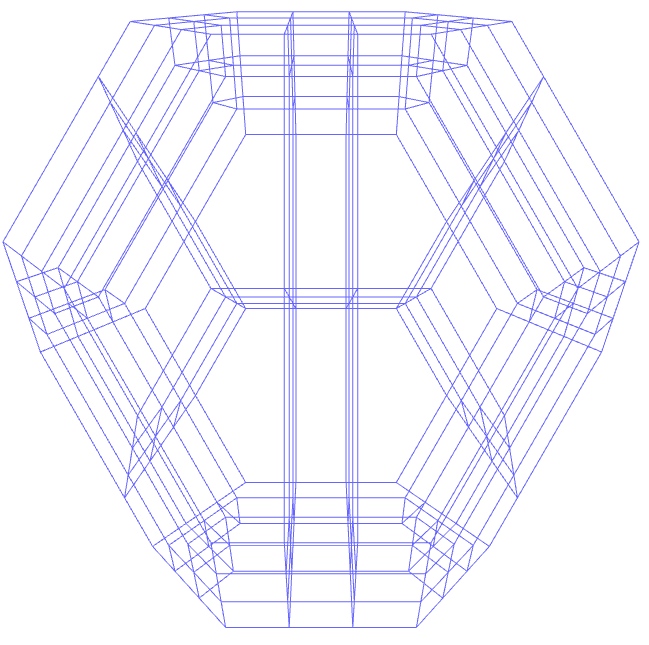} &
\includegraphics[height=4.2cm]{s0333} &
\includegraphics[height=4.2cm]{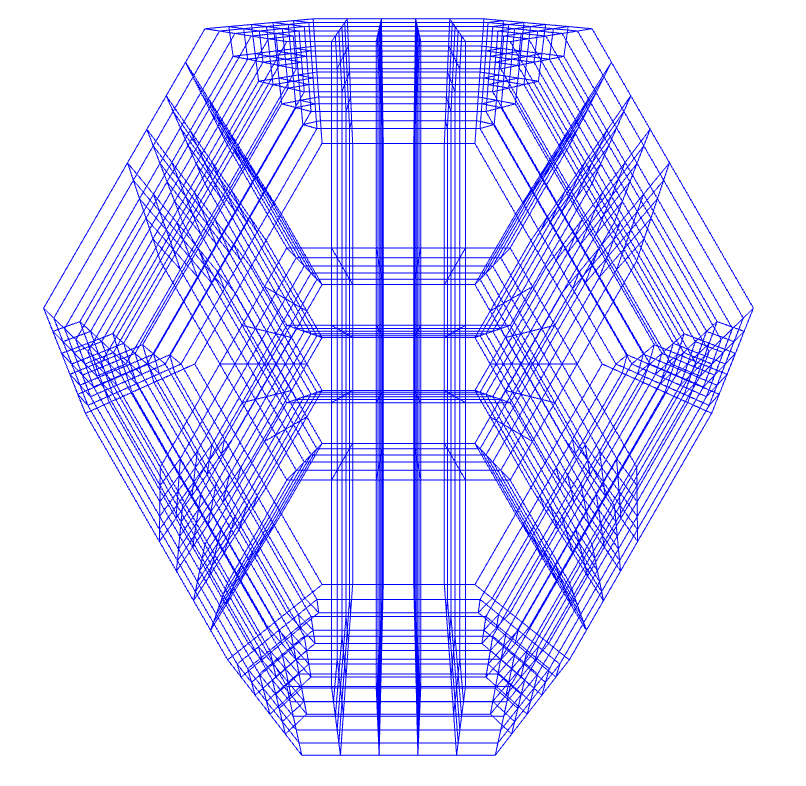} \\ 
$s = (0,3,2,3)$ & $s = (0,3,3,3)$ & $s = (0,5,5,5)$  
\end{tabular}
}
\caption{Some geometric realizations of $s$-permutahedra in dimension~3.}
\label{fig:s-perm_dim3}
\end{center}
\end{figure}

\begin{conjecture}[Polytopal subdivision realization]
\label{coj:spermutahedra}
For any weak composition $s$,
the $s$-permutahedron can be geometrically realized as a polytopal subdivision of a polytope which is combinatorially isomorphic to $Z(s)$. This means, 
\begin{enumerate}
\item The inclusion poset of pure intervals of the $s$-weak order is the face poset of the subdivision. \label{item_spermutahedra_polytopal_one}  
\item The Hasse diagram of the $s$-weak order is the edge graph of the subdivision.\label{item_spermutahedra_polytopal_two}
\end{enumerate}
\end{conjecture}

Again, Item~\eqref{item_spermutahedra_polytopal_one} in this conjecture follows from Item~\eqref{item_spermutahedra_polytopal_two}, but we include it here because it is a nice property that we would like to highlight. As we will discuss in Section~\ref{sec_geometric_realizations_dim_2_3}, this conjecture holds in full generality in dimensions 2 and 3. Examples of geometric realizations of $s$-permutahedra are illustrated in Figures~\ref{fig:s-perm_dim2} and~\ref{fig:s-perm_dim3}. You can find 3-dimensional animations of these polyhedral subdivisions and more in  \href{https://www.lri.fr/~pons/static/spermutahedron/}{this webpage}~\cite{SPermutaDemo}. Besides, all examples in dimensions $2$ and $3$ can be computed with {\tt SageMath} as shown in our demo worksheet~\cite{SageDemoII}.

One remarkable property of the classical permutahedron is that the classical associahedron can be obtained from it by removing certain facets~\cite{HL07,HLT11}. We believe that if $s$ contains no zeros, this property also holds for the $s$-permutahedron and the $s$-assoociahedron.

\begin{conjecture}[Removing facets of the $s$-permutahedron]
\label{con:sassociahedra}
If $s$ has no zeros (except for $s(1)$), there exists a geometric realization of the \mbox{$s$-permutahedron} such that the $s$-associahedron can be obtained from it by removing certain facets.
\end{conjecture}

Examples of $s$-associahedra obtained from $s$-permutahedra by removing certain facets are illustrated in Figures~\ref{fig:2d-asso} and~\ref{fig_geometricrealizations_intro}. The $3$d examples are also available on the webpage~\cite{SPermutaDemo}. All these examples were computed with {\tt SageMath} and can be obtained from our demo worksheet~\cite{SageDemoII}.

\begin{figure}[h]
\begin{tabular}{cc}
\includegraphics[height=3.5cm]{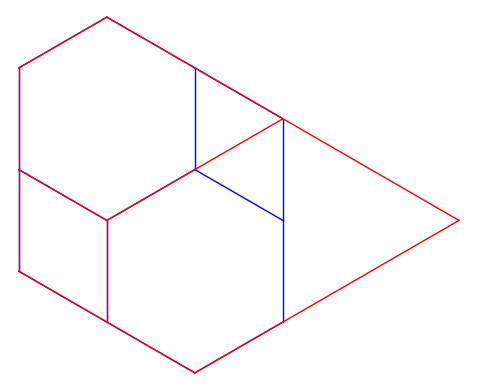} &
\includegraphics[height=3.5cm]{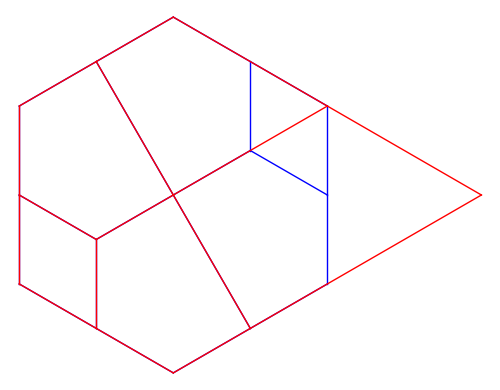} \\
$s = (0,1,2)$ & $s = (0,2,2)$ \\
\includegraphics[height=3.5cm]{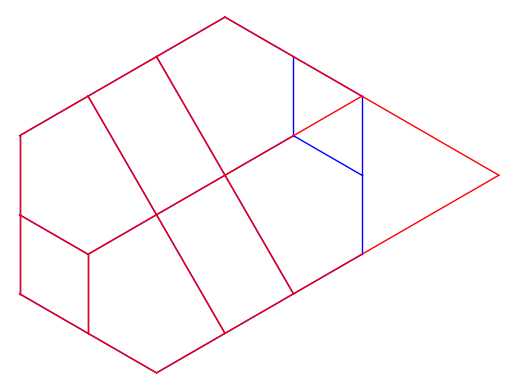} &
\includegraphics[height=3.5cm]{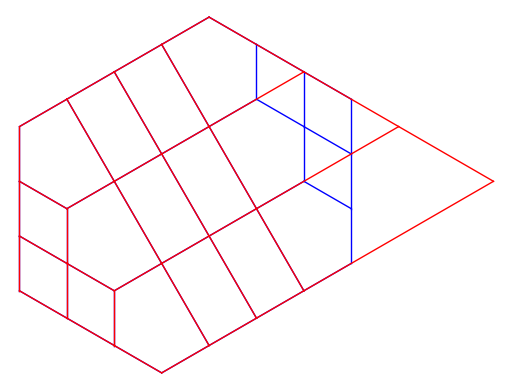} \\
$s = (0,3,2)$ & $s = (0, 4,3)$
\end{tabular}
\caption{Realizations of $s$-associahedra from $s$-permutahedra}
\label{fig:2d-asso}
\end{figure}

\begin{figure}[h!]
\begin{tabular}{cc}
\begin{tikzpicture}
\draw[blue, dashed] (0.0, 1.0) -- (0.0, 2.0);
\draw[blue, dashed] (0.0, 1.0) -- (-0.8660254037844386, 1.5);
\draw[blue, dashed] (0.0, 2.0) -- (-0.8660254037844386, 2.5);
\draw[blue, dashed] (0.0, 0.0) -- (0.0, 1.0);
\draw[blue] (0.0, 0.0) -- (-0.8660254037844386, 0.5);
\draw[blue] (-0.8660254037844386, 0.5) -- (-0.8660254037844386, 1.5);
\draw[blue] (-0.8660254037844386, 0.5) -- (-1.7320508075688772, 1.0);
\draw[blue, dashed] (-0.8660254037844386, 2.5) -- (-1.7320508075688772, 3.0);
\draw[blue, dashed] (-0.8660254037844386, 1.5) -- (-0.8660254037844386, 2.5);
\draw[blue] (-0.8660254037844386, 1.5) -- (-1.7320508075688772, 2.0);
\draw[blue] (-1.7320508075688772, 1.0) -- (-1.7320508075688772, 2.0);
\draw[blue] (-1.7320508075688772, 2.0) -- (-1.7320508075688772, 3.0);
\end{tikzpicture}
&
\begin{tikzpicture}
\draw[blue] (0.0, 0.0) -- (-0.8660254037844386, 0.5);
\draw[blue] (-0.8660254037844386, 0.5) -- (-0.8660254037844386, 1.5);
\draw[blue] (-0.8660254037844386, 0.5) -- (-1.7320508075688772, 1.0);
\draw[blue] (-0.8660254037844386, 1.5) -- (-1.7320508075688772, 2.0);
\draw[blue] (-1.7320508075688772, 1.0) -- (-1.7320508075688772, 2.0);
\draw[blue] (-1.7320508075688772, 2.0) -- (-1.7320508075688772, 3.0);
\end{tikzpicture}
\end{tabular}
\caption{The $s$-permutahedron and $s$-associahedron for $s=(0,0,2)$. In this case, the $s$-associahedron is not convex because $s(2)=0$.}
\label{fig_nonconvex_s002}
\end{figure}
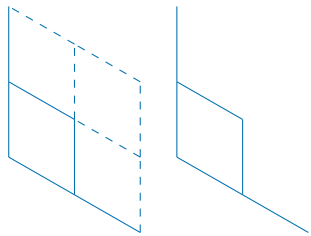

We remark that if $s$ contains zeros other than $s(1)$, then $\Asso{s}$ is not convex, see an example in Figure~\ref{fig_nonconvex_s002}. This follows from the analog result for $\nu$-associahedra in~\cite[Cor.~5.13]{CeballosPadrolSarmiento-geometryNuTamari}, which states that the $\nu$-associahedron is convex if and only if $\nu$ does not have two non-initial consecutive north steps. In such a case, the $\nu$-associahedron is a polytopal subdivision of a classical associahedron~\cite[Cor.~5.13]{CeballosPadrolSarmiento-geometryNuTamari}. 
Thus, Conjecture~\ref{con:sassociahedra} may be thought as a generalization of Hohlweg and Lange's result in~\cite{HL07}.

\section{Generalizations for finite Coxeter groups}\label{sec_Coxeter}
Another natural related question is whether there are generalizations of our constructions for other Coxeter groups. 

\begin{question}
Is there a natural definition of the $s$-permutahedron and the $s$-associahedron for other finite Coxeter groups?
\end{question}

\begin{figure}[htbp]
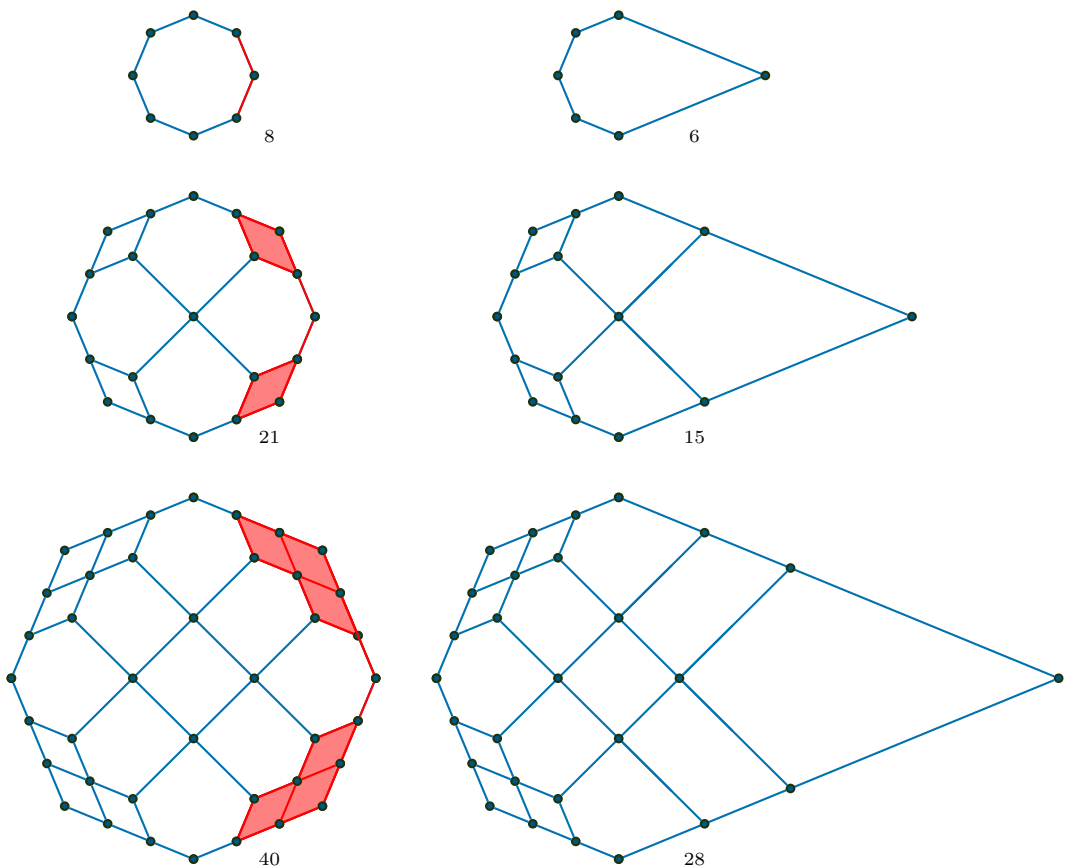

\begin{center}
\include{figures/typeB_spermutahedra_sassociahedra.tex}
\caption{A tantalizing option for $s$-permutahedra and $s$-associahedra of Coxeter type $B_2$.}
\label{fig_typeB_spermutahedra_sassociahedra}
\end{center}
\end{figure}

Figure~\ref{fig_typeB_spermutahedra_sassociahedra} shows a tantalizing partial answer to this question in type $B_2$. The left hand side of this figure shows what we believe could be Fuss-Catalan generalizations of the type $B_2$ permutahedron (subdivided octagons), while the right hand side shows the corresponding Fuss-Catalan type $B_2$ associahedra (subdivided cyclohedra) obtained by removing some facets. The number of vertices of these subdivided cyclohedra are $6, 15, 28, 45,66,\dots $, which are indeed the \defn{$m$-Catalan numbers} of type $B_2$ (also called hexagonal numbers in~\cite[Sequence A000384]{oeis}). This can be checked from the general formula of the Coxeter $m$-Catalan numbers, which is given by
\begin{align}
\prod_{i=1}^n \frac{e_i+mh+1}{e_i+1},
\end{align}
where $h$ denotes the Coxeter number and $e_1,\dots,e_n$ are the exponents of the Coxeter group. In type $B_2$, we have $h=4$, $e_1=1$ and $e_2=3$, so the previous formula reduces to $(2m+1)(m+1)$, which can be easily checked to count the number of vertices of the subdivided cyclohedra on the right of Figure~\ref{fig_typeB_spermutahedra_sassociahedra}. For instance, varying $m$ gives the desired sequence $6, 15, 28, 45,66,\dots $.  

\begin{figure}[t]
\begin{center}
\begin{tabular}{cc}
\includegraphics[width=0.25\textwidth]{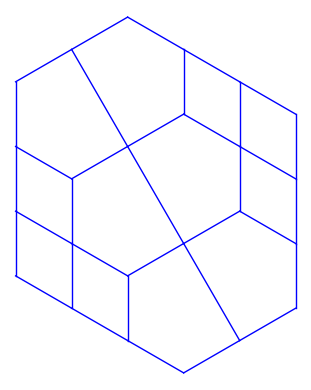} &
\includegraphics[width=0.4\textwidth]{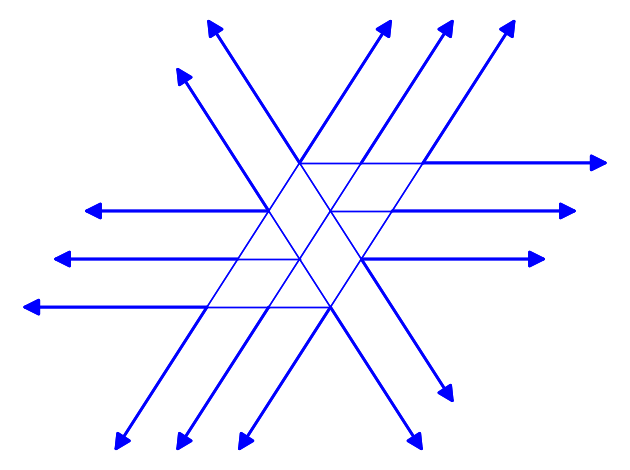}
\end{tabular}

\caption{The $s$-braid arrangement for $s=(0,2,3)$.}
\label{fig_braid_arrangement_s023}
\end{center}
\end{figure}

As far as we know, there are no \defn{m-generalizations} of the $n!$ number (or equivalently, $|W|$) for other finite Coxeter groups $W$. In type $A$, if we set $s=(m,m,\dots,m)$ to be a sequence consisting of $n$ copies of $n$, then the number of elements of the $s$-weak order is the $s!$ number~\cite{CP22}
\begin{align}\label{eq_mgeneralization_nfactorial}
s!=\prod_{i=0}^{n-1}(1+im).  
\end{align}
In particular, if $m=1$, this recovers the $n!$ number, and so may be regarded as the \defn{m-generalization} of~$n!$ in type $A$. 
On the other hand, the number of elements of a finite Coxeter group $W$ of rank $n$ can be computed by the uniform formula
\begin{align}
|W| = \prod_{i=1}^n (1+e_i).
\end{align}
The $n!$ number is recovered for type $A_{n-1}$, in which case the exponents are $e_1,\dots,e_{n-1}=1,\dots,n-1$.
It seems plausible for us, that a natural \defn{$m$-generalization of $|W|$} is
\begin{align}
\prod_{i=1}^n (1+me_i).
\end{align}
In type $A_{n-1}$, this recovers the number of elements of the $s$-week order for $s=(m,m,\dots,m)$, expressed in Equation~\eqref{eq_mgeneralization_nfactorial}.
For type $B_2$, we get the sequence $(1+m)(1+3m)$, which counts exactly the number of vertices of the Fuss-Catalan $s$-permutahedron of type $B_2$ that we are suggesting on the left of Figure~\ref{fig_typeB_spermutahedra_sassociahedra}. The first terms of this sequence are $8,21,40,65,96,\dots$. 
These numbers are called octagonal numbers in~\cite[Sequence~A000567]{oeis}. 
It would be interesting to investigate if the $m$-generalization of $|W|$ is the dimension of certain $W$-module in representation theory. Remarkable $S_n$-modules whose dimensions are given by $n!$ appear in~\cite{garsia_graded_1993,garsia_natural_1996,haiman_hilbert_2001}. 

In forthcoming work, we will introduce and investigate the \defn{$s$-braid arrangement}, an arrangement of hyperplanes which induces a cell decomposition that is dual to the $s$-permutahedron. This point of view will provide further insights on potential generalizations in the context of finite Coxeter groups. An example is illustrated on the right of Figure~\ref{fig_braid_arrangement_s023}.

\section{Geometric realizations in dimensions 2 and 3}\label{sec_geometric_realizations_dim_2_3}

In this section, we provide an explicit construction that shows that Conjectures~\ref{conj:pure-polytopal},~\ref{coj:spermutahedra} and~\ref{con:sassociahedra} hold in dimensions 2 and 3.
There is a natural way of assigning coordinates to each $s$-decreasing tree. 
Let $e_{ij}:= e_i-e_j$ for $i<j$, where $e_1,\dots ,e_n\in \R^n$ are the standard basis vectors in $\R^n$. 
Let $s=(s(1),s(2),\dots,s(n))$ be a weak composition, $T$ be an $s$-decreasing tree and $A$ be a subset of tree-ascents of $T$. We define

\begin{equation*}
v_T = \sum_{i<j} \card_T(j,i) e_{ij} \quad \text{and} \quad 
F_{(T,A)}= \conv\{v_{T'}: T'\in [T,T+A]\}.
\end{equation*}

%In dimension 2, this gives a realization of the $s$-permutahedron. Some examples are shown in Figure~\ref{fig:s-perm_dim2}. 
For $n=3$ and $s(3)\neq 0$, this gives a 2-dimensional realization of the $s$-permutahedron in the  subspace $\{(x_1,x_2,x_3)\in\R^3: x_1+x_2+x_3=0\}\subset \R^3$, see Figure~\ref{fig:s-perm_dim2}. 
This realization ``cuts'' a polygon (a hexagon if $s_2\neq 0$ or a quadrilateral if $s(2)=0$) into smaller polygons. Each polygon is convex and corresponds to one facet of the $s$-permutahedron.

One would hope that this construction directly extends to higher dimensions, but it is not the case. 
For $n\geq 4$, the convex hull of all $v_T$'s is the zonotope $Z(s)$ from~\eqref{eq_zonotope}, which is still cut into identifiable pieces; however, those pieces do not form convex polytopes. 
We were able to fix this realization in dimension 3 ($n=4$ and $s(4)\neq 0$), using a procedure illustrated in Figure~\ref{fig:fix-d3}. The first image shows the direct realization obtained by $v_T$: 
we notice some bent edges and can identify what we call a ``broken pattern'' in trees related to those edges. 
The solution is to push the selected trees into a given direction by a parameter given by a broken pattern itself. 
In certain compositions $s$, this push leads to a ``collision'' (see middle of Figure~\ref{fig:fix-d3}) which again forces us to push certain trees further away. 
The process can be explicitly described for dimension 3. 
The new coordinates are given by 
$\overline v_T = \sum_{i<j} (3\card_T(j,i) + f_T(j,i)) e_{ij}$
where $f_T(j,i)=0$ for $j \neq 3$ and 
\begin{equation*}
f_T(3,i) = \begin{cases}
0 & \mbox{if }\card_T(3,i) = 0, \\
s_3 + \left( \card_T(4,1) - \card_T(4,3) \right) + \left(\card_T(4,2) - \card_T(4,3) \right) & \mbox{if } 0 < \card_T(3,i) < s_3, \\
2 s_3 & \mbox{if } \card_T(3,i) = s_3.
\end{cases}
\end{equation*}
See Figure~\ref{fig_geometricrealizations_intro} for examples, the demo webpage~\cite{SPermutaDemo}, or the {\tt SageMath} demo worksheet~\cite{SageDemoII}. We conjecture that such a construction also exists for higher dimensions.

\begin{figure}[htbp]
\begin{center}
\begin{tabular}{ccc}
\includegraphics[height=4.8cm]{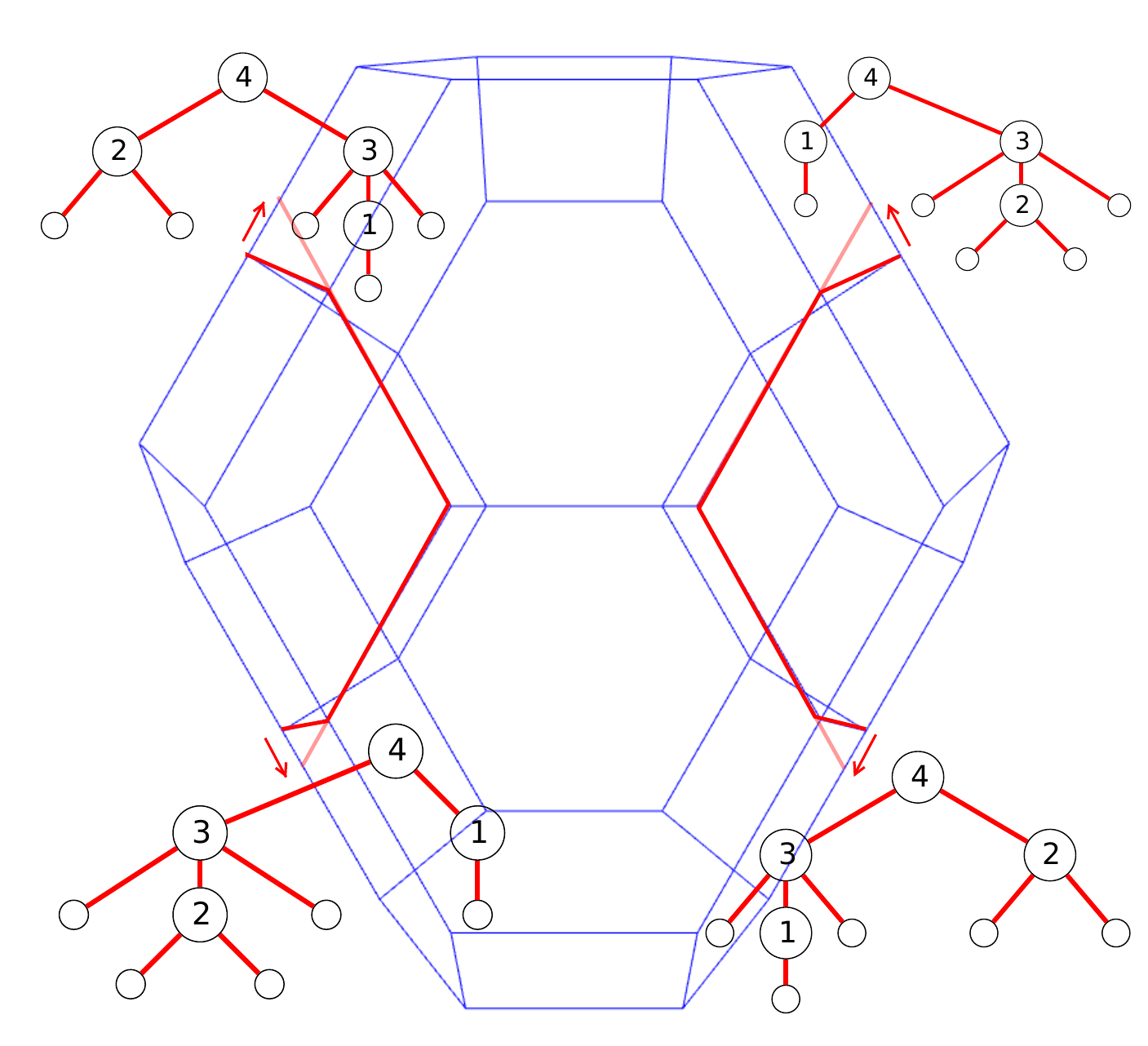} &
\includegraphics[height=4.8cm]{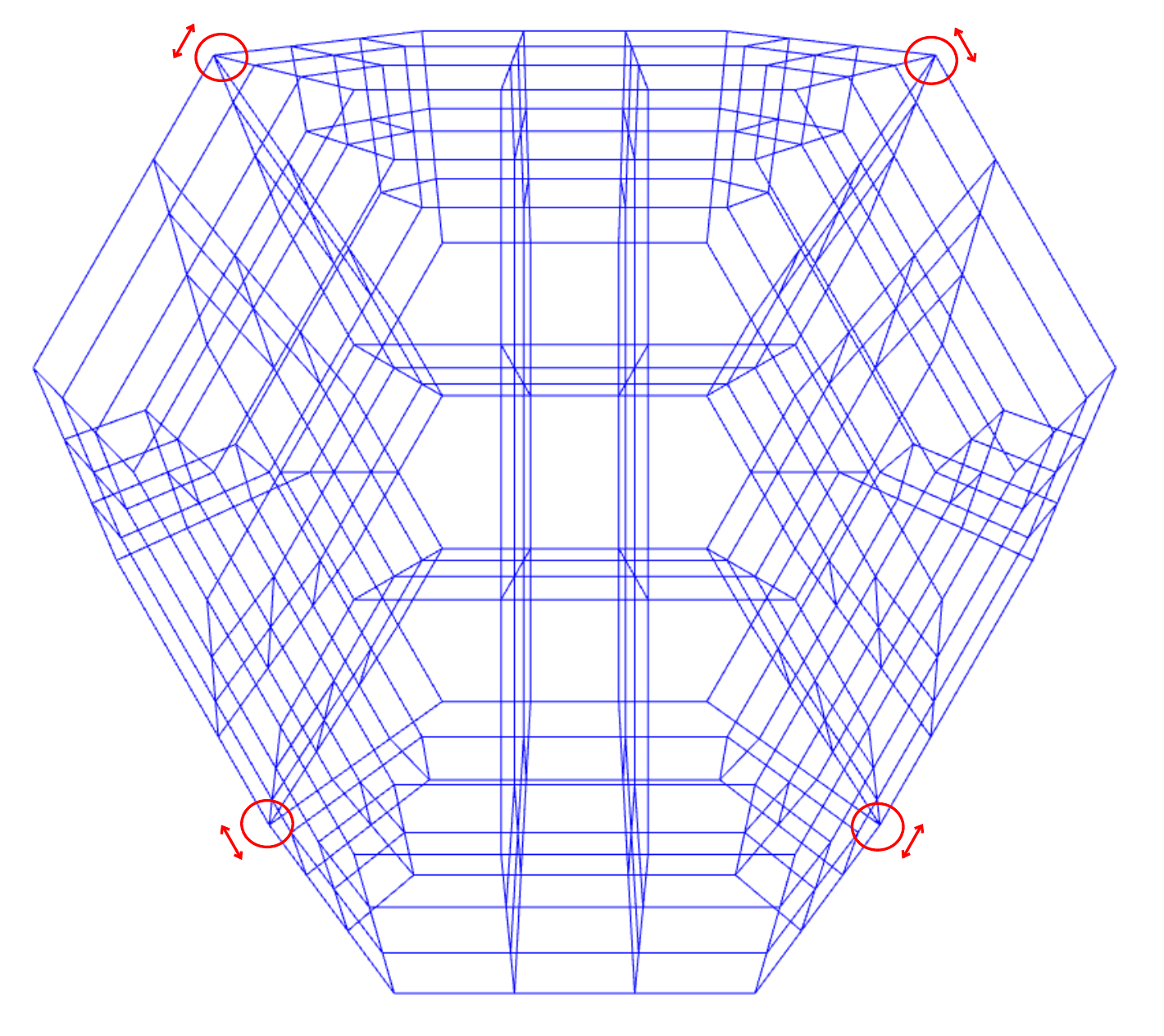} &
\includegraphics[height=4.8cm]{s0333.png}
\end{tabular}

\caption{Construction of a 3 dimensional realization.}
\label{fig:fix-d3}
\end{center}
\end{figure}

Once we have a geometric realization of the $s$-permutahedron, we are able to identify what we call \defn{Tamari-valid} faces and 
construct a realization of the $s$-associahedron by removing faces of the $s$-permutahedron. This works for both our $2D$ and $3D$ realizations,
and is an analogue of a construction in~\cite{HL07} giving rise to Loday's realization of the associahedron.
The process is illustrated in Figure~\ref{fig:2d-asso} in~$2D$, see Figure~\ref{fig_geometricrealizations_intro} for $3D$ examples.

\section{Realization via flow polytopes}
\label{sec:flows}
As mentioned before, the conjectures of this paper were announced in an extended abstract of our work back in 2019~\cite{FPSAC2019}, and this is the first time where we include all the details of our constructions.   
Meanwhile, Conjecture~\ref{coj:spermutahedra} (and consequently, Conjecture~\ref{conj:pure-polytopal}) have been proved in the particular case where $s$ contains no zeros~\cite[Theorem~1.3]{GMPT23}. The case when $s$ contains zeros is still open. Besides dimensions 2 and 3, Conjecture~\ref{con:sassociahedra} remains open for all $s$.

The construction in~\cite{GMPT23} is a beautiful work based on triangulations of flow polytopes.
For each composition~$s$ (with no zeros), the authors introduce a graph called the \emph{$s$-oruga graph}, and use a \emph{framing} of it to triangulate its corresponding flow polytope. They show that the dual of the poset of interior faces of the resulting triangulation is isomorphic to the face poset of the $s$-permutahedron. Finally, they use the Cayley trick and techniques from tropical geometry to obtain the desired geometric realization. %, as a polytopal subdivision of a polytope that is combinatorially equivalent to a permutahedron.  

The presentation in~\cite{GMPT23} has further important implications. 
One of them is a simpler and more convenient description of pure intervals of the $s$-weak order, when $s$ has no zeros. 
%Since they are in bijection with the interior faces of a framing triangulation of the flow polytope of the $s$-oruga graph, 
They can be described as certain subsets of \emph{coherent routes} of the $s$-oruga graph corresponding to \emph{interior} faces of the triangulation. This has the advantage that intersecting two pure intervals corresponds to just intersecting the associated sets of coherent routes. This intersection is literarily an intersection as sets, and therefore, it is very simple to compute. Exactly the same situation happens for the $\nu$-associahedron in Section~\ref{sec_nu_associahedron}, which makes the proofs much simpler in this case.

In contrast, the combinatorial description of the intersection of two pure intervals that we present in Section~\ref{sec_intersection_theorem} uses the characterization of pure intervals in Section~\ref{sec_characterization_pure_intervals}, which is rather involved.    
However, we believe that our deep combinatorial exploration of their combinatorial properties will be very useful. Combining this with the geometric/combinatorial techniques from~\cite{GMPT23} provides a powerful toolkit for future explorations.

%\section{Other geometric realizations}
%\label{sec:other-real}

\bibliographystyle{alphaurl}
\bibliography{../draft}
\end{document}